\documentclass[reqno,11pt]{amsart}

\usepackage[colorlinks, linkcolor=blue, citecolor=red, urlcolor=cyan,pagebackref]{hyperref}

\allowdisplaybreaks

\usepackage{cleveref}

\usepackage[rgb]{xcolor} 
\definecolor{mygreen}{rgb}{0,0.7,0.3}
\definecolor{myblue}{rgb}{0,0.50,1.20}
\definecolor{myorange}{rgb}{1,0.5,0.1}

\definecolor{fillred}{rgb}{1,0.9,0.9}
\definecolor{fillgreen}{rgb}{0.9,1,0.9}

 \definecolor{refkey}{rgb}{0,0.7,0.3}
 \definecolor{labelkey}{rgb}{1,0,0}

\usepackage{url}

\usepackage{amsthm,amssymb,amscd}
\usepackage{mathrsfs}
\usepackage{bbm}
\usepackage{graphicx}
\usepackage{epic,eepic}
\usepackage{tikz-cd}
\usepackage{stmaryrd}

\usepackage{accents}

\usepackage{here}
\usepackage{bm}
\usepackage[all]{xy}

\usepackage{tikz}
\usetikzlibrary{calc,decorations.markings}
\usetikzlibrary{arrows.meta}
\usetikzlibrary{positioning}
\usetikzlibrary{cd, intersections, calc, decorations.pathmorphing, arrows, decorations.pathreplacing}

\usepackage{version}

\excludeversion{NB}
\excludeversion{NB2}
\numberwithin{equation}{section}

\newtheorem{thm}{Theorem}[section]
\newtheorem{cor}[thm]{Corollary}
\newtheorem{lem}[thm]{Lemma}
\newtheorem{prop}[thm]{Proposition}

\newtheorem{introthm}{Theorem}

\theoremstyle{definition}
\newtheorem{dfn}[thm]{Definition}
\newtheorem{ex}[thm]{Example}

\newtheorem{rem}[thm]{Remark}
\newtheorem{conv}[thm]{Notation}

\newtheorem{lemdef}[thm]{Lemma-Definition}

\crefname{thm}{Theorem}{Theorems}
\crefname{cor}{Corollary}{Corollaries}
\crefname{lem}{Lemma}{Lemmas}
\crefname{prop}{Proposition}{Propositions}
\crefname{dfn}{Definition}{Definitions}
\crefname{ex}{Example}{Examples}
\crefname{claim}{Claim}{Claims}
\crefname{conj}{Conjecture}{Conjectures}
\crefname{rem}{Remark}{Remarks}
\crefname{figure}{Figure}{Figures}
\crefname{section}{Section}{Sections}
\crefname{subsection}{Section}{Sections}
\crefname{appendix}{Appendix}{Appendices}
\crefname{assum}{Assumption}{Assumptions}
\crefname{conv}{Notation}{Notations}
\crefname{lemdef}{Lemma-Definition}{Lemma-Definitions}

\crefname{introthm}{Theorem}{Theorems}
\crefname{introcor}{Corollary}{Corollaries}
\crefname{introconj}{Conjecture}{Conjectures}

\def\C{{\mathbb C}}
\def\Z{{\mathbb Z}}
\def\ve{{\varepsilon}}
\def\P{{\mathcal{P} }}

\newcommand{\std}{\mathrm{std}}
\newcommand{\Hom}{\mathop{\mathrm{Hom}}\nolimits}
\newcommand{\Ker}{\mathop{\mathrm{Ker}}\nolimits}

\newcommand{\inn}{\mathrm{in}}
\newcommand{\out}{\mathrm{out}}

\newcommand\vw{{\overline{w_0}}}

\newcommand{\op}{\mathrm{op}}

\newcommand\A{{\mathcal{A} }}
\newcommand\B{{\mathcal{B} }}

\newcommand{\cO}{\mathcal{O}}
\newcommand{\cC}{\mathcal{C}}
\newcommand\bs{{\boldsymbol{s}}}

\newcommand{\bB}{\mathbb{B}}
\newcommand{\bJ}{\mathbf{J}}
\newcommand{\bM}{\mathbb{M}}
\newcommand{\bR}{\mathbb{R}}

\newcommand{\bG}{\mathbb{G}}

\newcommand{\sfT}{\mathsf{T}}

\newcommand{\clA}{A}
\newcommand{\flA}{\mathsf{A}}
\newcommand{\flB}{\mathsf{B}}
\newcommand{\bw}{{\boldsymbol{w}}}
\newcommand{\bv}{{\boldsymbol{v}}}

\def\L{{\mathcal{L}}}%
\newcommand\Conf{{\mathrm{Conf}}}%
\newcommand\Ad{{\mathrm{Ad}}}
\newcommand\fiber{{\boldsymbol{o}}}

\DeclareMathOperator{\rank}{\mathrm{rank}}

\newcommand\qarrow[2]{\draw[->,shorten >=2pt,shorten <=2pt,>=latex] (#1) -- (#2) [thick];} 
\newcommand\qsarrow[2]{\draw[->,shorten >=3pt,shorten <=3pt] (#1) -- (#2) [thick];} 
\newcommand\qsharrow[2]{\draw[->,shorten >=4pt,shorten <=2pt] (#1) -- (#2) [thick];} 
\newcommand\qstarrow[2]{\draw[->,shorten >=2pt,shorten <=4pt] (#1) -- (#2) [thick];} 

\newcommand\qshdarrow[2]{\draw[->,dashed,shorten >=4pt,shorten <=2pt] (#1) -- (#2) [thick];} 
\newcommand\qstdarrow[2]{\draw[->,dashed,shorten >=2pt,shorten <=4pt] (#1) -- (#2) [thick];} 

\newcommand\qdrarrow[2]{\draw[->,dashed,shorten >=2pt,shorten <=2pt,bend right=0.5cm] (#1) to (#2) [thick];} 
\newcommand\qdlarrow[2]{\draw[->,dashed,shorten >=2pt,shorten <=2pt,bend left=0.5cm] (#1) to (#2) [thick];} 

\tikzset{
  mid arrow/.style={postaction={decorate,decoration={
        markings,
        mark=at position .5 with {\arrow[#1]{stealth}}
      }}},
}

\tikzset{pics/.cd,
handle/.style={code={
\draw (-0.72,0) to[bend left] (0.72,0);
\draw (-0.9,0.1) to[bend right] (0.9,0.1);
}}}

\newcommand{\tangent}[3]{
				\draw[->] (#1) --++(#2:#3);
				}

\tikzset{->-/.style 2 args={
	postaction={decorate},
	decoration={markings, mark=at position #1 with {\arrow[thick, #2]{>}}} 
    },
    ->-/.default={0.5}{}
}
\tikzset{-<-/.style 2 args={
	postaction={decorate},
	decoration={markings, mark=at position #1 with {\arrow[thick, #2]{<}}} 
    },
    -<-/.default={0.5}{}
}

\newcommand{\quiverplus}[3]{
\begin{scope}[>=latex]
{\color{mygreen}
    \path(#1) coordinate(x1);
    \path(#2) coordinate(x2);
    \path(#3) coordinate(x3);
    \foreach \l in {1,2}
    {
        \draw($(x1)!0.333*\l!(x2)$) circle(2pt) coordinate(x12\l);
        \draw($(x2)!0.333*\l!(x3)$) circle(2pt) coordinate(x23\l);
        \draw($(x3)!0.333*\l!(x1)$) circle(2pt) coordinate(x31\l);
    }
    \path($(x1)!0.5!(x2)$) coordinate(H);
    \draw($(x3)!0.667!(H)$) circle(2pt) coordinate(G);
    \qarrow{x121}{G}
    \qarrow{x231}{G}
    \qarrow{x311}{G}
    \qarrow{G}{x122}
    \qarrow{G}{x232}
    \qarrow{G}{x312}
    \qarrow{x312}{x121}
    \qarrow{x122}{x231}
    \qarrow{x232}{x311}
}
\end{scope}
}

\newcommand{\quiversquareC}[4]{
\begin{scope}[>=latex]
    {\color{mygreen}
    \path(#1) coordinate(x1);
    \path(#2) coordinate(x2);
    \path(#3) coordinate(x3);
				\path(#4) coordinate(x4);
				\foreach \i in {1,2}
				{
    \path($(x1)!\i/3!(x4)$) coordinate(x14\i);
				\path($(x2)!\i/3!(x3)$) coordinate(x23\i);
				}
				\foreach \j in {0,1,2,3,4}
				{
				\draw($(x141)!\j/4!(x231)$) circle(2pt) coordinate(v1\j);
				\draw($(x142)!\j/4!(x232)$) coordinate(v2\j);
				\dnode{v2\j}{mygreen};
				}
				\draw[myblue]($(x1)!1/4!(x2)$) circle(2pt) coordinate(yl);
				\draw[myblue]($(x1)!3/4!(x2)$) circle(2pt) coordinate(yr);
				\draw($(x4)!1/4!(x3)$) coordinate(zl);
				\dnode{zl}{myblue};
				\draw($(x4)!3/4!(x3)$) coordinate(zr);
				\dnode{zr}{myblue};
				}
\end{scope}
}

\newcommand\dnode[2]{\draw[#2] (#1)circle(3pt) node[scale=0.6]{$2$}}

\newcommand{\smallsq}{\draw(0,0) -- (0,1) -- (1,1) -- (1,0) --cycle;\foreach \i in {0,1} \foreach \j in {0,1} \fill(\i,\j) circle(1.5pt);}

\newcommand{\CoG}[3]{
    \path(#1) coordinate(x1);
    \path(#2) coordinate(x2);
    \path(#3) coordinate(x3);
    \path($(x1)!0.5!(x2)$) coordinate(H);
    \path($(x3)!0.667!(H)$) circle(2pt) coordinate(G);}
\newcommand{\triv}[3]{
    \CoG{#1}{#2}{#3}
    \draw[red,thick,->-={0.7}{}] (#1) -- (G);
    \draw[red,thick,->-={0.7}{}] (#2) -- (G);
    \draw[red,thick,->-={0.7}{}] (#3) -- (G);
}   
\newcommand{\trivop}[3]{
    \CoG{#1}{#2}{#3}
    \draw[red,thick,-<-={0.7}{}] (#1) -- (G);
    \draw[red,thick,-<-={0.7}{}] (#2) -- (G);
    \draw[red,thick,-<-={0.7}{}] (#3) -- (G);
}   

\tikzset{
    wline/.style={
        shorten >=0.3pt,shorten <=0.3pt,line width=1.2pt, red!15, preaction={draw, shorten >=0.3pt,shorten <=0.3pt,line width=2.5pt, red}
    }
				}
				
\tikzset{
    wlinebdy/.style={
        shorten >=0.3pt,shorten <=0.3pt,line width=1.2pt, myblue!15, preaction={draw, shorten >=0.3pt,shorten <=0.3pt,line width=2.5pt, myblue}
    }
}

\tikzset{
	webline/.style={
		red, very thick
	}
}

\newcommand{\trivL}[3]{
    \CoG{#1}{#2}{#3}
    \draw[webline] (#1) -- (G);
    \draw[webline] (#2) -- (G);
    \draw[wline] (#3) -- (G);
}   
\newcommand{\trivR}[3]{
    \CoG{#1}{#2}{#3}
    \draw[webline] (#1) -- (G);
    \draw[wline] (#2) -- (G);
    \draw[webline] (#3) -- (G);
}   

\makeatletter
\newcommand{\oset}[3][0ex]{%
  \mathrel{\mathop{#3}\limits^{
    \vbox to#1{\kern-2\ex@
    \hbox{$\scriptstyle#2$}\vss}}}}
\makeatother

\pgfdeclarelayer{bg}    
\pgfsetlayers{bg,main}  

\begin{document}
\title[$\mathscr{A}=\mathscr{U}$ for cluster algebras from moduli spaces of $G$-local systems]
{$\mathscr{A}=\mathscr{U}$ for cluster algebras from moduli spaces of $G$-local systems}

\author[Tsukasa Ishibashi]{Tsukasa Ishibashi}
\address{Tsukasa Ishibashi, Mathematical Institute, Tohoku University, 6-3 Aoba, Aramaki, Aoba-ku, Sendai, Miyagi 980-8578, Japan.}
\email{tsukasa.ishibashi.a6@tohoku.ac.jp}

\author[Hironori Oya]{Hironori Oya}
\address{Hironori Oya, Department of Mathematics, Tokyo Institute of Technology, 2-12-1 Ookayama, Meguro-ku, Tokyo 152-8551, Japan.}
\email{hoya@math.titech.ac.jp}

\author[Linhui Shen]{Linhui Shen}
\address{Linhui Shen, Department of Mathematics, Michigan State University, 619 Red Cedar Road, 302 Wells Hall, East Lansing, Michigan 48824, United States}
\email{linhui@math.msu.edu}

\date{\today}

\begin{abstract}
For a finite-dimensional simple Lie algebra $\mathfrak{g}$ admitting a non-trivial minuscule representation and a connected marked surface $\Sigma$ with at least two marked points and no punctures, we prove that the cluster algebra $\mathscr{A}_{\mathfrak{g},\Sigma}$ associated with the pair $(\mathfrak{g},\Sigma)$ coincides with the upper cluster algebra $\mathscr{U}_{\mathfrak{g},\Sigma}$. The proof is based on the fact that the function ring $\cO(\A^\times_{G,\Sigma})$ of the moduli space of decorated twisted $G$-local systems on $\Sigma$ is generated by matrix coefficients of Wilson lines introduced in \cite{IO20}. As an application, we prove that the Muller-type skein algebras $\mathscr{S}_{\mathfrak{g}, \Sigma}[\partial^{-1}]$ \cite{Muller,IY21,IY_C2} for $\mathfrak{g}=\mathfrak{sl}_2, \mathfrak{sl}_3,$ or $\mathfrak{sp}_4$ are isomorphic to the cluster algebras $\mathscr{A}_{\mathfrak{g}, \Sigma}$.
\end{abstract}

\maketitle

\setcounter{tocdepth}{1}
\tableofcontents

\section{Introduction}\label{sec:introduction}
Cluster algebras $\mathscr{A}$, introduced by Fomin and Zelevinsky \cite{FZ02}, are a class of commutative algebras with distinguished generators called cluster variables.  The cluster variables are grouped into possibly infinitely many collections called \emph{clusters}, which are related by a particular type of transition maps called \emph{cluster transformations}. In \cite{BFZ}, Berenstein, Fomin, and Zelevinsky introduced the upper cluster algebras $\mathscr{U}$, defined as the intersections of Laurent polynomial rings associated with clusters. The upper cluster algebras $\mathscr{U}$ are more natural than $\mathscr{A}$ from the perspective of geometry. 
The Laurent phenomenon of cluster algebras implies that $\mathscr{A}\subset \mathscr{U}$, but in general $\mathscr{A}\neq \mathscr{U}$. 

The problem when $\mathscr{A}=\mathscr{U}$ is a mysterious but rather important question in cluster theory. For example, let $Q$ be a quiver with a non-degenerate potential $W$. Motivated by the representation theory of quivers with potential \cite{DWZ, CC}, the paper \cite{CLFS} introduces the Caldero-Chapoton algebra $\mathscr{CC}_{Q, W}$. When $Q$ has no loops or 2-cycles, we have 
\[\mathscr{A}_{Q}\subset \mathscr{CC}_{Q,W} \subset \mathscr{U}_Q,\] 
where $\mathscr{A}_Q$ and $\mathscr{U}_Q$ are the cluster algebra and the upper cluster algebra associated with $Q$, respectively. If $\mathscr{A}_Q=\mathscr{U}_Q$, then all the three aforementioned algebras are equal. As an application, after verifying a combinatorial condition on the existence of reddening mutation sequences of $Q$, the Caldero-Chapoton functions in $\mathscr{CC}_{Q,W}$ provide a natural linear basis on the cluster algebra $\mathscr{A}_Q$, called the generic basis \cite{Qin}. We refer to \cite{GLFS} for more details on the  significance of the $\mathscr{A}=\mathscr{U}$ problem in  the study of generic basis of cluster algebras.  We refer to \cite{GLS} for another application of $\mathscr{A}=\mathscr{U}$  on the quantization of cluster algebras.

\smallskip

In this paper, we investigate the  $\mathscr{A}=\mathscr{U}$ problem for cluster algebras from moduli spaces of $G$-local systems. Let $G$ be the simply-connected complex Lie group associated with a simple Lie algebra $\mathfrak{g}$. Let $\Sigma$ be a surface with punctures and marked points on its boundary. Fock and Goncharov \cite{FG03} introduced a moduli space
$\mathcal{A}_{G, \Sigma}$ of decorated $G$-local systems over $\Sigma$ as an algebro-geometric avatar of higher Teichm\"uller spaces. The moduli space $\mathcal{A}_{G, \Sigma}$ carries a natural cluster structure, constructed by Fock and Goncharov \cite{FG03} for $SL_n$, by Le \cite{Le} for other type of classical groups, and by Goncharov and the third named author \cite{GS19} for general groups. As a consequence, the cluster structure on $\mathcal{A}_{G, \Sigma}$ gives rise to a cluster algebra $\mathscr{A}_{\mathfrak{g}, \Sigma}$ and an upper cluster algebra $\mathscr{U}_{\mathfrak{g}, \Sigma}$ over the ground field $\mathbb{C}$.\footnote{In \cite{BMS}, Bucher, Machacek, and Shapiro show that the equality $\mathscr{A}=\mathscr{U}$ depends on the choice of ground ring. In this paper, we always choose the ground field $\mathbb{C}$. By an easy exercise of linear algebra, all the results of the paper can be generalized to $\mathbb{Q}$.}

Our first main result is as follows. 
\begin{introthm}\label{main theorem}    
For a finite-dimensional simple Lie algebra $\mathfrak{g}$ admitting a non-trivial minuscule representation 
(namely, not of type $E_8,F_4,G_2$)
and a connected marked surface $\Sigma$ with at least two marked points and no punctures, we have 
\begin{align*}
    \mathscr{A}_{\mathfrak{g},\Sigma}=\mathscr{U}_{\mathfrak{g},\Sigma}.
\end{align*}
\end{introthm}
We prove Theorem \ref{main theorem} in Section \ref{section4}.
Our proof is based on the following geometric considerations on the moduli space $\A_{G,\Sigma}$:
\begin{enumerate}
    \item We first show that the function ring $\cO(\A_{G,\Sigma}^\times)$ is generated by matrix coefficients of \emph{Wilson lines}, which are originally introduced in \cite{IO20} on the closely related moduli space $\P_{G',\Sigma}$ \cite{GS19} for adjoint groups $G'$. In this paper, we introduce a ``lifted'' version of Wilson lines defined on $\A_{G,\Sigma}^\times$, whose values are in the simply-connected group $G$.
    \item Then we show that the upper cluster algebra $\mathscr{U}_{\mathfrak{g},\Sigma}$ coincides with the function ring $\cO(\A_{G,\Sigma}^\times)$ over $\mathbb{C}$ by a covering argument up to codimension 2, similarly to the proof of \cite[Theorem 1.1]{S21}. Therefore we obtain a geometric generating set of the upper cluster algebra provided by the Wilson lines.
    \item Finally, we show that the generalized minors of certain \emph{simple} Wilson lines are single cluster variables, multiplied by several frozen variables (\cref{prop:simple_minors}). Under the assumption of \cref{main theorem}, these are enough to generate $\cO(\A_{G,\Sigma}^\times)=\mathscr{U}_{\mathfrak{g},\Sigma}$ and thus we get the desired inclusion  $\mathscr{U}_{\mathfrak{g},\Sigma} \subset \mathscr{A}_{\mathfrak{g},\Sigma}$.
\end{enumerate}
We remark here that \cref{prop:simple_minors} implies that the generalized minors of simple Wilson lines are contained in the \emph{theta basis} \cite{GHKK}, and hence they are universally Laurent polynomials with positive integral coefficients. This strengthens the positivity result in \cite{IO20} for this particular class of Wilson lines and matrix coefficients.

We include a list of results preceding us:
\begin{itemize}
    \item 
Muller \cite{Mul13} proved $\mathscr{A}=\mathscr{U}$ for locally acyclic cluster algebras. When $\mathfrak{g}=\mathfrak{sl}_2$ and $\Sigma$ is unpunctured and contains at least two marked points, the cluster algebra $\mathscr{A}_{\mathfrak{sl}_2, \Sigma}$ is locally acyclic, and hence $\mathscr{A}_{\mathfrak{sl}_2, \Sigma}=\mathscr{U}_{\mathfrak{sl}_2, \Sigma}$.
\item
Canakci, Lee, and Schiffler \cite{CLR} prove $\mathscr{A}_{\mathfrak{sl}_2, \Sigma}=\mathscr{U}_{\mathfrak{sl}_2, \Sigma}$, where $\Sigma$ is an unpunctured surface with one marked point. We expect that the same result can be generalized to arbitrary $\mathfrak{g}$.

\item Goodearl and Yakimov \cite{GY} prove a quantum analog of $\mathscr{A}=\mathscr{U}$ for cluster algebras associated with double Bruhat cells.
We leave it for a future project to achieve a quantum analog of Theorem \ref{main theorem}.

\item Shen and Weng \cite{SW21} prove $\mathscr{A}=\mathscr{U}$ for cluster algebras associated with double Bott-Samelson cells. Examples of double Bott-Samelson cells include all the double Bruhat cells and the augmentation varieties associated with positive-braid Legendrian links. Our Theorem  \ref{main theorem} can be viewed as a generalization of the result of \cite{SW21} from disks to surfaces.

\item In the other direction, Berenstein, Fomin and Zelevinsky \cite{BFZ} prove $\mathscr{A}_{\mathfrak{sl}_2, \Sigma}\neq \mathscr{U}_{\mathfrak{sl}_2, \Sigma}$ when $\Sigma$ is a closed torus with exactly one puncture. Ladkani \cite{Lad13} extended this result for closed surfaces of genus $g \geq 1$ with exactly one puncture. 
A very recent work of Moon and Wong \cite{MW22} proves $\mathscr{A}_{\mathfrak{sl}_2, \Sigma}\neq \mathscr{U}_{\mathfrak{sl}_2, \Sigma}$ when $\Sigma$ is a closed torus with $n \geq 1$ many punctures. They also conjecture $\mathscr{A}_{\mathfrak{sl}_2, \Sigma}\neq \mathscr{U}_{\mathfrak{sl}_2, \Sigma}$ when $\Sigma$ is a closed surface of genus $g \geq 1$ with $n \geq 1$ many punctures. Therefore our assumption on the absence of punctures is crucial.
\end{itemize}

\smallskip

Closely related to the moduli space $\mathcal{A}_{G, \Sigma}$ are the \emph{(stated) skein algebras}, which are generated by combinatorial objects called $\mathfrak{g}$-webs modulo a collection of explicit graphical relations. For $\mathfrak{g}=\mathfrak{sl}_2$, the connections between such skein algebras and the cluster algebra $\mathscr{A}_{\mathfrak{sl}_2, \Sigma}$ have been broadly studied in the literature (cf. \cite{BW,Muller,CL19}). More recently, the study of the relations of skein and cluster algebras and their web bases has been extended to $\mathfrak{sl}_3$ (cf.~\cite{DS20a, DS20b, K20, IY21}). The $\mathfrak{sp}_4$-case is also studied in \cite{IY_C2}.

Following the notation of \cite{IY21,IY_C2}, we consider the boundary-localized skein algebra $\mathscr{S}^q_{\mathfrak{g}, \Sigma}[\partial^{-1}]$ for $\mathfrak{g}=\mathfrak{sl}_2,\mathfrak{sl}_3,\mathfrak{sp}_4$. 
Let $\mathscr{S}_{\mathfrak{g}, \Sigma}[\partial^{-1}]$ be its classical specialization $q=1 \in \C$, which is a $\C$-algebra. As an application of Theorem \ref{main theorem}, we prove

\begin{introthm}
\label{main thm 2}
If $\mathfrak{g}=\mathfrak{sl}_2$, $\mathfrak{sl}_3$ or $\mathfrak{sp}_4$ and $\Sigma$ is an unpunctured surface with at least two marked points, then the skein algebra $\mathscr{S}_{\mathfrak{g}, \Sigma}[\partial^{-1}] $ is isomorphic to the cluster algebra $\mathscr{A}_{\mathfrak{g}, \Sigma}$.
\end{introthm}
We prove Theorem \ref{main thm 2} in Section \ref{section 5}.
The $\mathfrak{sl}_2$-case is exactly the result obtained by \cite{Muller}. 
For $\mathfrak{g}=\mathfrak{sl}_3$ or $\mathfrak{sp}_4$, the inclusion $\mathscr{S}_{\mathfrak{g}, \Sigma}[\partial^{-1}] \subset \mathscr{A}_{\mathfrak{g}, \Sigma}$ is proved in \cite{IY21,IY_C2}. We further verify the inclusion 
$\cO(\A_{G,\Sigma}^\times) \subset \mathscr{S}_{\mathfrak{g}, \Sigma}[\partial^{-1}]$ 
by verifying that the matrix entries of simple Wilson lines can be written as explicit $\mathfrak{g}$-webs, similarly to the step (3) in the proof of \cref{main theorem}. Then the equality $\mathscr{A}_{\mathfrak{g}, \Sigma}=\cO(\A_{G,\Sigma}^\times)$ over $\mathbb{C}$ allows us to combine these inclusions to get \cref{main thm 2}. 
Theorem \ref{main thm 2} confirms a conjecture of the first named author and Yuasa in the classical setting (\cite[Conjecture 3]{IY21}). 

\subsection*{\textbf{Acknowledgments}}
T. I. is grateful to Wataru Yuasa for a valuable discussion on relations to skein algebras, and giving careful comments on a draft of this paper. The authors are grateful to the anonymous referees for the careful reading and the several comments that significantly improve our paper. 
T. I. is partially supported by JSPS KAKENHI Grant-in-Aid for Research Activity Start-up (No.~20K22304).
H. O. is supported by JSPS KAKENHI Grant-in-Aid for Early-Career Scientists (No.~19K14515).
L. S. is supported by the Collaboration Grants for Mathematicians from the Simons Foundation ($\#$~711926) and the NSF grant DMS-2200738.

\section{The function ring \texorpdfstring{$\cO(G)$}{O(G)}}\label{sec:O(G)}

\subsection{Notations from Lie theory}

In this section, we briefly recall basic terminologies in Lie theory. We refer the reader to, for example,  \cite{Jan} for missing definitions. 

Let $G$ be a simply-connected connected simple algebraic group over $\mathbb{C}$. Let $B^+$ and $H$ be a Borel subgroup and a maximal torus contained in $B^+$, respectively. Let $U^+$ be the unipotent radical of $B^+$. Let 
\begin{itemize}
    \item $X^*(H)=\Hom(H,\mathbb{G}_m)$ be the weight lattice, $X_*(H)=\Hom(\mathbb{G}_m, H)$ the coweight lattice, and $\langle -,-\rangle$ the natural pairing 
\[
\langle -,-\rangle\colon~ X_*(H)\times X^*(H)\longrightarrow \Hom(\mathbb{G}_m, \mathbb{G}_m)\simeq \Z; 
\]
    \item $\Phi\subset X^*(H)$ the root system of $(G, H)$; 
    \item $\Phi_+\subset \Phi$ the set of positive roots consisting of the $H$-weights of the Lie algebra of $U^+$;
    \item $\{\alpha_s\mid s\in S\}\subset \Phi_+$ the set of simple roots, where $S$ is the index set with $|S|=r$;
    \item $\{\alpha_s^\vee \mid s \in S\}$ the set of simple coroots.
\end{itemize}
For $s\in S$, let $\varpi_s\in X^*(H)$ be the $s$-th fundamental weight  such that $\langle \alpha_t^{\vee}, \varpi_s\rangle=\delta_{st}$. In other words, we have $\alpha_t=\sum_{u \in S} C_{ut} \varpi_u$ for $t \in S$, where $C_{st}:=\langle \alpha_s^\vee, \alpha_t\rangle\in \mathbb{Z}$. 
We have \[X^*(H)=\sum_{s\in S}\mathbb{Z}\varpi_s.\] 
The sub-lattice 
generated by $\alpha_s$ for $s\in S$ is  called the \emph{root lattice}.

For $h\in H$ and $\mu\in X^*(H)$, the evaluation of $\mu$ at $h$ is denoted by $h^{\mu}$. For $s\in S$, we have a pair of root homomorphisms $x_{s}, y_{s}\colon \mathbb{G}_a\to G$ such that 
\[
hx_{s}(t)h^{-1}=x_{s}(h^{\alpha_s}t), \qquad hy_s(t)h^{-1}=y_s(h^{-\alpha_s}t).
\]
After a suitable normalization, we obtain a homomorphism 
$\varphi_{s}\colon SL_2\to G$
such that 
\[
\varphi_{s}\left( \begin{pmatrix}
    1&a\\
    0&1
\end{pmatrix}\right)=x_{s}(a), \qquad
\varphi_{s} \left(\begin{pmatrix}
1&0\\
a&1
\end{pmatrix}\right)= y_s(a), \qquad
\varphi_{s} \left(\begin{pmatrix}
a&0\\
0&a^{-1}
\end{pmatrix}\right)= \alpha_s^\vee(a).
\]

\smallskip

\paragraph{\textbf{Weyl groups}}
Let $W(G):=N_{G}(H)/H$ denote the Weyl group of $G$, where $N_G(H)$ is the normalizer subgroup of $H$ in $G$.  
For $s \in S$, we set \[\overline{r}_{s}:=\varphi_s\left(
\begin{pmatrix}
    0&-1\\
    1&0
\end{pmatrix}
\right) \in N_{G}(H).\] 
The elements $r_s:=\overline{r}_s H \in W(G)$ have order $2$, and give rise to a Coxeter generating set for $W(G)$ with the following presentation:
\begin{align*}
    W(G)= \langle  r_s\ (s \in S) \mid r_s^2=1, ~~ (r_sr_t)^{m_{st}}=1\ (s,t \in S) \rangle,
\end{align*}
where $m_{st} \in \Z$ is given by the following table
\[
\begin{tabular}{rccccc}
$C_{st}C_{ts}:$ & $0$ & $1$ & $2$ & $3$  \\
$m_{st}:$         & $2$ & $3$ & $4$ & $6$ 
\end{tabular}.
\]
There is a left action of $W(G)$ on $X^*(H)$ induced from the (right) conjugation action of $N_{G}(H)$ on $H$. 
The action of each reflection $r_s$ is given by
\begin{align*}
    r_s.\mu=\mu - \langle \alpha_s^\vee, \mu \rangle \alpha_s, \qquad \forall s\in S, \quad \forall \mu \in X^*(H).
\end{align*}

For $w\in W(G)$, a sequence $\bs=(s_1,\dots,s_\ell)$ of elements of $S$ is called a reduced word of $w$ if $w=r_{s_1}\dots r_{s_\ell}$ and $l$ is the smallest among all the sequences with this property. For a reduced word $\bs=(s_1,\dots,s_\ell)$ of $w\in W(G)$, the number $l(w):=\ell$ is called the length of $w$, and set  $\overline{w}:=\overline{r}_{s_1}\dots\overline{r}_{s_\ell} \in N_G(H)$. Then it turns out that $\overline{w}$ does not depend on the choice of the reduced word. 

Let $w_0\in W(G)$ be the longest element of $W(G)$, and set $s_G:=\overline{w_0}^2 \in N_G(H)$. It turns out that $s_G\in Z(G)$, and $s_G^2=1$ (cf.~\cite[\S 2]{FG03}). We define an involution $S\to S, s\mapsto s^{\ast}$ by 
\begin{align*}
\alpha_{s^{\ast}}=-w_0.\alpha_s. 
\end{align*}

\smallskip

\paragraph{\textbf{The Dynkin involution.}} There exists an anti-involution $\sfT:G\to G, g\mapsto g^\mathsf{T}$ of the algebraic group $G$ given by $\mathsf{T}\circ x_s=y_s$ and $h^\mathsf{T}=h$ for $s\in S$, $h\in H$. This is called the \emph{transpose} in $G$. 
	Let $\ast: G\to G, g\mapsto g^{\ast}$ be a group automorphism defined by 
	\begin{align*}
	g\mapsto \overline{w}_0(g^{-1})^{\mathsf{T}}\overline{w}_0^{-1}.
	\end{align*}
Then $(g^{\ast})^{\ast}=g$ for all $g\in G$. This is called the \emph{Dynkin involution} on $G$ (cf. \cite[(2)]{GS18}).

\smallskip

\paragraph{\textbf{Irreducible modules and matrix coefficients}}
Let $V$ be a finite dimensional representation of $G$ (over $\mathbb{C}$).
For $f\in V^{\ast}$ and $v\in V$, we define the element $c_{f, v}^V\in \cO(G)$ by  
\begin{align}
g\mapsto \langle f, g.v\rangle, \qquad \forall g\in G\label{eq:mat_coeff}
\end{align}
An element of this form is called a \emph{matrix coefficient}. 

Set $X^*(H)_+:=\sum_{s\in S}\mathbb{Z}_{\geq 0}\varpi_s\subset X^*(H)$. 
For $\lambda\in X^*(H)_+$, let $V(\lambda)$ be the irreducible representation of $G$ of highest weight $\lambda$. A fixed highest weight vector of $V(\lambda)$ is denoted by $v_{\lambda}$. For $\lambda^* = -w_0.\lambda$, there is a natural non-degenerate pairing
\[
\langle-,-\rangle:~ V(\lambda^*)\times V(\lambda) \longrightarrow \mathbb{C},
\] which identify $V(\lambda^*)$  with the dual vector space of $V(\lambda)$. We further fix a lowest weight vector of $V(\lambda^*)$, denoted by $f_{\lambda^*}$, such that $\langle f_{\lambda^*}, v_\lambda\rangle =1$. 

For $\lambda\in X^*(H)_+$ 
and  $w,w'\in W(G)$, the matrix coefficient
\begin{align}
\Delta_{w\lambda,w'\lambda}(g):=\langle \overline{w}.f_{\lambda^*},g\overline{w}'.v_{\lambda}\rangle   \label{eq:minor}
\end{align}
is called a \emph{generalized minor}.

\subsection{The generators of $\cO(G)$}
\begin{prop}\label{prop:generate}
Let $G$ be a semisimple algebraic group over $\C$, and $\rho\colon G\to GL(V)$ a faithful rational representation. Then the ring of regular functions $\cO(G)$ is generated by the matrix coefficients of $V$. 
\end{prop}
\begin{proof}
Consider the composition map $\det \circ \rho \colon G\to \C^{\ast}$. Then it gives a one-dimensional rational representation of $G$, and it must be trivial since $G$ is semisimple. Therefore, $\rho(G)\subset \Ker \det=SL(V)$. Hence, by \cite[Proposition 2.2.5]{Spr}, $\rho(G)$ is a Zariski closed subgroup of $SL(V)$. 
Here note that $G$ is isomorphic to $\rho(G)$ as an algebraic group since $\rho$ is faithful. 
 By fixing a basis of $V$, $SL(V)$ is regarded as a Zariski closed set of $\C^{(\dim V)^2}$. Hence, $\rho(G)$ is also considered as a Zariski closed set of $\C^{(\dim V)^2}$, and there exists a surjective algebra homomorphism $R\colon \cO(\C^{(\dim V)^2})=\C[X_{ij}| 1\leq i, j\leq \dim V]\twoheadrightarrow \cO(\rho(G))\simeq \cO(G)$ corresponding to the restriction of the domain. By construction, each $R(X_{ij})$ ($1\leq i, j\leq \dim V$) is actually a matrix coefficient of $V$, hence the claim follows. 
\end{proof}

\begin{prop}\label{prop:minors} 
Let $G$ be a simply-connected simple algebraic group $G$ over $\C$ of type $A_n$, $B_n$, $C_n$, $D_n$, $E_6$, or $E_7$. Then $\cO(G)$ is generated by generalized minors. 
\end{prop}
\begin{proof}
By \cref{prop:generate}, it suffices to show that there exist minuscule representations $V_1,\dots, V_k$ such that their direct sum $V_1\oplus\cdots\oplus V_k$ provides a faithful representation of $G$. Hence the result follows from the following known facts in representation theory (see, for example, \cite[Appendix]{Ste94}): 
\begin{itemize}
    \item When $G$ is of type $A_n$, the $(n+1)$-dimensional vector representation is a faithful minuscule representation. 
    \item When $G$ is of type $B_n$, the $2^n$-dimensional spin representation is a faithful minuscule representation. 
    \item When $G$ is of type $C_n$, the $2n$-dimensional vector representation is a faithful minuscule representation.     
    \item When $G$ is of type $D_n$, there exist two mutually non-isomorphic $2^{n-1}$-dimensional minuscule representations, called the half spin representations. Then the direct sum of these two representations is  faithful. 
    \item When $G$ is of type $E_6$, there are two faithful minuscule representations of dimension 27.
    \item When $G$ is of type $E_7$, there exists a faithful minuscule representation of dimension 56.
\end{itemize}
\end{proof}
\begin{rem}
The simple algebraic group of type $E_8$, $F_4$, $G_2$ does not admit non-trivial minuscule representations. This is why we exclude these types from the assumption. However, the statement makes sense also for these types and we do not know whether it holds in these cases. 
\end{rem}

\subsection{Decorated flags and pinnings}
The homogeneous spaces $\A_G:=G/U^+$ and $\B_G:=G/B^+$ are called the \emph{principal affine space} and the \emph{flag variety}, respectively. An element of $\A_G$ (resp. $\B_G$) is called a \emph{decorated flag} (resp. a \emph{flag}). There is a canonical $G$-equivariant projection $\pi: \A_G \to \B_G$. 
The basepoint of $\A_G$ is denoted by $[U^+]$. We also adopt the notation $[U^-]:=\vw.[U^+] \in \A_G$. 
The flag variety $\B_G$ is identified with the set of connected maximal solvable subgroups of $G$ via $g.B^+\mapsto gB^+g^{-1}$. 
The Cartan subgroup $H$ acts on $\A_G$ from the right by $g.[U^+].h:=gh.[U^+]$ for $g \in G$ and $h \in H$, making the projection $\pi: \A_G \to \B_G$ a principal $H$-bundle. 

For $k \in \Z_{\geq 2}$, the configuration space of decorated flags is defined to be the stack
\[
\Conf_k \A_G := [\overbrace{\A_G \times \dots \times \A_G}^{k\text{ times}}/G],
\]
where we consider the diagonal left action of $G$. 

By the Bruhat decomposition $G = \bigcup_{w \in W(G)} U^+H \overline{w} U^+$, any $G$-orbit in the space $\Conf_2 \A_G$ has a unique representative of the form $(\flA_1,\flA_2)=\left(h.[U^+], \overline{w}.[U^+]\right)$
for some $h \in H$ and $w \in W(G)$. The parameters $h(\flA_1,\flA_2):=h$ and $w(\flA_1,\flA_2):=w$ are called the \emph{$h$-invariant} and the \emph{$w$-distance} of $(\flA_1,\flA_2)$, respectively. They only depend on the $G$-orbit $[\flA_1,\flA_2]$. See \cite[Section 2.2]{GS19} for details on the properties of these parameters. 
A pair $(\flA_1,\flA_2)$ of decorated flags (or its $G$-orbit) is said to be \emph{generic} 
if $w(\flA_1,\flA_2) = w_0$. Genericity only depends on the underlying flags. 

\smallskip 

\paragraph{\textbf{Angle invariant.}}
Consider the configuration space $\Conf (\A_G,\B_G,\B_G)$ parametrizing the triples $(\flA_1,\flB_2,\flB_3) \in \A_G\times \B_G \times \B_G$ such that the pairs $(\flA_1,\flB_2)$ and $(\flA_1,\flB_3)$ are generic. The $G$-orbit of such a triple has a unique representative of the form $(\flA_1,\flB_2,\flB_3)=([U^+],B^-,u_+B^-)$ for some $u \in U^+$. Then we define the \emph{angle invariant} to be
\begin{align}\label{eq:angle_inv}
     \mathrm{an}:\Conf(\A_G,\B_G,\B_G) \to U^+,\quad [[U^+],B^-,u_+B^-] \mapsto u_+.
\end{align}
Composing the canonical character $\chi_s:=\Delta_{\varpi_s,r_s\varpi_s}: U^+ \to \mathbb{C}$, we get the \emph{potential functions} $W_s:=\chi_s\circ \mathrm{an}: \Conf(\A_G,\B_G,\B_G) \to \mathbb{C}$ for $s \in S$. 

\smallskip 

\paragraph{\textbf{Pinnings.}}
Following \cite{GS19}, we define a \emph{pinning} to be a generic pair $(\flA_1,\flB_2) \in \A_G\times \B_G$. Then the space $\P_G$ of pinnings is naturally a principal $G$-space. We will use the \emph{standard pinning} $p_\std:=([U^+],B^-)$ as a basepoint of this space. 
We define the \emph{opposite pinning} of $p=g.p_\std$ as $p^\ast:=g\vw^{-1}.p_\std$. 
Alternatively, a pinning is defined to be a generic pair $(\flA_1,\flA_2) \in \A_G\times \A_G$ such that $h(\flA_1,\flA_2)$ is trivial. Then we have $p_\std=([U^+],[U^-])$ and $p_\std^\ast=(s_G.[U^-],[U^+])$. 

\medskip
\section{Wilson lines on \texorpdfstring{$\A_{G,\Sigma}^\times$}{A(G,S)*}}\label{sec:Wilson_line}

\subsection{Notations on marked surfaces}\label{subsec:marked_surface}
A \emph{marked surface} $(\Sigma,\bM)$ consists of a  compact oriented surface $\Sigma$ and a fixed non-empty finite set $\bM \subset \partial\Sigma$ of \emph{marked points}. In particular, we do not consider marked points in the interior (\lq\lq punctures'').  
In this paper, we always assume that 
\begin{align}
    \text{each boundary component has at least one marked point.}
\end{align}
A connected component of $\partial\Sigma\setminus \bM$ is called a \emph{boundary interval}. The set of boundary intervals is denoted by $\bB=\bB(\Sigma)$. 
When no confusion can occur, we will simply denote a marked surface by $\Sigma$. 
We will always assume that
\begin{align}
    n(\Sigma):=-2\chi(\Sigma)+|\bM|>0.
\end{align}
This condition ensures that the marked surface $\Sigma$ admits an ideal triangulation with $n(\Sigma)$ triangles. 

The fiber bundle $T'\Sigma:=T\Sigma \setminus (\mbox{$0$-section})$ with fiber $\mathbb{R}^2 \setminus \{0\}$ is called the \emph{punctured tangent bundle}. The bundle projection $\pi:T'\Sigma \to \Sigma$ induces the exact sequence
\begin{align}\label{eq:bundle_sequence}
    1 \to \pi_1(\mathbb{R}^2\setminus \{0\}) \to \pi_1(T'\Sigma) \xrightarrow{\pi_\ast} \pi_1(\Sigma) \to 1.
\end{align}
The generator of $\pi_1(\mathbb{R}^2\setminus \{0\})\cong\mathbb{Z}$ is denoted by $\fiber$. 

\subsection{The boundary-fundamental groupoid.}\label{subsec:groupoid}
Now we are going to clarify our topological set-up for the definition of Wilson lines on $\A_{G,\Sigma}^\times$, for which we need a little care on the fiber direction of $T'\Sigma$. 
We fix a nowhere-vanishing vector field $v_\mathrm{ori}$ on $\partial\Sigma$ which induces an orientation on $\partial\Sigma$ compatible with that induced from $\Sigma$. 
Then we get two embeddings $\iota_\pm: \partial \Sigma \to T'\Sigma$, $x \mapsto (x,\pm v_\mathrm{ori}(x))$. For each boundary interval $E \in \bB$, fix a reference point $x_E \in E$ and define $E^\pm:=\iota_\pm(x_E) \in T'\Sigma$. We regard the points $E^\pm$ as representatives of oriented boundary intervals. 

\begin{dfn}\label{def:fundamental_groupoid}
The \emph{boundary-fundamental groupoid} $\Pi_1(T'\Sigma,\bB^\pm)$ of $T'\Sigma$ is the groupoid  where the objects are the points $E^\epsilon$ for $E \in \bB$ and $\epsilon \in \{+,-\}$, and morphisms from $E_1^{\epsilon_1}$ to $E_2^{\epsilon_2}$ are based-homotopy classes of continuous curves in $T'\Sigma$ from $E_1^{\epsilon_1}$ to $E_2^{\epsilon_2}$.  
We call a morphism $[c]:E_1^{\epsilon_1} \to E_2^{\epsilon_2}$ in this groupoid a \emph{framed arc class} in $\Sigma$. The composition of framed arc classes $[c_1]: E_1^{\epsilon_1} \to E_2^{\epsilon_2}$ and $[c_2]: E_2^{\epsilon_2} \to E_3^{\epsilon_3}$ is given by the concatenation $[c_1]\ast[c_2]: E_1^{\epsilon_1} \to E_3^{\epsilon_3}$.
\end{dfn}
Recall from \cite{IO20} the groupoid $\Pi_1(\Sigma,\bB)$, whose objects are the points $x_E$ and the morphisms are based-homotopy classes of continuous curves between them. 
We have a natural projection $\pi_\ast:\Pi_1(T'\Sigma,\bB^\pm) \to \Pi_1(\Sigma,\bB)$.

\begin{dfn}[transverse immersions and their standard lifts]
A \emph{transverse immersion} is an immersed curve $\overline{c}:[0,1] \to \Sigma$ such that $\overline{c}(0)=x_{E_0}$ and $\overline{c}(1)=x_{E_1}$ for some boundary intervals $E_0,E_1 \in \bB$, and transverse to $\partial\Sigma$. Then its \emph{standard lift} is the continuous curve $c:[0,1] \to T'\Sigma$ from $E_0^-$ to $E_1^-$ obtained as the composite of the following two paths:
\begin{itemize}
    \item the path $t \mapsto (\overline{c}(t),v_{\overline{c}}^\bot(t))$ in $T'\Sigma$, where $v_{\overline{c}}^\bot(t)$ is a nowhere-vanishing normal vector field along $\overline{c}$ which points toward the left side of $\overline{c}$, and such that $v_{\overline{c}}^\bot(0)=E_0^-$, $v_{\overline{c}}^\bot(1)=E_1^+$. 
    \item a path from $E_1^+$ to $E_1^-$ in the tangent space at $x_{E_1}$, which rotates the tangent vector clockwisely.
\end{itemize}
See \cref{fig:curve_lift}. Note that the homotopy class $[c]$ does not depend on the choice of the vector field $v_{\overline{c}}^\bot$ which satisfies the condition. 
\end{dfn}

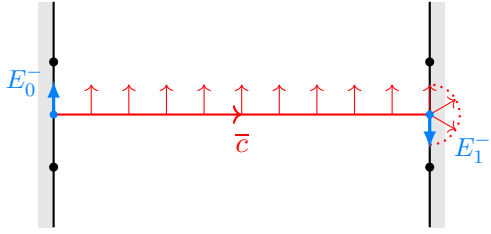
\begin{figure}
\begin{tikzpicture}
\fill[gray!20] (0,1.5) -- (-0.2,1.5) -- (-0.2,-1.5) -- (0,-1.5) --cycle;
\fill[gray!20] (5,1.5) -- (5+0.2,1.5) -- (5+0.2,-1.5) -- (5,-1.5) --cycle;
\draw[thick] (0,1.5) -- (0,-1.5);
\draw[thick] (5,-1.5) -- (5,1.5);
\filldraw(0,0.7) circle(1.5pt); 
\filldraw(0,-0.7) circle(1.5pt);
\filldraw(5,0.7) circle(1.5pt);
\filldraw(5,-0.7) circle(1.5pt);
\draw[red,thick,->-] (0,0) --node[midway,below=0.3em]{$\overline{c}$} (5,0);
{\color{red}
\foreach \i in {90,30,-30}
\tangent{5,0}{\i}{0.4};
\draw[thick,dotted] (5,-0.4) arc(-90:90:0.4);
\foreach \x in {0.5,1,1.5,2,2.5,3,3.5,4,4.5}
\tangent{\x,0}{90}{0.4};
\fill[myblue] (0,0) circle(1.5pt);
\draw[myblue,very thick,-latex] (0,0)--++(0,0.45) node[left,scale=0.9]{$E_0^-$};
\fill[myblue] (5,0) circle(1.5pt);
\draw[myblue,very thick,-latex] (5,0)--++(0,-0.45) node[right=0.5em,scale=0.9]{$E_1^-$};
}
\end{tikzpicture}
    \caption{The standard lift of a transverse immersion.}
    \label{fig:curve_lift}
\end{figure}

\begin{rem}
The fundamental groupoid $\Pi_1(T'\Sigma,\bB^\pm)$ has appeared in \cite[Section 8]{CL19} in their study on stated skein algebras. Our standard lifts are their ``good lifts'' with respect to the negative orientation, which are chosen so that certain matrix coefficients of the Wilson line along them give rise to cluster variables. See \cref{subsec:computation_minors}.
\end{rem}
For later use, we are going to give a good generating set of the fundamental groupoid $\Pi_1(T'\Sigma,\bB^\pm)$. For each boundary interval $E \in \bB$, let $\sqrt{\fiber_E}^{\inn},\sqrt{\fiber_E}^{\out}:E^- \to E^+$ denote the framed arc classes contained in the tangent space at $x_E$ which go from $E^-$ to $E^+$ clockwisely and counter-clockwisely, respectively. Then 
\begin{align*}
    \fiber_{E^-}:=\sqrt{\fiber_E}^{\inn}\ast (\sqrt{\fiber_E}^{\out})^{-1} \in \mathrm{Aut}(E^-)\quad \mbox{and} \quad \fiber_{E^+}:=(\sqrt{\fiber_E}^{\out})^{-1}\ast \sqrt{\fiber_E}^{\inn} \in \mathrm{Aut}(E^+)
\end{align*}
represent the clockwise fiber loop. We say that a framed arc class $[c]:E_1^- \to E_2^-$ is \emph{simple} if $E_1 \neq E_2$ and it is represented by the standard lift $c$ of a transverse immersion $\overline{c}$ without self-intersection. Notice that it implies that there exists a band neighborhood $B_c$ of $\overline{c}$ which gives an immersed quadrilateral which is embedded except for its vertices\footnotemark, as shown in \cref{fig:band_nbd}.

\begin{figure}[ht]
\begin{tikzpicture}
\draw(0,-0.2) ellipse(4.5cm and 2.5cm);
\draw[shorten >=-15pt,shorten <=-15pt] (-1,1) to[bend right=20] 
node[inner sep=0,pos=-0.15](A){} node[inner sep=0,pos=1.15](B){} 
(1,1);
\draw(A) to[bend left=20] node[inner sep=0,pos=0.7](C){} (B);
\filldraw[thick,fill=gray!20] (-3,-0.5) circle(.7cm);
\filldraw[thick,fill=gray!20] (3,-0.5) circle(.7cm);
\foreach \i in {-30,90,210}
\fill(-3,-0.5)++(\i:0.7) circle(2pt);
\draw(-3,-0.5)++(90:0.7) node[above=0.2em,scale=0.9]{$m_L$};
\draw(-3,-0.5)++(-30:0.7) node[below right,scale=0.9]{$m_R$};
\foreach \i in {120,240}
\fill(3,-0.5)++(\i:0.7) circle(2pt);
\draw(3,-0.5)++(120:0.7) node[above,scale=0.9]{$m^L$};
\draw(3,-0.5)++(240:0.7) node[below,scale=0.9]{$m^R$};
\draw[red,thick,->-={0.6}{}] (-3,-0.5)++(30:0.7) to[out=30,in=180] node[midway,above]{$c$} ($(3,-0.5)+(180:0.7)$);
\begin{pgfonlayer}{bg}  
\filldraw[fill=pink!30,draw=red,dashed] (-3,-0.5)++(90:0.7) to[out=30,in=180] ($(3,-0.5)+(120:0.7)$) arc[start angle=120, end angle=240, radius=0.7] to[out=180,in=30] node[red,midway,below]{$B_c$} ($(-3,-0.5)+(-30:0.7)$) arc[start angle=-30, end angle=60, radius=0.7];
\end{pgfonlayer}
\end{tikzpicture}
    \caption{A simple framed arc class $[c]:E_1^- \to E_2^-$ and a band neighborhood $B_c$. Here it is allowed that one of $m_L=m^L$ or $m_R=m^R$ holds.}
    \label{fig:band_nbd}
\end{figure}
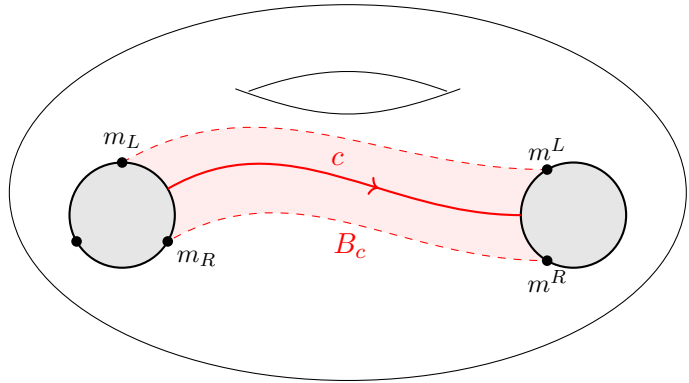

\footnotetext{This is seen as follows. Since the transverse immersion $\overline{c}:[0,1] \to \Sigma$ has no self-intersections,  we can surely take an embedding $\iota:[0,1] \times [-\epsilon,\epsilon] \to \Sigma$ so that $\iota(t,0)=\overline{c}(t)$, $I_1:=\iota(\{0\} \times [-\epsilon,\epsilon]) \subset E_1$ and $I_2:=\iota(\{1\} \times [-\epsilon,\epsilon]) \subset E_2$. Since $E_1 \neq E_2$ by assumption, we can modify ('spread out') this embedding in small neighborhoods of $E_1$ and $E_2$ so that $I_1=E_1$ and $I_2=E_2$. The image of this modified map is $B_c$.}

\begin{lem}\label{lem:groupoid_generator}
The fundamental groupoid $\Pi_1(T'\Sigma,\bB^\pm)$ is finitely generated. Namely, there exists a finite set $\mathsf{S}$ of framed arc classes such that any framed arc class can be written as a finite concatenation of those in $\mathsf{S}$ and their inverses. Moreover if $\Sigma$ has at least two marked points, then the generating set $\mathsf{S}$ can be chosen so that it consists of simple framed arc classes and $\sqrt{\fiber_E}^{\inn},\sqrt{\fiber_E}^{\out}$ for $E \in \mathbb{B}$.
\end{lem}

\begin{proof}
Fix a boundary interval $E_0 \in \bB$. We choose a collection of transverse immersions in $\Sigma$, as follows.
\begin{itemize}
    \item Let $S_0=\{(\overline{\alpha}_i)_{i=1}^g,(\overline{\beta}_i)_{i=1}^g,(\overline{\gamma}_j)_{j=1}^b\}$ be transversely-immersed loops based at $x_0:=x_{E_0}$ without self-intersections, whose homotopy classes generate the fundamental group $\pi_1(\Sigma,x_0)$. 
    \item For each boundary interval $E \neq E_0$, let $\overline{\epsilon}_E$ be a transverse immersion without self-intersection running from $x_0$ to $x_E$. 
\end{itemize}
We first claim that the standard lifts of these transverse immersions, together with the framed arc classes $\sqrt{\fiber_E}^{\inn},\sqrt{\fiber_E}^{\out}$ for $E \in \mathbb{B}$, generate the fundamental groupoid. Let $\mathcal{G} \subset \Pi_1(T'\Sigma,\bB^\pm)$ denote the sub-groupoid generated by these framed arc classes.

First note that the fiber loops $\fiber_{E^\pm}$ are contained in $\mathcal{G}$. 
Take an arbitary framed arc class $[c]:E_1^{\epsilon_1} \to E_2^{\epsilon_2}$. By concatenating the framed arc classes $\sqrt{\fiber_{E_i}}^{\out}$, $i=1,2$ or their inverses if necessary, we may assume $\epsilon_1=\epsilon_2=-$. 
Consider its projection $[\overline{c}]:=\pi_\ast([c])$. Then the concatenation $[\overline{\epsilon}_{E_1}]\ast [\overline{c}] \ast [\overline{\epsilon}_{E_2}]^{-1}$ is an element of $\pi_1(\Sigma,x_0)$, which can be written as a concatenation of the elements in $S_0$. It implies that there exists a framed arc class $[c']: E_1^- \to E_2^-$ that lies in $\mathcal{G}$ such that $\pi_\ast([c'])=\pi_\ast([c])$. Then by the exact sequence \eqref{eq:bundle_sequence}, the arc class $[c]$ can be written as a concatenation of $[c']$ and several powers of fiber loops. Hence $[c]$ belongs to $\mathcal{G}$. Thus the first assertion is proved.

Assume that $\Sigma$ has at least two marked points. It implies that $\Sigma$ has at least two boundary intervals, say $E_0 \neq E_1$. Then we can decompose each loop $\eta \in S_0$ into two simple arcs $\eta_1$ and $\eta_2$, as shown in \cref{fig:decomposition_simple_class}. Therefore we can replace the standard lift of $\eta$ with the standard lifts of $\eta_1$ and $\eta_2$ in the generating set, where the latter two are simple framed arc classes. Thus the second assertion is proved.
\end{proof}

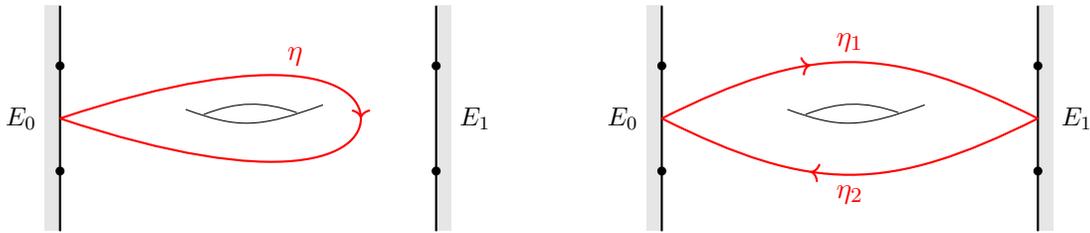
\begin{figure}[htbp]
\centering
\begin{tikzpicture}
\fill[gray!20] (0,1.5) -- (-0.2,1.5) -- (-0.2,-1.5) -- (0,-1.5) --cycle;
\fill[gray!20] (5,1.5) -- (5+0.2,1.5) -- (5+0.2,-1.5) -- (5,-1.5) --cycle;
\draw[thick] (0,1.5) -- (0,-1.5);
\draw[thick] (5,-1.5) -- (5,1.5);
\filldraw(0,0.7) circle(1.5pt); 
\filldraw(0,-0.7) circle(1.5pt);
\filldraw(5,0.7) circle(1.5pt);
\filldraw(5,-0.7) circle(1.5pt);
\draw[red,thick,->-] (0,0) ..controls (3,1) and (4,0.5).. node[midway,above]{$\eta$} (4,0)
..controls (4,-0.5) and (3,-1).. (0,0);
\draw[shorten >=-15pt,shorten <=-10pt] (2,0) to[bend right=20] 
node[inner sep=0,pos=-0.10](A){} node[inner sep=0,pos=1.15](B){} 
(3,0);
\draw(A) to[bend left=20] node[inner sep=0,pos=0.7](C){} (B);
\draw (0,0) node[left=0.5em,scale=0.9]{$E_0$};
\draw(5,0) node[right=0.5em,scale=0.9]{$E_1$};

\begin{scope}[xshift=8cm]
\fill[gray!20] (0,1.5) -- (-0.2,1.5) -- (-0.2,-1.5) -- (0,-1.5) --cycle;
\fill[gray!20] (5,1.5) -- (5+0.2,1.5) -- (5+0.2,-1.5) -- (5,-1.5) --cycle;
\draw[thick] (0,1.5) -- (0,-1.5);
\draw[thick] (5,-1.5) -- (5,1.5);
\filldraw(0,0.7) circle(1.5pt); 
\filldraw(0,-0.7) circle(1.5pt);
\filldraw(5,0.7) circle(1.5pt);
\filldraw(5,-0.7) circle(1.5pt);
\draw[red,thick,->-={0.4}{}] (0,0) ..controls (2,1) and (3,1).. node[midway,above]{$\eta_1$} (5,0);
\draw[red,thick,->-={0.6}{}] (5,0) ..controls (3,-1) and (2,-1)..node[midway,below]{$\eta_2$} (0,0);
\draw[shorten >=-15pt,shorten <=-10pt] (2,0) to[bend right=20] 
node[inner sep=0,pos=-0.10](A){} node[inner sep=0,pos=1.15](B){} 
(3,0);
\draw(A) to[bend left=20] node[inner sep=0,pos=0.7](C){} (B);
\draw (0,0) node[left=0.5em,scale=0.9]{$E_0$};
\draw(5,0) node[right=0.5em,scale=0.9]{$E_1$};
\end{scope}
\end{tikzpicture}
    \caption{Decomposition of a loop into two arcs.}
    \label{fig:decomposition_simple_class}
\end{figure}

\begin{rem}\label{rem:opposite_arc}
For a transverse immersion $\overline{c}:[0,1] \to \Sigma$, let $\overline{c}^{\op}:[0,1] \to \Sigma$ be the transverse immersion in the opposite direction given by $\overline{c}^{\op}(t):=\overline{c}(1-t)$. Let $c: E_0^- \to E_1^-$, $c^{\op}:E_1^- \to E_0^-$ be the standard lifts of $\overline{c}$, $\overline{c}^{\op}$, respectively. Then they satisfy 
\begin{align*}
    [c] \ast [c^{\op}] = \fiber_{E_0^-}, \quad [c^{\op}]\ast [c] = \fiber_{E_1^-}.
\end{align*}
In particular, note that $[c]^{-1} \neq [c^{\op}]$ as framed arc classes. 
\end{rem}

\subsection{Wilson lines on $\A_{G,\Sigma}^\times$}

Let $G$ be a semisimple simply-connected algebraic group, $\Sigma$ a marked surface without punctures. 
Fix an outward tangent vector field $v_\out$ on $\partial\Sigma$, which gives rise to an embedding $\iota_\out:\partial\Sigma \to T'\Sigma$, $x \mapsto (x,v_\out(x))$. 
Recall from \cite[Definition 2.3]{FG03} that a \emph{twisted $G$-local system on $\Sigma$} is a $G$-local system on the punctured tangent bundle $T'\Sigma$ with the monodromy $s_G$ along the fiber loop $\fiber$. A \emph{decoration} of a twisted $G$-local system $\L$ is a flat section $\alpha$ of the associated bundle $\L_\A:=\L \times_G \A_G$ defined over (a small neighborhood of) $\iota_\out(\partial\Sigma \setminus \{x_E\}_{E \in \bB})$.
In particular, the flat section is defined on a vicinity of the outward lift of each marked point. A decoration is said to be \emph{generic} if for each $E \in \bB$, the pair of flat sections defined near the endpoints of $E$ gives rise to a generic pair of decorated flags when they are evaluated at a common point via a parallel-transport.

\begin{dfn}[{\cite[Definition 2.4]{FG03}}]
Let $\A_{G,\Sigma}$ denote the moduli space of decorated twisted $G$-local systems on $\Sigma$. The open substack consisting of the generically decorated twisted $G$-local systems is denoted by $\A_{G,\Sigma}^\times \subset \A_{G,\Sigma}$. 
\end{dfn}
Let us briefly mention the description of $\A_{G,\Sigma}$ as a quotient stack. Fix a basepoint $\xi=(x,v) \in T'\Sigma$. 
A \emph{rigidification} of a twisted $G$-local system $(\L,\alpha)$ at $\xi$ is a choice of a point $s \in \L_\xi$. Let $A_{G,\Sigma}$ denote the set of isomorphism classes of rigidified twisted $G$-local systems $(\L,\alpha;s)$. The group $G$ acts on $A_{G,\Sigma}$ by $g.[\L,\alpha;s] :=[\L,\alpha;s.g]$ for $g\in G$. 
The following can be verified similarly to \cite[Definition 2.2]{FG03}, with a little care on twistings:

\begin{lem}\label{lem:moduli_atlas}
The set $A_{G,\Sigma}$ has a natural structure of quasi-affine $G$-variety, isomorphic to
\begin{align*}
     \Hom^{\mathrm{tw}}(\pi_1(T'\Sigma,\xi),G) \times (\A_G)^{\bM}.
\end{align*}
Here $\Hom^{\mathrm{tw}}(\pi_1(T'\Sigma,\xi),G) \subset \Hom(\pi_1(T'\Sigma,\xi),G)$ denotes the subspace such that $\rho(\fiber_\xi)=s_G$; $\fiber_\xi \in \pi_1(T'\Sigma,\xi)$ being the fiber loop based at $\xi$. 
\end{lem}

Thus we get a stacky definition $\A_{G,\Sigma}:=[A_{G,\Sigma}/G]$. In particular, its function ring is $\cO(\A_{G,\Sigma})=\cO(A_{G,\Sigma})^G$. 

\begin{rem}\label{rem:integral}
Since $\pi_1(\Sigma)$ is a free group, the central extension \eqref{eq:bundle_sequence} splits. Therefore we have non-canonical isomorphisms $\pi_1(T'\Sigma) \cong \pi_1(\Sigma) \times \Z$ and $\Hom^{\mathrm{tw}}(\pi_1(T'\Sigma,\xi),G) \cong \Hom(\pi_1(\Sigma),G) \cong G^{2g+b-1}$, where $b$ denotes the number of boundary components of $\Sigma$. In particular, $\cO(A_{G,\Sigma})$ is an integral domain, so is its subalgebra $\cO(\A_{G,\Sigma})$. The field $\mathcal{K}(\A_{G,\Sigma})$ of rational functions is defined to be its field of fractions. 
\end{rem}

\begin{ex}\label{ex:polygon}
Let $\Sigma=\mathbb{D}_k$ be a $k$-gon, which is a disk with $k$ marked points on the boundary. 
Choose one distinguished marked point, and let $m_1,\dots,m_k$ denote the marked points in the counter-clockwise order, where $m_1$ is the distinguished one. 
Fix a trivialization $T'\mathbb{D}_k \cong \mathbb{D}_k \times (\bR^2 \setminus \{0\})$, and let $v_j \in T'_x \mathbb{D}_k$ denote a tangent vector points toward $m_j$ for $j=1,\dots,k$. 
We may assume that $\xi=(x,v_1)$. Then we choose the following paths in $T'\mathbb{D}_k$:
\begin{itemize}
    \item The straight line $\epsilon_1$ from $\xi$ to $\iota_\out(m_1)$;
    \item For $j=2,\dots,k$, the path $\epsilon_j$ which first rotates from $\xi=(x,v_1)$ to $(x,v_j)$ in the counter-clockwise direction inside the tangent space at $x$, and then goes straight from $(x,v_j)$ to $\iota_\out(m_j)$.
\end{itemize}
See the left picture in \cref{fig:polygon}. Let $(\L,\alpha;s)$ be a rigidified twisted $G$-local system. The rigidification $s$ determines a trivialization $(\L_\A)_\xi \cong \A_G$. 
We have a flat section $\alpha_j$ of $\L_\A$ defined near $\iota_\out(m_j)$ for $j=1,\dots,k$. Via the parallel-transport along the path $\epsilon_j$, we can evaluate it at the basepoint $\xi$, which gives rise to a decorated flag $\alpha_j(\xi) \in (\L_\A)_\xi \cong \A_G$. Thus we get a $G$-equivariant isomorphism
\begin{align*}
    A_{G,\mathbb{D}_k} \xrightarrow{\sim} \A_G^k, \quad (\L,\alpha;s) \mapsto (\alpha_1(\xi),\alpha_2(\xi),\dots,\alpha_k(\xi)),
\end{align*}
which descends to an isomorphism $f_{m_1}:\A_{G,\mathbb{D}_k} \xrightarrow{\sim} \Conf_k \A_G:=[\A_G^k/G]$ of stacks. Alternatively, one can think that we are choosing a branch cut $\ell$ in the fiber direction $\bR^2 \setminus \{0\}$ as shown in \cref{fig:polygon} by waved orange line, and trivializing the twisted local systems on the contractible region $\mathbb{D}_k \times (\bR^2 \setminus \ell)$. When we discuss the polygon case, we will only show this branch cut to indicate the isomorphism we use.

When we replace the distinguished marked point $m_1$ with $m_2$ and use a similar choice of paths as shown in the right picture in \cref{fig:polygon}, then we get another isomorphism $f_{m_2}:\A_{G,\mathbb{D}_k} \xrightarrow{\sim} \Conf_k \A_G$. Then the coordinate transformation $f_{m_2}\circ f_{m_1}^{-1}$ is given by the \emph{twisted cyclic shift}
\begin{align*}
    \mathcal{S}_k: \Conf_k \A_G \xrightarrow{\sim} \Conf_k \A_G, \quad [\flA_1,\dots,\flA_{k-1},\flA_k] \mapsto [\flA_2,\dots,\flA_k,s_G.\flA_1].
\end{align*}
The substack of $\Conf_k \A_G$ corresponding to $\A_{G,\mathbb{D}_k}^\times$ is denoted by $\Conf_k^\times \A_G$. 
\end{ex}

\begin{figure}[htbp]
\begin{tikzpicture}
\draw (0,0) circle(2cm);
\foreach \i in {1,2,3,4,5,6}
{
\fill (\i*60+30:2) circle(1.5pt);
\draw[thick,-latex] (\i*60+30:2) --++(\i*60+30:0.4);
} 
\draw [myorange,thick,decorate,decoration={snake,amplitude=2pt,pre length=2pt,post length=0pt}](0,0) -- (60:2.4);
\draw[red,very thick,-latex] (0,0)  --++(90:0.4);
\draw[red,thick,dashed,->-] (0,0) -- (0,2);
\draw[red,thick,dashed] (0,0.4) arc(90:390:0.4);
\foreach \j in {0,2,3,4,5}
\draw[red,thick,dashed,->-] (\j*60+30:0.4) --++(\j*60+30:1.6);
\foreach \i in {2,3,5,6}
\node[scale=0.9] at (\i*60+30:2.9) {$\iota_\mathrm{out}(m_\i)$};
\foreach \i in {1,4}
\node[scale=0.9] at (\i*60+30:2.7) {$\iota_\mathrm{out}(m_\i)$};
\foreach \i in {1,2,3,4,5,6}
\node[red] at (\i*60+20:1.4) {$\epsilon_\i$};

\begin{scope}[xshift=8cm]
\draw (0,0) circle(2cm);
\foreach \i in {1,2,3,4,5,6}
{
\fill (\i*60+30:2) circle(1.5pt);
\draw[thick,-latex] (\i*60+30:2) --++(\i*60+30:0.4);
} 
\draw[red,very thick,-latex] (0,0)  --++(150:0.4);
\draw [myorange,thick,decorate,decoration={snake,amplitude=2pt,pre length=2pt,post length=0pt}](0,0) -- (120:2.4);
\draw[red,thick,dashed,->-] (0,0) -- (150:2);
\draw[red,thick,dashed] (150:0.4) arc(150:450:0.4);
\foreach \j in {0,1,3,4,5}
\draw[red,thick,dashed,->-] (\j*60+30:0.4) --++(\j*60+30:1.6);
\foreach \i in {2,3,5,6}
\node[scale=0.9] at (\i*60+30:2.9) {$\iota_\mathrm{out}(m_\i)$};
\foreach \i in {1,4}
\node[scale=0.9] at (\i*60+30:2.7) {$\iota_\mathrm{out}(m_\i)$};
\foreach \i in {1,2,3,4,5,6}
\node[red] at (\i*60+20:1.4) {$\epsilon'_\i$};
\end{scope}
\end{tikzpicture}
    \caption{Construction of isomorphisms $\A_{G,\mathbb{D}_k} \cong \Conf_k \A_G$ with $k=6$. The flat sections defined near $\iota_\out(m_j)$ are parallel-transported along the dashed lines towards the basepoint $\xi$ (shown by the solid arrow in the center of the disk).}
    \label{fig:polygon}
\end{figure}
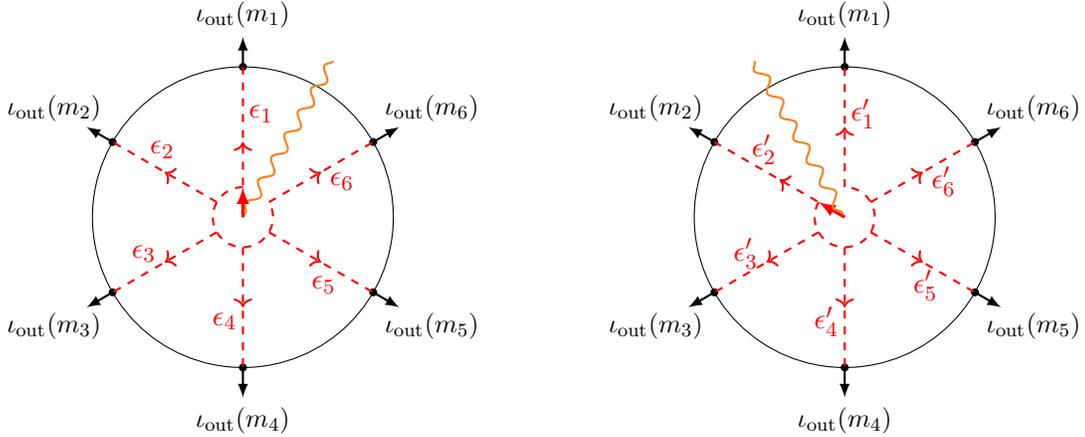

\paragraph{\textbf{Pinnings.}}

For each boundary interval $E \in \bB$, let $m^\pm_E \in \bM$ be its endpoint in the direction $E^\pm$. Given a generic decorated twisted $G$-local system $(\L,\alpha)$, let $\alpha^\pm_E$ be the flat section defined near $\iota_\out(m_E^\pm)$. To the pair $(\alpha_E^-,\alpha_E^+)$, we associate a pinning as follows. Parallel-transport the flat sections $\alpha_E^\pm$ along the line $\iota_\out(E)$ to the common point $\iota_\out(x_E)$, and continue to transport them along the path in the tangent space at $x_E$ to the point $E^-=\iota_-(x_E)$ in the clockwise direction. Then we get a generic pair $(\flA_E^-,\flA_E^+)$ in $(\L_\A)_{E^-}$, and the pinning $p_{E^-}:=(\flA_E^-,\pi(\flA_E^+)) \in (\L\times_G \P_G)_{E^-}$. We may associate another pinning $p_{E^+}:=(\hat{\flA}_E^+,\pi(\hat{\flA}_E^-)) \in (\L\times_G \P_G)_{E^+}$, where the pair $(\hat{\flA}_E^-,\hat{\flA}_E^+)$ in $(\L_\A)_{E^+}$ is obtained by the parallel-transport of the pair $(\flA_E^-,\flA_E^+)$ via the outward path $\sqrt{\fiber_{E}}^{\out}$. 

\begin{rem}
Thanks to the absence of punctures, the decorated twisted $G$-local system $(\L,\alpha)$ can be uniquely recovered from the data $(\L,(p_{E^-})_{E \in \bB})$ or $(\L,(p_{E^+})_{E \in \bB})$. Thus the subspace $\A_{G,\Sigma}^\times$ can be identified with a variant $\P_{G,\Sigma}$ of the $\P$-type moduli space in \cite{GS19} for simply-connected groups $G$. Then similarly to \cite[Lemma 3.8]{IO20}, one can verify that $\A_{G,\Sigma}^\times$ is an algebraic variety.
\end{rem}

Now we define the Wilson lines on $\A_{G,\Sigma}^\times$ in a similar manner as in \cite{IO20}. Let $[c]: E_1^{\epsilon_1} \to E_2^{\epsilon_2}$ be a framed arc class, and $(\L,\alpha)$ a generically decorated twisted $G$-local system. Choose a local trivialization $s_1$ of $\L$ on a vicinity of $E_1^{\epsilon_1}$ so that $p_{E_1^{\epsilon_1}}=p_\std$. Extend $s_1$ via the parallel-transport along the path $c$, until the terminal point $E_2^{\epsilon_2}$. Then 
\[p_{E_2^{\epsilon_2}}=g.p_\std^\ast=g\vw^{-1}.p_\std
\]
for a unique element $g \in G$ under this trivialization. Define $g_{[c]}([\L,\alpha]):=g$. The following can be verified in the same way as \cite[Proposition 3.25]{IO20}:

\begin{lemdef}[Wilson lines]
    For any framed arc class $[c]: E_1^{\epsilon_1} \to E_2^{\epsilon_2}$, the construction above produces a morphism
    \begin{align*}
        g_{[c]}: \A_{G,\Sigma}^\times \to G
    \end{align*}
    of stacks, which we call the \emph{Wilson line} along $[c]$. The morphism $g_{[c]}^\mathrm{tw}:=g_{[c]}\vw^{-1}$ is called the \emph{twisted Wilson line}.
\end{lemdef}
These Wilson lines are lifts of those introduced in \cite{IO20}.

\begin{rem}
At first, the twisted Wilson lines $g^\mathrm{tw}_{[c]}$ might look more natural than $g_{[c]}$. However, it turns out that the Wilson lines $g_{[c]}$ are compatible with the positivity structure, especially with the cluster structure. The twisted ones have negative coefficients due to $-1$ contained in $\vw$. The Wilson lines are also defined so as to fit in well with the amalgamation (cf.~\cite[Proposition 3.27]{IO20}).
\end{rem}
In order to state the relation to the Wilson lines in \cite{IO20}, let $G'=G/Z(G)$ be the adjoint group, and recall the moduli space $\P_{G',\Sigma}$ of framed $G'$-local systems with pinnings \cite{GS19}.

\begin{prop}
For any framed arc class $[c]: E_1^{\epsilon_1} \to E_2^{\epsilon_2}$, the following diagram commutes:
\begin{equation*}
    \begin{tikzcd}
    \A_{G,\Sigma}^\times \ar[d,"p_\Sigma"'] \ar[r,"g_{[c]}"] & G \ar[d] \\
    \P_{G',\Sigma} \ar[r,"g_{[\overline{c}]}"'] & G'.
    \end{tikzcd}
\end{equation*}
Here $p_\Sigma: \A_{G,\Sigma}^\times \to \P_{G',\Sigma}$ denotes the projection given in \cite[Section 9.2]{GS19}, $G \to G'$ is the canonical projection, and $g_{[\overline{c}]}$ is the Wilson line along the arc class $[\overline{c}]:=\pi_\ast([c])$ \cite{IO20}.
\end{prop}

The twisted Wilson lines contain the Wilson loops in the following sense:

\begin{prop}\label{lem:Wilson_loop}
For a free loop $|\gamma|$ on $T'\Sigma$, choose a representative $c$ based at $E^\epsilon$ for some $E \in \bB$ and $\epsilon \in \{+,-\}$, whose based homotopy class defines a framed arc class $[c]:E^\epsilon \to E^\epsilon$. Then the twisted Wilson line $g_{[c]}^\mathrm{tw}$ represents the Wilson line $\rho_{|\gamma|}$, namely, the following diagram commutes:
\begin{equation*}
    \begin{tikzcd}
    \A_{G,\Sigma}^\times \ar[r,"g_{[c]}^\mathrm{tw}"] \ar[rd,"\rho_{|\gamma|}"'] & G \ar[d] \\
     & {[G/\Ad G]}.
     \end{tikzcd}
\end{equation*}
Here $G \to [G/\Ad G]$ denotes the canonical projection. 
\end{prop}
 

\begin{prop}[Internal multiplicativity]\label{prop:multiplicativity}
For any framed arc classes $[c_1]:E_1^{\epsilon_1} \to E_2^{\epsilon_2}$ and $[c_2]:E_2^{\epsilon_2} \to E_3^{\epsilon_3}$, we have
\begin{align*}
    g_{[c_1]\ast[c_2]}^\mathrm{tw}=g_{[c_1]}^\mathrm{tw} g_{[c_2]}^\mathrm{tw} \quad \text{or equivalently,} \quad g_{[c_1]\ast[c_2]}=g_{[c_1]}\vw^{-1} g_{[c_2]}.
\end{align*}
In other words, for any point $[\L,\alpha] \in \A_{G,\Sigma}^\times$, the twisted Wilson lines $g^\mathrm{tw}_\bullet([\L,\alpha])$ defines a morphism $\Pi_1(T'\Sigma,\bB^\pm) \to G$ of groupoids.
\end{prop}

\begin{proof}
Given $[\L,\alpha] \in \A_{G,\Sigma}^\times$, let us prove
\begin{align*}
    g^\mathrm{tw}_{[c_1]\ast[c_2]}([\L,\alpha])=g^\mathrm{tw}_{[c_1]}([\L,\alpha])g^\mathrm{tw}_{[c_2]}([\L,\alpha]).
\end{align*}
Take a local trivialization $s_1$ of $\L$ so that $(p_{E_1},p_{E_2})=(p_\std,g_1.p_\std)$, where $g_1:=g^\mathrm{tw}_{[c_1]}([\L,\alpha])$. Under the local trivialization $s_2:=s_1.g_1$, the sections are given by $(p_{E_1},p_{E_2},p_{E_3})=(g_1^{-1}.p_\std,p_\std,g_2.p_\std)$, where $g_2=g^\mathrm{tw}_{[c_2]}([\L,\alpha])$ by definition. Going back to the first trivialization, we get $(p_{E_1},p_{E_3})=(p_\std,g_1g_2.p_\std)$, and hence $g_1g_2=g^\mathrm{tw}_{[c_1]\ast[c_2]}([\L,\alpha])$.
\end{proof}

\begin{rem}\label{rem:oppsite_Wilson_lines}
As a special case of \cref{prop:multiplicativity}, we have 
\begin{align*}
    g_{[c]^{-1}} = \vw g_{[c]}^{-1}\vw = (g_{[c]}^\mathsf{T})^\ast \cdot s_G.
\end{align*}
In view of \cref{rem:opposite_arc} and $g^\mathrm{tw}_{\fiber_{E}}=s_G$, we have $g_{[c^\mathrm{op}]} = (g_{[c]}^\mathsf{T})^\ast$, which is a relation preserving the positivity. 
\end{rem}

\begin{rem}\label{rem:locality}
\begin{enumerate}
    \item Let $(\Sigma_0,\bM_0) \subset (\Sigma,\bM)$ be a sub-marked surface, by which we mean an embedding $\Sigma_0 \subset \Sigma$ which restrict to $\bM_0 \subset \bM$. Then we have an obvious restriction morphism $\mathrm{res}:\A_{G,\Sigma} \to \A_{G,\Sigma_0}$. Suppose that a framed arc class $[c]:E_1^{\epsilon_1}\to E_2^{\epsilon_2}$ on $\Sigma$ admits a representative contained in $T'\Sigma_0$ and $E_1, E_2$ are also boundary intervals of $\Sigma_0$. Then $[c]$ can be regarded as a morphism in $\Pi_1(T'\Sigma_0;\bB_0^{\pm})$, and the Wilson line $g_{[c]}$ factors through $\A_{G,\Sigma_0}^\times$:
    \begin{equation*}
        \begin{tikzcd}
        \A_{G,\Sigma}^\times \ar[r,"\mathrm{res}"] \ar[d,"g_{[c]}"'] & \A_{G,\Sigma_0}^\times \ar[d,"g_{[c]}"] \\
        G \ar[r,equal] & G.
        \end{tikzcd}
    \end{equation*}
    \item For a subset $\Xi \subset \bB$, let $\A_{G,\Sigma;\Xi}^\times \subset \A_{G,\Sigma}$ denote the open substack consisting of decorated twisted $G$-local systems such that the pair of flat sections associated to each $E \in \Xi$ is generic. We have $\A_{G,\Sigma;\emptyset}^\times=\A_{G,\Sigma}$ and $\A_{G,\Sigma;\bB}^\times=\A_{G,\Sigma}^\times$. Then the Wilson line along a framed arc class $[c]:E_1^{\epsilon_1}\to E_2^{\epsilon_2}$ can be defined on $\A_{G,\Sigma;\Xi}^\times$ for any subset $\Xi$ containing $E_1$ and $E_2$, which makes the following diagram commute:
    \begin{equation*}
        \begin{tikzcd}
        \A_{G,\Sigma}^\times \ar[r,"\mathrm{incl}"] \ar[d,"g_{[c]}"'] & \A_{G,\Sigma;\Xi}^\times \ar[d,"g_{[c]}"] \\
        G \ar[r,equal] & G.
        \end{tikzcd}
    \end{equation*}
\end{enumerate}
\end{rem}

\begin{dfn}
We give some special names to certain Wilson lines, as follows.
\begin{enumerate}
    \item For $E \in \mathbb{B}$, we call the $g_{\sqrt{\fiber_E}^\inn}$ and $g_{\sqrt{\fiber_E}^\out}$ the \emph{boundary Wilson lines} along $E$. 
    \item For $m \in \bM$, let $[c_m]:E_1^-\to E_2^+$ denote the framed arc class as shown in \cref{fig:corner_Wilson_line} around $m$. Then we call $g_{[c_m]}$ the \emph{corner Wilson line} around $m$.
    \item We call the Wilson lines along simple framed arc classes the \emph{simple Wilson lines}. 
\end{enumerate}
\end{dfn}

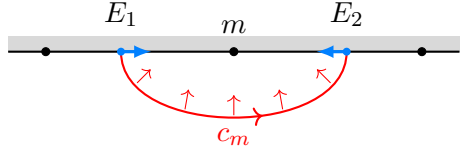
\begin{figure}[ht]
\begin{tikzpicture}
\fill[gray!30] (-3,0) -- (3,0) --(3,0.2) -- (-3,0.2) --cycle;
\draw[thick] (-3,0) -- (3,0);
\filldraw (0,0) circle(1.5pt) node[above=0.3em]{$m$};
\filldraw (-2.5,0) circle(1.5pt);
\filldraw (2.5,0) circle(1.5pt);
\draw[red,thick,->-={0.6}{}] (-1.5,0) to[out=-90,in=-90] node[midway,below]{$c_m$} 
node[pos=0.15,inner sep=0](A){}
node[pos=0.35,inner sep=0](B){} node[pos=0.5,inner sep=0](C){} node[pos=0.65,inner sep=0](D){} node[pos=0.85,inner sep=0](E){}
(1.5,0);
{\color{red}
\tangent{A}{45}{0.3};
\tangent{B}{80}{0.3};
\tangent{C}{90}{0.3};
\tangent{D}{100}{0.3};
\tangent{E}{135}{0.3};
}
\node at (-1.5,0.5) {$E_1$};
\node at (1.5,0.5) {$E_2$};
\fill[myblue] (-1.5,0) circle(1.5pt);
\draw[myblue,very thick,-latex] (-1.5,0) --++(0.4,0);
\fill[myblue] (1.5,0) circle(1.5pt);
\draw[myblue,very thick,-latex] (1.5,0) --++(-0.4,0);
\end{tikzpicture}
    \caption{The corner Wilson line around $m \in \bM$.}
    \label{fig:corner_Wilson_line}
\end{figure}

Recall the angle invariant \eqref{eq:angle_inv}. 
For a marked point $m \in \bM$, let $T_m$ be the unique triangle which contains $m$ as a vertex and the two boundary intervals $E_1,E_2$ incident to $m$ as its sides. Then by \cref{rem:locality} above, we have $\A_{G,\Sigma}^\times \xrightarrow{\mathrm{res}} \A^\times_{G,T_m;\{E_1,E_2\}} \cong \Conf(\A_G,\B_G,\B_G)$. Here the latter isomorphism is obtained by parallel-transport of flags through the interval $\iota_\out(\overline{E_1\cap E_2})$, or equivalently, with respect to the branch-cut on $T_m$ intersecting the opposite side of $m$. Then we get the angle invariant $\mathrm{an}_m:\A_{G,\Sigma}^\times \to U^+$ associated with $m$ as the pull-back of the angle invariant via this morphism. 

\begin{prop}\label{prop:boundary condition}
\begin{enumerate}
    \item For $E \in \bB$, the boundary Wilson line gives the edge invariants $g_{\sqrt{\fiber_E}^{\inn}}=h(\flA^-_E,\flA^+_E)^{-1}$ and $g_{\sqrt{\fiber_E}^{\out}}=(h(\flA^+_E,\flA^-_E)^\ast)^{-1}$, where $\flA^\pm_E$ is the decorated flag assigned to $m_E^\pm$. 
    \item For $m \in \bM$, the corner Wilson line gives the angle invariant $g_{[c_m]}=\mathrm{an}_m\cdot \vw$ twisted by $\vw$.
\end{enumerate}
\end{prop}

\begin{proof}
In view of \cref{rem:locality} (1)(2), the proof reduces to the local computations in \cref{ex:triangle_case} below.  
\end{proof}

\begin{ex}\label{ex:triangle_case}
Let us consider the case $\Sigma=\mathbb{D}_3$, and choose an isomorphism $\A_{G,\mathbb{D}_3} \cong \Conf_3 \A_G$ as described in \cref{ex:polygon}, corresponding to the branch cut shown in \cref{fig:triangle_case}. The boundary intervals are denoted by $E_{1}, E_{2},E_{3}$ as shown there.  
Notice that an arbitrary $G$-orbit in $\Conf_3^\times \A_G$ has a unique representative of the form
\begin{align}\label{eq:standard_config}
    (\flA_1,\flA_2,\flA_3) = ([U^+],\vw^{-1}h_1.[U^+],u_+h_2^{-1}\vw.[U^+])
\end{align}
for $h_1,h_2 \in H$ and $u_+ \in U^+_\ast$. In this parametrization, we have $h_1=h(\flA_2,\flA_1)$, $h_2=h(\flA_1,\flA_3)$ and $u_+=\mathrm{an}_m$.
This parametrization can be extended to a parametrization of triples $(\flA_1,\flA_2,\flA_3)$ such that the pairs $(\flA_1,\flA_2)$ and $(\flA_1,\flA_3)$ are generic by $(h_1,h_2,u_+) \in H \times H \times U^+$. 
Let $\flB_i:=\pi(\flA_i)$ denote the underlying flag for $i=1,2,3$. 
Then we can compute the associated pinnings as
\begin{align*}
    &p_{E_{1}}^-=(\flA_1,\flB_2) = p_\std, \\
    &p_{E_{1}}^+=(\flA_2,\flB_1) = \vw^{-1}h_1.p_\std, \\
    &p_{E_{2}}^+=(\flA_1,\flB_3) = u_+.p_\std, \\
    &p_{E_{2}}^-=(\flA_3,\flB_1) = u_+h_2^{-1}\vw.p_\std
\end{align*}
Using the relation $p_\std=\vw.p_\std^\ast$, we get:
\begin{itemize}
    \item $g_{\sqrt{\fiber_{E_1}}^\out}=w_0(h_1)=(h(\flA_2,\flA_1)^\ast)^{-1}$, which implies that $g_{\sqrt{\fiber_{E_1}}^\inn}=w_0(h_1)s_G=h(\flA_1,\flA_2)^{-1}$.
    \item $g_{[c_m]}=u_+\vw=\mathrm{an}_m\cdot \vw$,
\end{itemize}
from which one can confirm the statements in \cref{prop:boundary condition}. Also note that the framed arc class $[c]:=[c_m]\ast [\sqrt{\fiber_{E_2}}^\out]^{-1}:E_1^- \to E_2^-$ is simple, and that $g_{[c]}=u_+h_2^{-1}s_G$.
\end{ex}

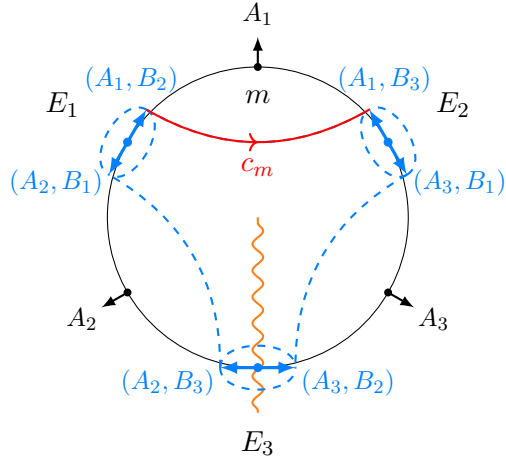
\begin{figure}[ht]
\begin{tikzpicture}
\draw (0,0) circle(2cm);
\node at (0,1.6) {$m$};
\draw [myorange,thick,decorate,decoration={snake,amplitude=2pt,pre length=2pt,post length=0pt}](0,0) -- (0,-2.6);
\foreach \i in {1,2,3}
{
\fill (\i*120-30:2) circle(1.5pt);
\draw[thick,-latex] (\i*120-30:2) --++(\i*120-30:0.4);
\draw[myblue,fill] (\i*120+30:2) circle(1.5pt);
\draw[myblue,very thick,-latex] (\i*120+30:2) --++(\i*120-60:0.5) coordinate(A\i);
\draw[myblue,very thick,-latex] (\i*120+30:2) --++(\i*120+120:0.5)coordinate(B\i); 
} 
\draw[myblue,thick,dashed] (A1) to[out=-30,in=210] (B3); 
\draw[myblue,thick,dashed] (A2) to[out=90,in=-30] (B1); 
\draw[myblue,thick,dashed] (A3) to[out=210,in=90] (B2); 
\draw[myblue,thick,dashed] (A1) to[out=-30,in=-30] (B1); 
\draw[myblue,thick,dashed] (A1) to[out=150,in=150] (B1); 
\draw[myblue,thick,dashed] (A2) to[out=90,in=90] (B2); 
\draw[myblue,thick,dashed] (A2) to[out=-90,in=-90] (B2); 
\draw[myblue,thick,dashed] (A3) to[out=210,in=210] (B3); 
\draw[myblue,thick,dashed] (A3) to[out=30,in=30] (B3); 
\foreach \i in {2,3}
\node[scale=0.9] at (\i*120-30:2.7) {$A_\i$};
\node[scale=0.9] at (90:2.7) {$A_1$};
\draw[myblue] (A1)++(-0.2,0.4) node[scale=0.9]{$(A_1,B_2)$};
\draw[myblue] (B1)++(-0.7,-0.1) node[scale=0.9]{$(A_2,B_1)$};
\draw[myblue] (B3)++(0.2,0.4) node[scale=0.9]{$(A_1,B_3)$};
\draw[myblue] (A3)++(0.7,-0.1) node[scale=0.9]{$(A_3,B_1)$};
\draw[myblue] (A2)++(-0.7,-0.2) node[scale=0.9]{$(A_2,B_3)$};
\draw[myblue] (B2)++(0.7,-0.2) node[scale=0.9]{$(A_3,B_2)$};
\draw[red,thick,->-] (A1) to[out=-30,in=210] node[midway,below=0.2em]{$c_m$} (B3);
\node at (150:3) {$E_{1}$};
\node at (270:3) {$E_{3}$};
\node at (30:3) {$E_{2}$};
\end{tikzpicture}
    \caption{Computation of Wilson lines on a triangle. Here the framed arc classes giving rise to boundary and corner Wilson lines are shown in dashed blue lines.}
    \label{fig:triangle_case}
\end{figure}
By \cref{lem:groupoid_generator}, $\Hom(\Pi_1(T'\Sigma,\bB^\pm),G)$ has a natural structure of affine variety. 

\begin{thm}
The morphism
\begin{align*}
    \A^\times_{G,\Sigma} \to \Hom\left(\Pi_1(T'\Sigma,\bB^\pm),G\right), \quad [\L,\alpha] \mapsto g^\mathrm{tw}_\bullet([\L,\alpha])
\end{align*}
is a closed embedding. The image is characterized by the conditions 
\begin{align*}
    g^\mathrm{tw}_{\fiber_{E^\pm}}=s_G,\quad g^\mathrm{tw}_{\sqrt{\fiber_E}^\inn} \in H\vw^{-1}, \quad g^\mathrm{tw}_{[c_m]} \in U^+
\end{align*}
for $E \in \mathbb{B}$ and $m \in \bM$.
\end{thm}

\begin{proof}
Let us prove that a decorated twisted $G$-local system can be characterized by its (twisted) Wilson lines. Fix a basepoint $E_0^{\epsilon_0}$ for some $E_0 \in \bB$ and $\epsilon \in \{+,-\}$. Then by \cref{lem:Wilson_loop}, the monodromy homomorphism $\pi_1(T'\Sigma,E_0^{\epsilon_0}) \to G$ of the twisted $G$-local system can be reconstructed from the twisted Wilson lines. For a marked point $m \in \bM$, choose an object $E^{\epsilon}$, where $E \in \bB$ is a boundary interval incident to $m$ and $E^{\epsilon}$ is the one pointing to $m$ among $E^\pm$. Then the decoration assigned to $m$ can be read off from the Wilson line along a framed arc class of the form $[c]: E_0^{\epsilon_0} \to E^{\epsilon}$. Since the image is characterized by closed conditions, the assertion is proved. 
\end{proof}

Note that the (twisted) Wilson lines define ring homomorphisms
\begin{align*}
    g_{[c]}^*: \cO(G) \to \cO(\A_{G,\Sigma}^\times), \quad (g_{[c]}^\mathrm{tw})^*: \cO(G) \to \cO(\A_{G,\Sigma}^\times).
\end{align*}
For a matrix coefficient $c_{f,v}^V \in \cO(G)$, we write $c_{f,v}^V(g_{[c]}):=g_{[c]}^*(c_{f,v}^V)$ and $c_{f,v}^V(g_{[c]}^\mathrm{tw}):=(g_{[c]}^\mathrm{tw})^*(c_{f,v}^V)=c_{f,\vw.v}^V(g_{[c]})$. 

Combining the previous theorem with \cref{lem:groupoid_generator}, we get:

\begin{cor}\label{cor:generation_by_Wilson_lines}
The function ring $\cO(\A_{G,\Sigma}^\times)$ is generated by the matrix coefficients of Wilson lines. Moreover if $\Sigma$ has at least two marked points, then $\cO(\A_{G,\Sigma}^\times)$ is generated by matrix coefficients of simple Wilson lines and boundary Wilson lines.
\end{cor}

\begin{rem}
When $\Sigma$ has punctures, the moduli space $\mathrm{Loc}_{G,\Sigma}^\mathrm{un}$ obtained from $\A_{G,\Sigma}^\times$ by forgetting the decorations on punctures is similarly embedded into $\Hom\left(\Pi_1(T'\Sigma,\bB^\pm),G\right)$.
\end{rem}

\medskip
\section{Equality  of cluster and upper cluster algebras}
\label{section4}
\subsection{Generalities on cluster algebras}
 We rapidly recall the necessary definitions of cluster algebras $\mathscr{A}$ and upper cluster algebras $\mathscr{U}$ following the notations of \cite{FG09}. 

Let $\mathcal{F}=\mathbb{C}(A_1, \ldots, A_n)$ be the field of rational functions in $n$ many independent variables $A_1, \ldots, A_n$ with coefficients in $\mathbb{C}$. Fix a positive integer $m\leq n$.
Let $\varepsilon=(\varepsilon_{ij})$ be an $m\times n$ integer matrix such that its $m\times m$ submatrix given by the first $m$ columns is skew-symmetrizable. The set ${\bf A}=\{A_1, \ldots, A_n\}$ is called a {\it cluster chart}. The pair $\mathbf{i}=({\bf A}, \varepsilon)$ is called a {\it seed} in $\mathcal{F}$. Let   $\mathbb{L}_\mathbf{i}=\mathbb{C}[A_1^{\pm 1}, \ldots, A_n^{\pm 1}]\subset \mathcal{F}$ be the ring of Laurent polynomials in $A_1,\ldots, A_n$.

For $1\leq k \leq m$, the {\em seed mutation} of $\mathbf{i}$ in the direction $k$ produces a new seed $\mathbf{i}_k := \left( \{A_1', \ldots, A_n'\}, \varepsilon'\right)$ as follows:
\begin{align*}
A_i'&= \begin{cases} A_i & \mbox{if $i\neq k$},\\
A_k^{-1}\left(\prod_{j | \varepsilon_{kj}>0} A_{j}^{\varepsilon_{kj}}+ \prod_{j | \varepsilon_{kj}<0} A_{j}^{-\varepsilon_{kj}}\right) & \mbox{otherwise},
\end{cases}\\
\varepsilon_{ij}'&= \begin{cases}
-\varepsilon_{ij}, &\mbox{if } k\in \{i,j\},\\
\varepsilon_{ij}, &\mbox{if } \varepsilon_{ik}\varepsilon_{kj}<0, ~~ k\notin\{i,j\},\\
\varepsilon_{ij}+|\varepsilon_{ik}|\varepsilon_{kj}, &\mbox{if } \varepsilon_{ik}\varepsilon_{kj}>0, ~~ k\notin\{i,j\}.
\end{cases}
\end{align*}
We say that a seed $\mathbf{i}'$ is mutation equivalent to $\mathbf{i}$, and denote by $\mathbf{i}'\sim \mathbf{i}$, if $\mathbf{i}'$ can be obtained from $\mathbf{i}$ by a sequence of seed mutations. The variables $A_i'$ in each $\mathbf{i}'$ are called cluster variables. Note that the cluster variables $A_{m+1}, \ldots, A_{n}$ are invariant under mutations and  are called frozen variables.  

Let us start with an initial seed $\mathbf{i}$. Following \cite{BFZ}, there are three versions of commutative algebras associated with $\mathbf{i}$. 
\begin{dfn} The {\it upper bound} associated with a seed $\mathbf{i}$ is a ring
\[\overline{\mathscr{U}}(\mathbf{i}) = \mathbb{L}_\mathbf{i}\cap \mathbb{L}_{\mathbf{i}_1}\cap \ldots \cap \mathbb{L}_{\mathbf{i}_m}.\]
The {\em upper cluster algebra} associated with $\mathbf{i}$ is the intersection of Laurent polynomials for all seeds $\mathbf{i}'$ that are mutation equivalent to $\mathbf{i}$:
\[
\mathscr{U}(\mathbf{i}) = \bigcap_{\mathbf{i}'\sim \mathbf{i}} \mathbb{L}_{\mathbf{i}'}.
\]
The {\em cluster algebra} $\mathscr{A}(\mathbf{i})$ is the unital $\mathbb{C}$-subalgebra of $\mathcal{F}$ generated by the cluster variables and the inverses $A_{m+1}^{-1}, \ldots, A_n^{-1}$ of the frozen variables. 
\end{dfn}
We frequently write $\mathscr{U}$ and $\mathscr{A}$ instead of $\mathscr{U}(\mathbf{i})$ and $\mathscr{A}(\mathbf{i})$ when there is no confusion.

We have the following inclusion relations
\[
\mathscr{A}(\mathbf{i}) \subseteq \mathscr{U}(\mathbf{i})\subseteq \overline{\mathscr{U}}(\mathbf{i}),
\]
where the first inclusion is a consequence of the Laurent phenomenon of the cluster variables \cite{FZ02}, and the second inclusion is by definition. 

The following result of \cite[Corollary 1.9]{BFZ} will be useful in this paper.

\begin{prop}
\label{upperbound=upper}
If the exchange matrix $\varepsilon$ in $\mathbf{i}$ has full rank $m$, then the upper cluster algebra $\mathscr{U}(\mathbf{i})$ coincides with the upper bound $\overline{\mathscr{U}}(\mathbf{i})$.
\end{prop}

\subsection{The upper cluster algebra coincides with the function ring $\cO(\A_{G,\Sigma}^\times)$}\label{subsec:U=O}

Let $\mathcal{T}$ be an ideal triangulation of the marked surface $\Sigma$. Let $e(\mathcal{T})$ and $t(\mathcal{T})$ denote the sets of edges and triangles of $\mathcal{T}$. Note that
\[
|t(\mathcal{T})|=-2\chi(\Sigma)
+|\bM|, \quad \quad |e(\mathcal{T})| = -3\chi(\Sigma)
+2|\bM|.
\] 
For each triangle $T\in t(\mathcal{T})$, we choose a vertex $v_T$ of $T$ together with a reduced decomposition $\mathbf{s}_T$ of $w_0$. The data
\[
\bm{\mathcal T}:=({\mathcal T}, \{v_T\}, \{\mathbf{s}_T\})
\]
is called a {\em decorated ideal triangulation}. Following the construction of \cite{GS19}, every $\bm{\mathcal{T}}$ gives rise to a seed $\mathbf{i}_{\bm{\mathcal{T}}}$ in the field $\mathcal{K}(\mathcal{A}_{G, \Sigma})$ of rational functions (recall \cref{rem:integral}). 
The seeds obtained from different $\bm{\mathcal{T}}$ are mutation equivalent, and thus give rise to the canonical cluster algebras 
\[
\mathscr{A}_{\mathfrak{g},\Sigma} \subseteq \mathscr{U}_{\mathfrak{g},\Sigma} \subset \mathcal{K}(\mathcal{A}_{G, \Sigma}).
\]
The exchange matrix $\varepsilon$ in each $\mathbf{i}_{\bm{\mathcal{T}}}$ is an $m \times n$ matrix, where
\[
n= |e(\mathcal{T})|r + |t(\mathcal{T})|(l(w_0)-r), \quad \quad n-m = |\bM|r
\]
with $r:=\mathop{\mathrm{rank}} G=|S|$. 
Briefly speaking, the construction goes as follows: 
\begin{itemize}
    \item the ideal triangulation $\mathcal{T}$ gives rise to a decomposition of the moduli space $\A_{G,\Sigma}$ into the pieces $\A_{G,T}$ for $T \in t(\mathcal{T})$; 
    \item the vertex $v_T$ determines an isomorphism $f_{v_T}:\A_{G,T} \xrightarrow{\sim} \Conf_3 \A_G$ as in \cref{ex:polygon}; 
    \item the pull-back of the coordinate system on $\Conf_3 \A_G$ associated with the reduced word $\bs_T$ gives a coordinate system on $\A_{G,T}$ for all $T \in t(\mathcal{T})$, which glue together to give a coordinate system on $\A_{G,\Sigma}$, \emph{i.e.}, 
    an open dense embedding 
    \[
    i_{\bm{\mathcal{T}}}: ~ (\mathbb{C}^\times)^n \longrightarrow \mathcal{A}^\times_{G, \Sigma}.
    \]
\end{itemize}
The transition maps $i_{\bm{\mathcal{T}}'}^{-1}\circ i_{\bm{\mathcal{T}}}$ for different ideal triangulations are given by sequences of seed mutations. For details, we refer the reader to \cite{GS19}. In \cref{subsec:cluster_quadrilateral}, we will review the cluster $K_2$-structure on $\A_{G,Q}^\times$ when $\Sigma=Q$ is a quadrilateral, from which the coordinates on triangles and edges can be also read. 

We are going to prove:
\begin{thm} \label{co=cl}
For any marked surface $\Sigma$ satisfying the assumptions in \cref{subsec:marked_surface}, the upper cluster algebra $\mathscr{U}_{\mathfrak{g},\Sigma}$ coincides with the function ring $\cO(\A_{G,\Sigma}^\times)$.
\end{thm}

The proof goes through the same steps as the proof of \cite[Theorem 1.1]{S21}.

\begin{lem}
\label{full.rank.lem}
The exchange matrix $\varepsilon$ in each $\mathbf{i}_{\bm{\mathcal{T}}}$ is of full rank.  
\end{lem}
\begin{proof}
Let $G'$ be the associated adjoint group of $G$. Fock and Goncharov \cite{FG03} also considered a moduli space $\mathcal{X}_{G', \Sigma}$, called the moduli space of framed $G'$-local systems. 
Similarly to $\A_{G,\Sigma}^\times$, after imposing a generic condition on each boundary interval of $\Sigma$, we get an open subspace $\mathcal{X}_{G', \Sigma}^\times \subset \mathcal{X}_{G', \Sigma}$. Every $\bm{\mathcal{T}}$ equips $\mathcal{X}_{G', \Sigma}^\times$ with a collection of cluster Poisson coordinates $\{X_i~|~ 1\leq i\leq m\}$. Let $\{A_j~|~ 1\leq j \leq n\}$ be the cluster coordinates on $\mathcal{A}_{G, \Sigma}^\times$ associated with $\bm{\mathcal{T}}$. Following \cite{FG03} and \cite{GS19}, there is a natural map 
\[
p: \mathcal{A}^\times_{G, \Sigma} \rightarrow \mathcal{X}^\times_{G',\Sigma}
\]
such that 
\begin{equation}
\label{asn}
p^* X_i =\prod_{j=1}^n A_j^{\varepsilon_{ij}}.
\end{equation}
When $\Sigma$ has no punctures, the map $p$ is surjective. Therefore the ring homomorphism $p^*$ in \eqref{asn} is injective, and $\varepsilon$ is of full rank. 
\end{proof}

\begin{lem}
\label{disk.case}
Theorem \ref{co=cl} holds when $\Sigma=\mathbb{D}_k$ is a disk with $k$ marked points.
\end{lem}
\begin{proof} Let $H$ be the Cartan subgroup of $G$. If $k=3$ or $4$, then 
\[
\mathcal{A}^\times_{G, \mathbb{D}_3} \stackrel{\sim}{=} G^{e, w_0} \times H, \quad \quad \mathcal{A}^\times_{G, \mathbb{D}_4} \stackrel{\sim}{=} G^{w_0, w_0} \times H^2,
\]
where $G^{e, w_0}$ and $G^{w_0, w_0}$ are double Bruhat cells. After a careful comparison between the cluster structures on double Bruhat cells and on $\mathcal{A}^\times_{G, \Sigma}$, the lemma is a direct consequence of \cite[Theorem 2.10]{BFZ}. 
In general, we have
\[\mathcal{A}_{G,\mathbb{D}_k}^\times \stackrel{\sim}{=}{\rm Conf}_{w_0^{k-2}}^e(\mathcal{A}_{\rm sc})\times H^{k-2},
\]
where ${\rm Conf}_{w_0^{k-2}}^e(\mathcal{A}_{\rm sc})$ is a double Bott-Samelson variety in \cite{SW21}. The lemma follows from \cite[Theorem 1.1]{SW21}.
\end{proof}

Fix a decorated ideal triangulation $\bm{\mathcal{T}}$. Let $E \in e({\mathcal T})$ be a diagonal (i.e., an internal edge). Let $\mathbf{i}_{\bm{\mathcal{T}}, E}$ be the seed obtained from $\mathbf{i}_{\bm{\mathcal{T}}}$ by freezing all the mutable vertices that are placed on the diagonals different from $E$. Let $\mathcal{A}_{G, \Sigma}^{\mathcal{T}, E}\subset \mathcal{A}_{G, \Sigma}^\times$ be the open subspace such that for every $E'\neq E$ in $e(\mathcal{T})$, its associated pair of decorated flags is generic. 

\begin{lem}[Proved in \cref{app:amalgamation}] \label{local.iso.q}
The coordinate ring $\mathcal{O}(\mathcal{A}_{G, \Sigma}^{\mathcal{T}, E})$ coincides with the upper cluster algebra $\mathscr{U}(\mathbf{i}_{\bm{\mathcal{T}}, E})$.
\end{lem}

Now let us go through all the diagonals $E$ of ${\mathcal T}$ and set 
\[
\widetilde{\mathcal{A}}_{G, \Sigma}^\times:= \bigcup_E \mathcal{A}_{G, \Sigma}^{\mathcal{T}, E} \subset \mathcal{A}_{G, \Sigma}^\times
\]
\begin{lem}
\label{codim.lem}
We have $\mathcal{O}(\widetilde{\mathcal{A}}_{G, \Sigma}^\times)=\mathcal{O}({\mathcal{A}}_{G, \Sigma}^\times)$.
\end{lem}
\begin{proof}
For a pair $E_1,E_2$ of diagonals of ${\mathcal T}$, let $\mathcal{A}_{G, \Sigma}^{\mathcal{T},E_1,E_2}\subset \mathcal{A}_{G, \Sigma}^\times$ be the subspace such that the pairs of the decorated flags associated with $E_1$ and $E_2$ are not generic. 
The complement of $\widetilde{\mathcal{A}}_{G, \Sigma}^\times$ in ${\mathcal{A}}_{G, \Sigma}^\times$ is $\bigcup_{E_1,E_2}\mathcal{A}_{G, \Sigma}^{\mathcal{T},E_1,E_2}$, which is of codimension $\geq 2$. See \cref{app:codimension} for a detailed computation of the codimension. Therefore $\mathcal{O}(\widetilde{\mathcal{A}}_{G, \Sigma}^\times)=\mathcal{O}({\mathcal{A}}_{G, \Sigma}^\times)$.
\end{proof}

\begin{proof}[Proof of Theorem \ref{co=cl}]
Combining \cref{codim.lem,local.iso.q,full.rank.lem}, and \cref{upperbound=upper}, we have
\[
\mathcal{O}({\mathcal{A}}_{G, \Sigma}^\times)=\mathcal{O}(\widetilde{\mathcal A}_{G, \Sigma}^\times) = \bigcap_{E} \mathcal{O}(\mathcal{A}_{G,\Sigma}^{\mathcal{T},E})=\bigcap_{E} \mathscr{U}(\mathbf{i}_{\bm{\mathcal{T}},E})=\bigcap_{E} \overline{\mathscr{U}}(\mathbf{i}_{\bm{\mathcal{T}},E})= \overline{\mathscr{U}}(\mathbf{i}_{\bm{\mathcal{T}}})= \mathscr{U}_{\mathfrak{g}, \Sigma}.
\]
\end{proof}

Combining with \cref{cor:generation_by_Wilson_lines}, we get: 

\begin{cor}\label{cor:generation_UCA_Wilson_lines}
The upper cluster algebra $\mathscr{U}_{\mathfrak{g},\Sigma}$ is generated by the matrix coefficients of Wilson lines. Moreover if $\Sigma$ has at least two marked points, then $\mathscr{U}_{\mathfrak{g},\Sigma}$ is generated by matrix coefficients of simple Wilson lines and boundary Wilson lines.
\end{cor}

\begin{rem}
As stated in \cite[Proposition 3.17 (i)]{GS18}, the upper cluster algebra no longer coincides with $\cO(\A_{G,\Sigma}^\times)$ when $\Sigma$ has punctures. 
\end{rem}

\subsection{Cluster $K_2$-coordinates on $\Conf_4^\times \A_G$}\label{subsec:cluster_quadrilateral}
For the computations in the next subsection, let us recall the cluster $K_2$-coordinates on the configuration space $\Conf_4^\times \A_G=\Conf_{w_0}^{w_0}\A_G$ from \cite[Section 7]{GS19}, which can be regarded as an open subspace of the moduli space $\A_{G,Q}$ for a quadrilateral $Q$. Here are notations from \cite[Section 5]{GS19}:
\begin{itemize}
    \item Given $w \in W$, we set
    \begin{align*}
        S(w)&:=\{ s \in S \mid w\alpha_s^\vee <0 \}, \\
        H(w)&:=\left\{\prod_{s \in S(w)}\alpha_s^\vee(b_s) \mid b_s \in \bG_m\right\}.
    \end{align*}
    Then $S(w)=S \ \Longleftrightarrow\  H(w) = H \ \Longleftrightarrow\  w=w_0$. 
    \item Given a reduced word $\bs=(s_1,\dots,s_l)$ of $w$, we get a sequence of distinct coroots 
    \begin{align*}
        \beta_k^\bs:=r_{s_l}\dots r_{s_{k+1}}\alpha_{s_k}^\vee, \quad k=1,\dots,l.
    \end{align*}
    They are precisely the positive coroots $\alpha^\vee$ such that $w\alpha^\vee$ are negative.
    \item Recall that we have 
    \begin{align*}
        \cO(\Conf_2 \A_G)\cong \bigoplus_{\lambda \in X^\ast(H)_+} \big(V(\lambda)\otimes V(\lambda^\ast)\big)^G,
    \end{align*}
    and consider the function $\Delta_s \in \cO(\Conf_2 \A_G)$ for $s \in S$ such that $\Delta_s \in \big(V(\varpi_s)\otimes V(\varpi_s^\ast)\big)^G$ and $\Delta_s([U^+],[U^-])=1$. We have the relation 
    \begin{align*}
        \Delta_{u\varpi_s,v\varpi_s}(g)=\Delta_s(g\overline{v}.[U^+],\overline{u}.[U^-])
    \end{align*}
    for $g \in G$ and $u,v \in W$.
\end{itemize}

The configuration space $\Conf_4^\times \A_G$ parametrizes the $G$-orbits of quadruples $(\flA^L,\flA_L,\flA_R,\flA^R)$ of decorated flags such that the cyclically consecutive pairs $(\flA^L,\flA_L)$, $(\flA_L,\flA_R)$, $(\flA_R,\flA^R)$, $(\flA^R,\flA^L)$ are generic. Such a quadruple $(\flA^L,\flA_L,\flA_R,\flA^R)$ is illustrated as 
\begin{align*}
    \begin{tikzpicture}[>=latex]
    \draw[thick] (-1,1) --node[midway,left]{$w_0$} (-1,-1); 
    \draw[thick] (1,1) --node[midway,right]{$w_0$} (1,-1);
    \draw[thick] (-1,1) --node[midway,above]{$w_0$} (1,1);
    \draw[thick] (-1,-1) --node[midway,below]{$w_0$} (1,-1);
    {\color{red}
    \filldraw(-1,1) circle(2pt);
    \filldraw(1,1) circle(2pt);
    \filldraw(-1,-1) circle(2pt);
    \filldraw(1,-1) circle(2pt);
    }
    \draw(-1,1) node[above left]{$\flA^L$};
    \draw(1,1) node[above right]{$\flA^R$};
    \draw(-1,-1) node[below left]{$\flA_L$};
    \draw(1,-1) node[below right]{$\flA_R$};
    \end{tikzpicture}
\end{align*}

\begin{lem}[cf.~{\cite[Lemma-Definition 5.3]{GS19}}]\label{lem:decorated_chain}
Let $(\flA_l,\flA_r)$ be a generic pair of decorated flags, and $\bs=(s_1,\dots,s_N)$ a reduced word of $w_0$. Then there exists a unique chain $\flA_l=\flA_0\xleftarrow{s_1} \flA_1\xleftarrow{s_2}\dots\xleftarrow{s_N} \flA_N=\flA_r$ of decorated flags such that
\begin{itemize}
    \item $w(\flA_k,\flA_{k-1})=r_{s_k}$ for $k=1,\dots,N$, and
    \item $h(\flA_k,\flA_{k-1})=\alpha_{s_k}^\vee(c_k)$. 
\end{itemize}
Here $c_k \in \bG_m$ is given by
\begin{align*}
    c_k:=\begin{cases}
    h(\flA_r,\flA_l)^{\varpi_t} & \mbox{if $\beta_k^\bs=\alpha_t^\vee$ is simple}, \\
    1 & \mbox{otherwise}.
    \end{cases}
\end{align*}
\end{lem}

\begin{proof}
When $(\flA_r,\flA_l)=(h.[U^+],\vw.[U^+])$ sits in the standard configuration, the intermediate flags are given by
\begin{align*}
    \flA_k:=\overline{r}_{s_N}\dots \overline{r}_{s_{k+1}} h_k.[U^+],
\end{align*}
where $h_k:=u_k\left(\prod_{s \in S(w_0 u_k)}\alpha_s^\vee(h(\flA_r,\flA_l)^{\varpi_s})\right)$ with $u_k:=r_{s_{k+1}}\dots r_{s_N}$. 
Indeed, we have
\begin{align*}
    [\flA_k,\flA_{k-1}] = [h_k.[U^+], \overline{r}_{s_k}h_{k-1}.[U^+]] = [r_{s_k}(h_{k-1}^{-1})h_{k}.[U^+], \overline{r}_{s_k}.[U^+]]
\end{align*}
and the unique solution of $r_{s_k}(h_{k-1}^{-1})h_{k}=\alpha_{s_k}^\vee(c_k)$ for $k=1,\dots,N$ is given by
\begin{align*}
    h_k&=\prod_{j=1}^k r_{s_k}\dots r_{s_{j+1}}\alpha_{s_j}^\vee(c_j).
\end{align*}
Then applying $u_k^{-1}$ to both sides, we get $u_k^{-1}(h_k)=\prod_{j=1}^k \beta_j^{\bs}(c_j)=\prod_{t \in S(w_0 u_k)} \alpha_t^\vee(h^{\varpi_t})$, as desired. The intermediate decorated flags for a pair $g.(\flA_r,\flA_l)$ are given by the translates $g.\flA_k$ for $g \in G$. 
\end{proof}
Now we choose a double reduced word $\bs$ of $(w_0,w_0) \in W \times W$. The subword of $\bs$ consisting of letters in $S$ gives a reduced word $\bs_{\bullet}=(s_1,\dots,s_N)$ of $w_0$; similarly, the $\overline{S}$-part gives another reduced word $\bs^{\bullet}=(s^1,\dots,s^N)$. 
Given a quadruple $(\flA^L,\flA_L,\flA_R,\flA^R)$ whose orbit lies in $\Conf_4^\times \A_G$, we apply \cref{lem:decorated_chain} to the pair $(\flA^R,\flA^L)$ with the word $\bs^\bullet$, and to the pair $(\flA_L,\flA_R)$ with the word $\bs_\bullet^\ast$. Then we get the following chains of decorated flags:
\begin{align}
    &\flA^L=\flA^N\xrightarrow{s^N} \flA^{N-1}\xrightarrow{s^{N-1}}\dots\xrightarrow{s^1} \flA^0=\flA^R, \qquad w(\flA^k,\flA^{k-1})=r_{s^k}, \quad h(\flA^k,\flA^{k-1})=\alpha_{s^k}^\vee(c^k),\label{eq:chain_upper} \\
    &\flA_L=\flA_0\xleftarrow{s_1^\ast} \flA_1\xleftarrow{s_2^\ast} \dots\xleftarrow{s_N^\ast} \flA_N=\flA_R, \qquad w(\flA_k,\flA_{k-1})=r_{s_k^\ast}, \quad h(\flA_k,\flA_{k-1})=\alpha_{s_k^\ast}^\vee(c_k),\label{eq:chain_lower}
\end{align}
where $c^k$ and $c_k$ are given by
\begin{align*}
    c^k:=\begin{cases}
    h(\flA^L,\flA^R)^{\varpi_t} & \mbox{if $\beta_k^{\bs^\bullet}=\alpha_t^\vee$ is simple}, \\
    1 & \mbox{otherwise},
    \end{cases}
\end{align*}
and 
\begin{align*}
    c_k:=\begin{cases}
    h(\flA_R,\flA_L)^{\varpi_u} & \mbox{if $\beta_k^{\bs_\bullet^\ast}=\alpha_u^\vee$ is simple}, \\
    1 & \mbox{otherwise},
    \end{cases}
\end{align*}
respectively. Using these chains, the double reduced word $\bs$ gives rise to a decomposition of the configuration $[\flA^L,\flA_L,\flA_R,\flA^R]$ into elementary configurations, as explained by the following example.
\begin{ex}[Type $A_2$]
The double reduced word $\bs=(1,\overline{1},\overline{2},2,\overline{1},1)$ gives rise to the decomposition below. The locations where the coroot $\beta_k^{\bs^\bullet}$ or $\beta_k^{\bs_\bullet}$ becomes simple are shown in green. 
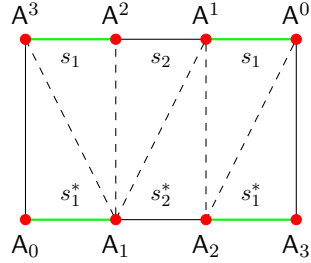
\begin{align*}
\begin{tikzpicture}[scale=1.2]
\draw(0,-1) -- (0,1) -- (3,1) -- (3,-1) --cycle;
\draw[dashed] (0,1) -- (1,-1) -- (1,1);
\draw[dashed] (1,-1) -- (2,1) -- (2,-1) -- (3,1);
\draw[green,thick] (0,1) -- (1,1);
\draw[green,thick] (2,1) -- (3,1);
\draw[green,thick] (0,-1) -- (1,-1);
\draw[green,thick] (2,-1) -- (3,-1);
\foreach \i in{0,1,2,3}
{
\filldraw[red] (\i,1) circle(1.6pt);
\filldraw[red] (\i,-1) circle(1.6pt);
}
\node[above=0.2em,scale=0.8] at (0.5,-1) {$s_1^\ast$}; 
\node[above=0.2em,scale=0.8] at (1.5,-1) {$s_2^\ast$};
\node[above=0.2em,scale=0.8] at (2.5,-1) {$s_1^\ast$};
\node[below=0.2em,scale=0.8] at (0.5,1) {$s_1$};
\node[below=0.2em,scale=0.8] at (1.5,1) {$s_2$};
\node[below=0.2em,scale=0.8] at (2.5,1) {$s_1$};
\node[below=0.2em,scale=0.9] at (0,-1) {$\flA_0$};
\node[below=0.2em,scale=0.9] at (1,-1) {$\flA_1$};
\node[below=0.2em,scale=0.9] at (2,-1) {$\flA_2$};
\node[below=0.2em,scale=0.9] at (3,-1) {$\flA_3$};
\node[above=0.2em,scale=0.9] at (0,1) {$\flA^3$};
\node[above=0.2em,scale=0.9] at (1,1) {$\flA^2$};
\node[above=0.2em,scale=0.9] at (2,1) {$\flA^1$};
\node[above=0.2em,scale=0.9] at (3,1) {$\flA^0$};
\end{tikzpicture}
\end{align*}
\end{ex}

Then for each pair $(\flA^k,\flA_l)$ connected by a dashed line or a vertical solid line, we consider the functions $\Delta_s(\flA^k,\flA_l)$ for $s \in S$. These functions are not distinct: for example if we have a triple
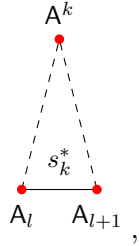
\begin{align*}
\begin{tikzpicture}
\draw(0,-1) -- (1,-1);
\draw[dashed](0,-1) -- (0.5,1) -- (1,-1);
\node[above=0.2em,scale=0.9] at (0.5,1) {$\flA^k$};
\node[below=0.2em,scale=0.9] at (0,-1) {$\flA_l$};
\node[below=0.2em,scale=0.9] at (1,-1) {$\flA_{l+1}$};
\node[above=0.2em,scale=0.9] at (0.5,-1) {$s_k^\ast$};
\filldraw[red] (0,-1) circle(1.6pt);
\filldraw[red] (0.5,1) circle(1.6pt);
\filldraw[red] (1,-1) circle(1.6pt);
\end{tikzpicture},
\end{align*}
we have $\Delta_s(\flA^k,\flA_l) = \Delta_s(\flA^k,\flA_{l+1})$ for $s \neq s_k^\ast$. We have a similar relation for a down-pointing triangle. See \cite[Lemma 7.13]{GS19} for details. Collectively, we get $|S|+\#(\text{dashed lines})+1$ distinct functions among $\Delta_s(\flA^k,\flA_l)$'s. They can be assigned to the vertices of the quiver $\bJ(\bs)$ in \cite[Definition 7.5]{GS19}. We also have additional functions 
$h(\flA^{k},\flA^{k-1})^{\varpi_t}=h(\flA^L,\flA^R)^{\varpi_t}$ if $\beta_k^{\bs^\bullet}=\alpha_t^\vee$ is simple, and $h(\flA_l,\flA_{l-1})^{\varpi_u^\ast}=h(\flA_R,\flA_L)^{\varpi_u^\ast}$ if $\beta_k^{\bs_\bullet^\ast}=\alpha_u^\vee$ is simple. 
They are assigned to the green lines, and supply the coordinates on the remaining ``extra'' vertices in the quiver $\overline{\bJ}(\bs)$ in \cite[Definition 7.5]{GS19}. The quiver $\overline{\bJ}(\bs)$ together with the coordinates assigned to its vertices form a cluster $K_2$-seed in the field $\mathcal{K}(\Conf_4^\times \A_G)$. 

We remark that these cluster coordinates are regular functions on a larger space $\Conf_\inn^\out \A_G$ consisting of the $G$-orbits of quadruples $(\flA^L,\flA_L,\flA_R,\flA^R)$ such that the top and bottom pairs $(\flA^L,\flA^R)$ and $(\flA_L,\flA_R)$ are generic.

\subsection{Generalized minors of simple Wilson lines are cluster variables}\label{subsec:computation_minors}
Let $[c]:E_1^- \to E_2^-$ be a simple framed arc class and fix its band neighborhood $B_c$. The band $B_c$ can be regarded as a marked surface (\emph{i.e.}, a disk with four marked points $m^L,m_L,m_R,m^R$) as shown in \cref{fig:band_nbd}. 
By \cref{rem:locality} (1)(2), the Wilson line $g_{[c]}$ can be computed on the moduli space $\A_{G,B_c;\{E_1,E_2\}}^\times\cong \Conf_\inn^\out \A_G$, where the identification is determined by the branch cut and the correspondence of flags as shown in \cref{fig:simple_configuration}. For a weight $\lambda=\sum_{s \in S}a_s \varpi_s$, we write $[\lambda]_+:=\sum_{s \in S}[a_s]_+\varpi_s$, where $[a_s]_+:= \max\{0, a_s\}$.

\begin{prop}\label{prop:simple_minors}
Each generalized minor of a simple Wilson line $g_{[c]}$ is a single cluster variable in $\mathscr{A}_{\mathfrak{g},\Sigma}$ multiplied by inverses of several frozen variables. Specifically, we have
\begin{align}\label{eq:minor_Wilson_line}
    \Delta_{u_{>l}\varpi_s,v_{>k}\varpi_s}(g_{[c]}) = \frac{\Delta_s(\flA^k,\flA_l)}{h(\flA_R,\flA_L)^{[u_{>l}\varpi_s]^\ast_+}h(\flA^L,\flA^R)^{[v_{>k}\varpi_s]_+}}
\end{align}
for all $k,l=0,\dots,N$ and $s \in S$, where  $u_{>l}:=r_{s_N}\dots r_{s_{l+1}}$ and $v_{>k}:=r_{s^N}\dots r_{s^{k+1}}$
for any double reduced word $\bs$ of $(w_0,w_0) \in W \times W$.
\end{prop}

\begin{figure}[htbp]
\begin{tikzpicture}
\filldraw[fill=pink!30,draw=red,dashed] (0,0.7) -- (4,0.7) -- (4,-0.7) -- (0,-0.7) --cycle;
\fill[gray!20] (0,1.5) -- (-0.2,1.5) -- (-0.2,-1.5) -- (0,-1.5) --cycle;
\fill[gray!20] (4,1.5) -- (4+0.2,1.5) -- (4+0.2,-1.5) -- (4,-1.5) --cycle;
\draw[thick] (0,1.5) -- (0,-1.5);
\draw[thick] (4,-1.5) -- (4,1.5);
\filldraw(0,0.7) circle(1.5pt); 
\filldraw(0,-0.7) circle(1.5pt);
\filldraw(4,0.7) circle(1.5pt);
\filldraw(4,-0.7) circle(1.5pt);
\draw[red,thick,->-={0.45}{}] (0,0) --node[midway,below=0.2em]{$c$} (4,0);
{\color{red}
\foreach \i in {90,30,-30}
\tangent{4,0}{\i}{0.3};
\draw[thick,dotted] (4,-0.3) arc(-90:90:0.3);
\foreach \x in {0.5,1,1.5,2,2.5,3,3.5}
\tangent{\x,0}{90}{0.3};
}
\fill[myblue] (0,0) circle(1.5pt);
\draw[myblue,very thick,-latex] (0,0)--++(0,0.4);
\fill[myblue] (4,0) circle(1.5pt);
\draw[myblue,very thick,-latex] (4,0)--++(0,-0.4);
\draw[myorange,thick,decorate,decoration={snake,amplitude=2pt,pre length=2pt,post length=0pt}] (2.5,-0.3) --++(0,-1.5);
\draw[thick,-latex](0,-0.7) --++(-0.4,0) node[left,anchor=east]{$\vw^{-1}h'.[U^+]=\flA_R$};
\draw[thick,-latex](0,0.7) --++(-0.4,0) node[left,anchor=east]{$[U^+]=\flA_L$};
\draw[thick,-latex](4,-0.7) --++(0.4,0) node[right,anchor=west]{$\flA^R=g\vw.[U^+]$};
\draw[thick,-latex](4,0.7) --++(0.4,0) node[right,anchor=west]{$\flA^L=gh.[U^+]$};
\end{tikzpicture}
    \caption{Computation of the Wilson line $g_{[c]}$ on $\A_{G,B_c}^\times \cong \Conf_4^\times \A_G$.}
    \label{fig:simple_configuration}
\end{figure}
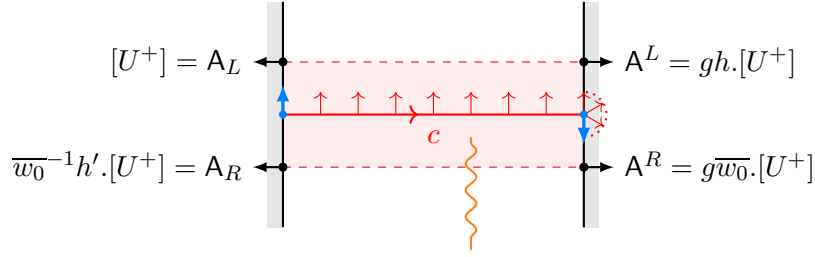

\begin{proof}
Any $G$-orbit $[\flA^L,\flA_L,\flA_R,\flA^R]$ contains a unique representative of the form
\begin{align*}
    (\flA^L,\flA_L,\flA_R,\flA^R)= (gh.[U^+], [U^+],\vw^{-1}h'.[U^+],g\vw.[U^+])
\end{align*}
for $h,h' \in H$ and $g \in G^{w_0,w_0}$. In other words, we have an isomorphism $H \times H\times G^{w_0,w_0} \xrightarrow{\sim} \Conf_4^\times \A_G$. 
In this parametrization, we have $h(\flA^L,\flA^R)=h$ and $h(\flA_R,\flA_L)=h'$. Moreover, $(\flA_L,\pi(\flA_R))=p_\std$ and $(\flA^R,\pi(\flA^L))=g\vw.p_\std=gs_G.p_\std^\ast$; hence $g_{[c]}=gs_G$. 

Let $\bs$ be a double reduced word of $(w_0,w_0) \in W \times W$. 
Then by the proof of \cref{lem:decorated_chain}, we explicitly get the associated chains $(\flA^k)_k$, $(\flA_l)_l$ of decorated flags in \eqref{eq:chain_upper} and \eqref{eq:chain_lower} as 
\begin{align*}
    \flA^k= g\overline{r}_{s^N}\dots \overline{r}_{s^{k+1}} h^k.[U^+]
\end{align*}
with $h^k:=w^k\left(\prod_{s \in S(w_0 w^k)}\alpha_s^\vee(h^{\varpi_s})\right)$ and $w^k:=r_{s^{k+1}}\cdots r_{s^{N}}$, and 
\begin{align*}
    \flA_l= \vw^{-1}\overline{r}_{s_N^\ast}\dots \overline{r}_{s_{l+1}^\ast} h_l.[U^+]
\end{align*}
with $h_l:=w_{l}\left(\prod_{s \in S(w_0 w_{l})}\alpha_s^\vee({h'}^{\varpi_s})\right)$ and $w_l:=r_{s_{l+1}^\ast} \cdots r_{s_N^\ast}$. Then we get
\begin{align*}
    \Delta_s(\flA^k,\flA_l) &= \Delta_s\left(g\overline{r}_{s^N}\dots \overline{r}_{s^{k+1}}.[U^+], \vw^{-1}\overline{r}_{s_N^\ast}\dots \overline{r}_{s_{l+1}^\ast}.[U^+]\right)\cdot (h^k h_l^\ast)^{\varpi_s} \\
    &= \Delta_s\left(g\overline{r}_{s^N}\cdots \overline{r}_{s^{k+1}}.[U^+], s_G\overline{r}_{s_N}\cdots \overline{r}_{s_{l+1}}.[U^-]\right)\cdot (h^k h_l^\ast)^{\varpi_s} \\
    &=\Delta_{r_{s_N}\cdots r_{s_{l+1}}.\varpi_s,r_{s^N}\cdots r_{s^{k+1}}\varpi_s}(gs_G)\cdot (h^k h_l^\ast)^{\varpi_s}.
\end{align*}
Thus we get
\begin{align*}
    \Delta_{u_{>l}\varpi_s,v_{>k}\varpi_s}(g_{[c]}) = (h^k h_l^\ast)^{-\varpi_s} \cdot \Delta_s(\flA^k,\flA_l).
\end{align*}
The frozen variables are computed as follows:
\begin{align*}
    &(h^k)^{\varpi_s}=\prod_{t \in S(w_0 w^k)}\alpha_t^\vee(h^{\varpi_t})^{(w^k)^{-1}\varpi_s} = \prod_{t \in S}\alpha_t^\vee({h}^{\varpi_t})^{[(w^k)^{-1}\varpi_s]_+} = h(\flA^L,\flA^R)^{[v_{>k}\varpi_s]_+},  \\
    &h_l^{\varpi_s^\ast}=\prod_{t \in S(w_0 w_l)}\alpha_t^\vee({h'}^{\varpi_t})^{w_l^{-1}\varpi_s^\ast} = \prod_{t \in S}\alpha_t^\vee({h'}^{\varpi_t})^{[w_l^{-1}\varpi_s^\ast]_+} = h(\flA_R,\flA_L)^{[u_{>l}^\ast\varpi_s^\ast]_+} = h(\flA_R,\flA_L)^{[u_{>l}\varpi_s]^\ast_+}.
\end{align*}
Here for the second equality in the first string: note that writing $(w^k)^{-1}\varpi_s = \sum_{t \in S}c_t \varpi_t$ with $c_t \in \Z$, we have
\begin{align*}
    \langle \varpi_s,\ w^k.\alpha_t^\vee \rangle = \langle (w^k)^{-1}.\varpi_s,\ \alpha_t^\vee \rangle = c_t.
\end{align*}
Hence $t \in S(w_0w^k) \Longleftrightarrow w^k.\alpha_t^\vee >0 \Longleftrightarrow c_t >0$. The second equality in the second string is similarly obtained. Thus the assertion is proved. 
\end{proof}


\begin{rem}\label{rem:BFZ variable}
Notice that the word $\bs_\bullet^\ast$ is read in the reversed order when we ``scan'' the distances of intermediate flags from the bottom to the top in \cref{fig:simple_configuration}. Therefore our seed associated with the double reduced word $\bs$ is the same as the one in \cite{BFZ} associated with $\bs^\bullet$ and $(\bs_\bullet^\ast)_\op:=(s_N^\ast,\dots,s_1^\ast)$. Letting $s'_{N-l+1}:=s_l$, the right-hand side of \eqref{eq:minor_Wilson_line} is written as $\Delta_{u'_{\leq N-l}\varpi_s,v_{>k}\varpi_s}(g_{[c]})$ with $u'_{\leq l'}:=r_{s'_1}\dots r_{s'_{l'}}$, which has the same form as the cluster variable in \cite[(2.11)]{BFZ}.
\end{rem}

\begin{rem}
When $m^L=m_L=:m$ and $\flA^L=\flA_L$, the situation reduces to the triangle case (\cref{ex:triangle_case}), where the parameters are identified as $h'=h_1$ and $g=u_+h_2^{-1}$. In the description of the cluster variables on $\Conf_3 \A_G$ given in \cite[Section 6]{GS19}, they are computed as follows:
\begin{align*}
    \Delta_{\lambda,\mu,\nu}(g\vw.[U^+],[U^+],\vw^{-1}h'.[U^+])
    &= \Delta_{\lambda,\mu,\nu}(g\vw^{-1}.[U^+],s_G.[U^+],\vw h'.[U^+]) \\
    &= \Delta_{\lambda,\mu,\nu}(g_{[c]}\vw.[U^+],[U^+],\vw.[U^+])\cdot s_G^\mu (h')^{\nu} \\
    &= F_{w\lambda}(g_{[c]}\vw.[U^+])\cdot (h')^{\nu} \\
    &= \Delta_{w\lambda,w_0\lambda}(g_{[c]})\cdot h(\flA_R,\flA_L)^{\nu}.
\end{align*}
This formula tells us how to express the cluster variables $\Delta_{\lambda,\mu,\nu}$ associated with an admissible triple $(\lambda,\mu,\nu)$ in terms of the matrix coefficients of Wilson lines.
\end{rem}

\subsection{Proof of \cref{main theorem}}\label{subsec:proof}
Recall from \cref{cor:generation_UCA_Wilson_lines} that the upper cluster algebra $\mathscr{U}_{\mathfrak{g},\Sigma}$ is generated by the matrix coefficients of simple Wilson lines and boundary Wilson lines, under the assumption that $\Sigma$ has at least two marked points. By \cref{prop:boundary condition} (1), the boundary Wilson line $g_{\sqrt{\fiber_E}^\out} \in H$ is the inverse of the $h$-invariant of the pair of decorated flags assigned to a boundary interval $E \in \mathbb{B}$. Hence its non-trivial matrix coefficients $(g_{\sqrt{\fiber_E}^\out})^{\varpi_s}$ for $s \in S$ are exactly the inverses of the $r$ many frozen variables on $E$. 
By \cref{prop:minors} and \cref{prop:simple_minors}, the matrix coefficients of simple Wilson lines are also cluster variables multiplied by the inverses of frozen variables. Therefore these generators are contained in the cluster algebra $\mathscr{A}_{\mathfrak{g},\Sigma}$, where the frozen variables are invertible. Thus we get $\mathscr{U}_{\mathfrak{g},\Sigma} \subset \mathscr{A}_{\mathfrak{g},\Sigma}$. 

\medskip
\section{Examples and skein description}
\label{section 5}
In this section, we  give an explicit description of the formula \eqref{eq:minor_Wilson_line} in the cases $G=SL_2,SL_3, Sp_4$, and their description in terms of the skein algebras studied in \cite{Muller,IY21,IY_C2}. Let us consider a simple Wilson line $g_{[c]}$, and use the notations in \cref{subsec:computation_minors}. We denote the frozen coordinates by 
\begin{align*}
    \clA_\inn^s:=h(\flA_R,\flA_L)^{\varpi_s^\ast} \quad\mbox{and}\quad \clA_\out^s:=h(\flA^L,\flA^R)^{\varpi_s}
\end{align*}
for $s \in S$.

\subsection{$SL_2$-case}
The vector representation $V(\varpi_1)=\mathbb{C}^2$ is minuscule. The weights in this representation are given by $\varpi_1$ and $-\varpi_1=w_0\varpi_1$. 
There is a unique choice of reduced word $\bs^\bullet=\bs_\bullet=(1)$ of $w_0 \in W$, which gives rise to the chains of decorated flags:
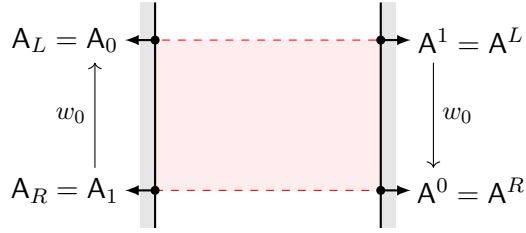
\begin{align*}
\begin{tikzpicture}
\filldraw[fill=pink!30,draw=red,dashed] (0,1.5) -- (3,1.5) -- (3,-.5) -- (0,-.5) --cycle;
\fill[gray!20] (0,2) -- (-0.2,2) -- (-0.2,-1) -- (0,-1) --cycle;
\fill[gray!20] (3,2) -- (3+0.2,2) -- (3+0.2,-1) -- (3,-1) --cycle;
\draw[thick] (0,2) -- (0,-1);
\draw[thick] (3,-1) -- (3,2);
\filldraw(0,1.5) circle(1.5pt); 
\filldraw(0,-.5) circle(1.5pt);
\filldraw(3,1.5) circle(1.5pt);
\filldraw(3,-.5) circle(1.5pt);
\draw[thick,-latex](0,-.5) --++(-0.4,0);
\draw[thick,-latex](0,1.5) --++(-0.4,0);
\draw[thick,-latex](3,-.5) --++(0.4,0);
\draw[thick,-latex](3,1.5) --++(0.4,0);
\node at (-1.2,-0.5) {$\flA_R=\flA_1$};
\node at (-1.2,1.5) {$\flA_L=\flA_0$};
\draw[->](-0.8,-0.2) --node[midway,left,scale=0.9]{$w_0$} (-0.8,1.2);
\node at (4.2,-.5) {$\flA^0=\flA^R$};
\node at (4.2,1.5) {$\flA^1=\flA^L$};
\draw[->](3.7,1.2) --node[midway,right,scale=0.9]{$w_0$}  (3.7,-0.2);
\end{tikzpicture}
\end{align*}


We have $[w_0\varpi_1]_+=0$ and $[\varpi_1]_+=\varpi_1$. 
Then by the formula \eqref{eq:minor_Wilson_line}, we get
\begin{align*}
    \Delta_{11}(g) &=\Delta_{\varpi_1,\varpi_1}(g) = \frac{\Delta_1(\flA^1,\flA_1)}{h(\flA^L,\flA^R)^{\varpi_1}h(\flA_R,\flA_L)^{\varpi_1}} = \frac{\Delta_1(\flA^1,\flA_1)}{\clA_\inn^1\clA_\out^1}, \\
    \Delta_{12}(g) &= \Delta_{\varpi_1,w_0\varpi_1}(g) = 
    \frac{\Delta_1(\flA^0,\flA_1)}{h(\flA^L,\flA^R)^{0}h(\flA_R,\flA_L)^{\varpi_1}} = \frac{\Delta_1(\flA^0,\flA_1)}{\clA_\inn^1}, \\
    \Delta_{21}(g) &= \Delta_{w_0\varpi_1,\varpi_1}(g) = \frac{\Delta_1(\flA^1,\flA_0)}{h(\flA^L,\flA^R)^{\varpi_1}h(\flA_R,\flA_L)^{0}} = \frac{\Delta_1(\flA^1,\flA_0)}{\clA_\out^1}, \\ 
    \Delta_{22}(g) &= \Delta_{w_0\varpi_1,w_0\varpi_1}(g) = \frac{\Delta_1(\flA^0,\flA_0)}{h(\flA^L,\flA^R)^{0}h(\flA_R,\flA_L)^{0}} =\Delta_1(\flA^0,\flA_0).
\end{align*}
Thus the Wilson line matrix $g_{[c]} \in SL_2(\cO(\A_{SL_2,\Sigma}^\times))$ is collectively given by
\begin{align}\label{eq:A1_matrix}
    g_{[c]} = \begin{pmatrix}
    \displaystyle\frac{\Delta_1(\flA^1,\flA_1)}{\clA_\inn^1\clA_\out^1} & 
    \displaystyle\frac{\Delta_1(\flA^0,\flA_1)}{\clA_\inn^1}\vspace{2mm} \\
    \displaystyle\frac{\Delta_1(\flA^1,\flA_0)}{\clA_\out^1} & \Delta_1(\flA^0,\flA_0)
    \end{pmatrix}.
\end{align}
The clusters corresponding to the two ideal triangulations are shown in \cref{fig:A1_clusters}. 

\begin{figure}[ht]
\begin{tikzpicture}
\begin{scope}[>=latex]
\draw[blue,very thin] (0,1.5) -- (3,1.5) -- (3,-1.5) -- (0,-1.5) --cycle;
\draw[blue,very thin] (3,1.5) -- (0,-1.5);
\foreach \i in {0,3} \foreach \j in {-1.5,1.5} \fill(\i,\j) circle(1.5pt);
{\color{mygreen}
\draw (1.5,0) circle(2pt) coordinate(A);
\draw (0,0) circle(2pt) coordinate(B);
\draw (1.5,-1.5) circle(2pt) coordinate(C);
\draw (3,0) circle(2pt) coordinate(D);
\draw (1.5,1.5) circle(2pt) coordinate(E);
\qarrow{A}{E};
\qarrow{E}{B};
\qarrow{B}{A};
\qarrow{A}{C};
\qarrow{C}{D};
\qarrow{D}{A};
}
\draw[myorange] (A)++(0,0.4) node[anchor=west,scale=0.9]{$\Delta_1(\flA^1,\flA_1)$};
\draw (C) node[below=0.2em,scale=0.9]{$\Delta_1(\flA^0,\flA_1)$};
\draw (E) node[above=0.2em,scale=0.9]{$\Delta_1(\flA^1,\flA_0)$};
\draw (B) node[left,scale=0.9]{$A_\inn^1$};
\draw (D) node[right,scale=0.9]{$A_\out^1$};
\draw[myorange] (A) circle(5pt);
\end{scope}
\draw[thick,|->] (4,0) -- (5,0);

\begin{scope}[xshift=6cm,>=latex]
\draw[blue,very thin] (0,1.5) -- (3,1.5) -- (3,-1.5) -- (0,-1.5) --cycle;
\draw[blue,very thin] (0,1.5) -- (3,-1.5);
\foreach \i in {0,3} \foreach \j in {-1.5,1.5} \fill(\i,\j) circle(1.5pt);
{\color{mygreen}
\draw (1.5,0) circle(2pt) coordinate(A);
\draw (0,0) circle(2pt) coordinate(B);
\draw (1.5,-1.5) circle(2pt) coordinate(C);
\draw (3,0) circle(2pt) coordinate(D);
\draw (1.5,1.5) circle(2pt) coordinate(E);
\qarrow{E}{A};
\qarrow{B}{C};
\qarrow{A}{B};
\qarrow{C}{A};
\qarrow{D}{E};
\qarrow{A}{D};
}
\draw (A)++(0,-0.5) node[anchor=west,scale=0.9]{$\Delta_1(\flA^0,\flA_0)$};
\draw (C) node[below=0.2em,scale=0.9]{$\Delta_1(\flA^0,\flA_1)$};
\draw (E) node[above=0.2em,scale=0.9]{$\Delta_1(\flA^1,\flA_0)$};
\draw (B) node[left,scale=0.9]{$A_\inn^1$};
\draw (D) node[right,scale=0.9]{$A_\out^1$};
\end{scope}
\end{tikzpicture}
    \caption{The two clusters in $\cO(\Conf_\inn^\out \A_{SL_2})$ related by a single mutation.}
    \label{fig:A1_clusters}
\end{figure}
Recall the skein model $\mathscr{A}_{\mathfrak{sl}_2,\Sigma} \cong \mathscr{S}_{\Sigma}^1[\partial^{-1}]$ given by Muller \cite{Muller}, where the cluster variables are identified with ideal arcs, and the clusters correspond to the ideal triangulations. In this language, the matrix \eqref{eq:A1_matrix} is written as
\begin{align*}
    g_{[c]} = \begin{pmatrix}
    \tikz{
    \smallsq;
    \draw[myblue,thick](0,0) to[bend right=20] (0,1); 
    \draw[myblue,thick] (1,0) to[bend left=20] (1,1);
    \draw[red,thick](0,0) -- (1,1);} & 
    \tikz{
    \smallsq;
    \draw[myblue,thick](0,0) to[bend right=20] (0,1); 
    \draw[red,thick](0,0) to[bend left=20] (1,0);} \vspace{2mm} \\
    \tikz{
    \smallsq;
    \draw[myblue,thick] (1,0) to[bend left=20] (1,1);
    \draw[red,thick](0,1) to[bend right=20] (1,1);} &
    \tikz{
    \smallsq;
    \draw[red,thick](0,1) -- (1,0);}
    \end{pmatrix}.
\end{align*}
Here the inverse of an ideal arc is shown in blue. 

\subsection{$SL_3$-case}
The vector representation $V(\varpi_1)=\mathbb{C}^3$ is minuscule. The weights in this representation are given by $\varpi_1$, $r_1\varpi_1=\varpi_2-\varpi_1$ and $r_2r_1\varpi_1=-\varpi_2=w_0\varpi_1$. 
Let us choose the reduced word $\bs^\bullet=\bs_\bullet=(1,2,1)$ of $w_0 \in W$, which gives rise to the following chains of decorated flags:
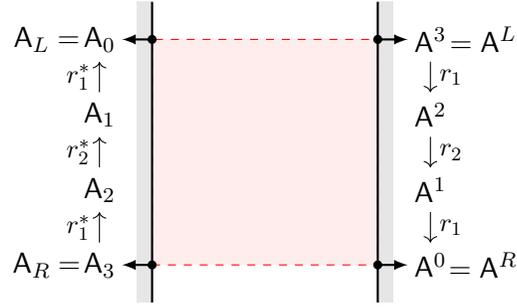
\begin{align*}
\begin{tikzpicture}
\filldraw[fill=pink!30,draw=red,dashed] (0,1.5) -- (3,1.5) -- (3,-1.5) -- (0,-1.5) --cycle;
\fill[gray!20] (0,2) -- (-0.2,2) -- (-0.2,-2) -- (0,-2) --cycle;
\fill[gray!20] (3,2) -- (3+0.2,2) -- (3+0.2,-2) -- (3,-2) --cycle;
\draw[thick] (0,2) -- (0,-2);
\draw[thick] (3,-2) -- (3,2);
\filldraw(0,1.5) circle(1.5pt); 
\filldraw(0,-1.5) circle(1.5pt);
\filldraw(3,1.5) circle(1.5pt);
\filldraw(3,-1.5) circle(1.5pt);
\draw[thick,-latex](0,-1.5) --++(-0.4,0);
\draw[thick,-latex](0,1.5) --++(-0.4,0);
\draw[thick,-latex](3,-1.5) --++(0.4,0);
\draw[thick,-latex](3,1.5) --++(0.4,0);
\node (L3) at (-0.7,-1.5) {$\flA_3$};
\node (L2) at (-0.7,-0.5) {$\flA_2$};
\node (L1) at (-0.7,0.5) {$\flA_1$};
\node (L0) at (-0.7,1.5) {$\flA_0$};
\draw (-0.7,1.5)++(-0.7,0) node{$\flA_L=$};
\draw (-0.7,-1.5)++(-0.7,0) node{$\flA_R=$};
\draw[->](L3) --node[midway,left,scale=0.9]{$r_1^\ast$} (L2);
\draw[->](L2) --node[midway,left,scale=0.9]{$r_2^\ast$} (L1);
\draw[->](L1) --node[midway,left,scale=0.9]{$r_1^\ast$} (L0);
\node (R3) at (3.7,-1.5) {$\flA^0$};
\node (R2) at (3.7,-0.5) {$\flA^1$};
\node (R1) at (3.7,0.5) {$\flA^2$};
\node (R0) at (3.7,1.5) {$\flA^3$};
\draw (3.7,1.5)++(0.7,0) node{$=\flA^L$};
\draw (3.7,-1.5)++(0.7,0) node{$=\flA^R$};
\draw[<-](R3) --node[midway,right,scale=0.9]{$r_1$} (R2);
\draw[<-](R2) --node[midway,right,scale=0.9]{$r_2$} (R1);
\draw[<-](R1) --node[midway,right,scale=0.9]{$r_1$} (R0);
\end{tikzpicture}
\end{align*}
Then the matrix entries of the simple Wilson line $g_{[c]}$ in $V(\varpi_1)$ are computed as follows:
\begin{align*}
    \Delta_{11}(g_{[c]}) &= \Delta_{\varpi_1,\varpi_1}(g_{[c]}) = \frac{\Delta_1(\flA^3,\flA_3)}{h(\flA_R,\flA_L)^{\varpi_1^\ast}h(\flA^L,\flA^R)^{\varpi_1}}=\frac{\Delta_1(\flA^3,\flA_3)}{\clA_\inn^1 \clA_\out^1}, \\
    \Delta_{12}(g_{[c]}) &= \Delta_{\varpi_1,r_1\varpi_1}(g_{[c]}) = \frac{\Delta_1(\flA^2,\flA_3)}{h(\flA_R,\flA_L)^{\varpi_1^\ast}h(\flA^L,\flA^R)^{\varpi_2}}=\frac{\Delta_1(\flA^2,\flA_3)}{\clA_\inn^1 \clA_\out^2},\\
    \Delta_{13}(g_{[c]}) &= \Delta_{\varpi_1,w_0\varpi_1}(g_{[c]}) = \frac{\Delta_1(\flA^0,\flA_3)}{h(\flA_R,\flA_L)^{\varpi_1^\ast}}=\frac{\Delta_1(\flA^0,\flA_3)}{\clA_\inn^1}, \\
    \Delta_{21}(g_{[c]}) &= \Delta_{r_1\varpi_1,\varpi_1}(g_{[c]}) = \frac{\Delta_1(\flA^3,\flA_2)}{h(\flA_R,\flA_L)^{\varpi_2^\ast}h(\flA^L,\flA^R)^{\varpi_1}}=\frac{\Delta_1(\flA^3,\flA_2)}{\clA_\inn^2 \clA_\out^1}, \\
    \Delta_{22}(g_{[c]}) &= \Delta_{r_1\varpi_1,r_1\varpi_1}(g_{[c]}) = \frac{\Delta_1(\flA^2,\flA_2)}{h(\flA_R,\flA_L)^{\varpi_2^\ast}h(\flA^L,\flA^R)^{\varpi_2}}=\frac{\Delta_1(\flA^2,\flA_2)}{\clA_\inn^2 \clA_\out^2},\\
    \Delta_{23}(g_{[c]}) &= \Delta_{r_1\varpi_1,w_0\varpi_1}(g_{[c]}) = \frac{\Delta_1(\flA^0,\flA_2)}{h(\flA_R,\flA_L)^{\varpi_2^\ast}}=\frac{\Delta_1(\flA^0,\flA_2)}{\clA_\inn^2}, \\
    \Delta_{31}(g_{[c]}) &= \Delta_{w_0\varpi_1,\varpi_1}(g_{[c]}) = \frac{\Delta_1(\flA^3,\flA_0)}{h(\flA^L,\flA^R)^{\varpi_1}}=\frac{\Delta_1(\flA^3,\flA_0)}{\clA_\out^1}, \\
    \Delta_{32}(g_{[c]}) &= \Delta_{w_0\varpi_1,r_1\varpi_1}(g_{[c]}) = \frac{\Delta_1(\flA^2,\flA_0)}{h(\flA^L,\flA^R)^{\varpi_2}}=\frac{\Delta_1(\flA^2,\flA_0)}{\clA_\out^2},\\
    \Delta_{33}(g_{[c]}) &= \Delta_{w_0\varpi_1,w_0\varpi_1}(g_{[c]}) = \Delta_1(\flA^0,\flA_0)=\Delta_1(\flA^0,\flA_0).
\end{align*}
Thus the Wilson line matrix $g_{[c]} \in SL_3(\cO(\A_{SL_3,\Sigma}^\times))$ is collectively given by
\begin{align}\label{eq:A2_matrix_cluster}
    g_{[c]} = \begin{pmatrix}
    \displaystyle\frac{\Delta_1(\flA^3,\flA_3)}{\clA_\inn^1 \clA_\out^1} & 
    \displaystyle\frac{\Delta_1(\flA^2,\flA_3)}{\clA_\inn^1 \clA_\out^2} & 
    \displaystyle\frac{\Delta_1(\flA^0,\flA_3)}{\clA_\inn^1}\vspace{3mm} \\
    \displaystyle\frac{\Delta_1(\flA^3,\flA_2)}{\clA_\inn^2 \clA_\out^1} & 
    \displaystyle\frac{\Delta_1(\flA^2,\flA_2)}{\clA_\inn^2 \clA_\out^2} & 
    \displaystyle\frac{\Delta_1(\flA^0,\flA_2)}{\clA_\inn^2} \vspace{3mm}\\
    \displaystyle\frac{\Delta_1(\flA^3,\flA_0)}{\clA_\out^1} & \displaystyle\frac{\Delta_1(\flA^2,\flA_0)}{\clA_\out^2} & 
    \Delta_1(\flA^0,\flA_0)
    \end{pmatrix}.
\end{align}
The clusters to which the cluster variables appearing above belong are shown in \cref{fig:A2_clusters}, as well as the mutation sequences relating them. Here we use the fact that the two words $(\cdots s \overline{t}\cdots)$ and $(\cdots \overline{t} s \cdots)$ give rise to the same cluster for $s \neq t$. 
The corresponding transformations of dashed diagonals connecting the decorated flags are shown in \cref{fig:A2_diagonals}.

\begin{figure}[ht]
\begin{tikzpicture}[scale=1.05]
\begin{scope}[>=latex]
\draw[blue,very thin] (0,2) -- (4,2) -- (4,-2) -- (0,-2) --cycle;
\draw[blue,very thin,rounded corners=5pt] (4,2) -- (8/3,0) -- (4/3,0) -- (0,-2);
\foreach \i in {0,4} \foreach \j in {-2,2} \fill(\i,\j) circle(1.5pt);
\draw[thick,|->] (5,0) -- (6,0);
{\color{mygreen}
\foreach \i in {0,1,2,3,4}
\draw (8/3,\i-2) circle(2pt) coordinate(x1\i);
\foreach \i in {0,1,2}
\draw (4/3,2*\i-2) circle(2pt) coordinate(x2\i);
\qarrow{x11}{x10}
\qarrow{x12}{x11}
\qarrow{x12}{x13}
\qarrow{x13}{x14}
\qarrow{x21}{x20}
\qarrow{x21}{x22}
\qarrow{x20}{x11}
\qarrow{x11}{x21}
\qarrow{x22}{x13}
\qarrow{x13}{x21}
\qarrow{x21}{x12}
\qdlarrow{x10}{x20}
\qdrarrow{x14}{x22}
}
\draw (x10)++(0,-0.3) node[anchor=west,scale=0.85]{$\Delta_1(\flA^0,\flA_3)$};
\draw (x20)++(0,-0.3) node[anchor=east,scale=0.85]{$\Delta_2(\flA^1,\flA_3)\quad$};
\draw (x20)++(0,-0.7) node[anchor=east,scale=0.85]{$=\Delta_2(\flA^0,\flA_3)$};
\draw (x11) node[anchor=west,scale=0.85]{$\ \Delta_1(\flA^2,\flA_3)$};
\draw (x11)++(0,-0.4) node[anchor=west,scale=0.85]{$\ =\Delta_1(\flA^1,\flA_3)$};
\draw[myorange] (x12) node[anchor=west,scale=0.85]{$\ \Delta_1(\flA^3,\flA_3)$};
\draw (x13) node[anchor=west,scale=0.85]{$\ \Delta_1(\flA^3,\flA_1)$};
\draw (x13)++(0,-0.4) node[anchor=west,scale=0.85]{$\ =\Delta_1(\flA^3,\flA_2)$};
\draw (x14)++(0,0.3) node[anchor=west,scale=0.85]{$\ \Delta_1(\flA^3,\flA_0)$};
\draw (x21)++(0,0.4) node[anchor=east,scale=0.85]{$\Delta_2(\flA^2,\flA_3)\quad$};
\draw (x21)++(0,0) node[anchor=east,scale=0.85]{$=\Delta_2(\flA^3,\flA_3)$};
\draw (x21)++(0,-0.4) node[anchor=east,scale=0.85]{$=\Delta_2(\flA^3,\flA_2)$};
\draw (x22)++(0,0.7) node[anchor=east,scale=0.85]{$\Delta_2(\flA^3,\flA_0)\quad$};
\draw (x22)++(0,0.3) node[anchor=east,scale=0.85]{$=\Delta_2(\flA^3,\flA_1)$};
\draw[myorange](x12) circle(5pt);
\draw (2,-3.5) node[draw,rectangle,rounded corners=0.3em]{$\overline{1},\overline{2},{\color{myorange}\underline{{\color{black}\overline{1},1}}},2,1$};
\end{scope}

\begin{scope}[xshift=7cm,>=latex]
\draw[blue,very thin] (0,2) -- (4,2) -- (4,-2) -- (0,-2) --cycle;
\foreach \i in {0,4} \foreach \j in {-2,2} \fill(\i,\j) circle(1.5pt);
{\color{mygreen}
\foreach \i in {0,1,2,3,4}
\draw (8/3,\i-2) circle(2pt) coordinate(x1\i);
\foreach \i in {0,1,2}
\draw (4/3,2*\i-2) circle(2pt) coordinate(x2\i);
\qarrow{x11}{x10}
\qarrow{x11}{x12}
\qarrow{x13}{x12}
\qarrow{x13}{x14}
\qarrow{x21}{x20}
\qarrow{x21}{x22}
\qarrow{x20}{x11}
\qarrow{x22}{x13}
\qarrow{x12}{x21}
\qdlarrow{x10}{x20}
\qdrarrow{x14}{x22}
}
\draw (x10)++(0,-0.3) node[anchor=west,scale=0.85]{$\Delta_1(\flA^0,\flA_3)$};
\draw (x20)++(0,-0.3) node[anchor=east,scale=0.85]{$\Delta_2(\flA^1,\flA_3)\quad$};
\draw (x20)++(0,-0.7) node[anchor=east,scale=0.85]{$=\Delta_2(\flA^0,\flA_3)$};
\draw (x11) node[anchor=west,scale=0.85]{$\ \Delta_1(\flA^2,\flA_3)$};
\draw[myorange] (x11)++(0,-0.4) node[anchor=west,scale=0.85]{$\ =\Delta_1(\flA^1,\flA_3)$};
\draw (x12) node[anchor=west,scale=0.85]{$\ \Delta_1(\flA^2,\flA_2)$};
\draw[myorange] (x13) node[anchor=west,scale=0.85]{$\ \Delta_1(\flA^3,\flA_1)$};
\draw (x13)++(0,-0.4) node[anchor=west,scale=0.85]{$\ =\Delta_1(\flA^3,\flA_2)$};
\draw (x14)++(0,0.3) node[anchor=west,scale=0.85]{$\ \Delta_1(\flA^3,\flA_0)$};
\draw[myorange] (x21)++(0,0) node[anchor=east,scale=0.85]{$\Delta_2(\flA^2,\flA_2)\quad$};
\draw (x21)++(0,-0.4) node[anchor=east,scale=0.85]{$=\Delta_2(\flA^2,\flA_3)$};
\draw (x22)++(0,0.7) node[anchor=east,scale=0.85]{$\Delta_2(\flA^3,\flA_0)\quad$};
\draw (x22)++(0,0.3) node[anchor=east,scale=0.85]{$=\Delta_2(\flA^3,\flA_1)$};
\draw[myorange](x11) circle(5pt);
\draw[myorange](x21) circle(5pt);
\draw[myorange](x13) circle(5pt);
\draw (2,-3.5) node[draw,rectangle,rounded corners=0.3em]{$\overline{1},\overline{2},1,\overline{1},2,1\ \sim\ {\color{myorange}\underline{{\color{black}\overline{1},1}}},{\color{myorange}\underline{{\color{black}\overline{2},2}}},{\color{myorange}\underline{{\color{black}\overline{1},1}}}$};
\end{scope}

\begin{scope}[xshift=2cm,yshift=-7.5cm,>=latex]
\draw[thick,|->] (-2,0) -- (-1,0);
\draw[blue,very thin] (0,2) -- (4,2) -- (4,-2) -- (0,-2) --cycle;
\foreach \i in {0,4} \foreach \j in {-2,2} \fill(\i,\j) circle(1.5pt);
{\color{mygreen}
\foreach \i in {0,1,2,3,4}
\draw (8/3,\i-2) circle(2pt) coordinate(x1\i);
\foreach \i in {0,1,2}
\draw (4/3,2*\i-2) circle(2pt) coordinate(x2\i);
\qarrow{x10}{x11}
\qarrow{x12}{x11}
\qarrow{x12}{x13}
\qarrow{x14}{x13}
\qarrow{x20}{x21}
\qarrow{x22}{x21}
\qarrow{x11}{x20}
\qarrow{x13}{x22}
\qarrow{x21}{x12}
\qdrarrow{x20}{x10}
\qdlarrow{x22}{x14}
}
\draw (x10)++(0,-0.3) node[anchor=west,scale=0.85]{$\Delta_1(\flA^0,\flA_3)$};
\draw (x20)++(0,-0.3) node[anchor=east,scale=0.85]{$\Delta_2(\flA^0,\flA_2)\quad$};
\draw (x20)++(0,-0.7) node[anchor=east,scale=0.85]{$=\Delta_2(\flA^0,\flA_3)$};
\draw (x11) node[anchor=west,scale=0.85]{$\ \Delta_1(\flA^0,\flA_1)$};
\draw (x11)++(0,-0.4) node[anchor=west,scale=0.85]{$\ =\Delta_1(\flA^0,\flA_2)$};
\draw (x12) node[anchor=west,scale=0.85]{$\ \Delta_1(\flA^2,\flA_2)$};
\draw[myorange] (x12)++(0,-0.4) node[anchor=west,scale=0.85]{$\ =\Delta_1(\flA^1,\flA_1)$};
\draw (x13) node[anchor=west,scale=0.85]{$\ \Delta_1(\flA^2,\flA_0)$};
\draw (x13)++(0,-0.4) node[anchor=west,scale=0.85]{$\ =\Delta_1(\flA^1,\flA_0)$};
\draw (x14)++(0,0.3) node[anchor=west,scale=0.85]{$\ \Delta_1(\flA^3,\flA_0)$};
\draw (x21)++(0,0) node[anchor=east,scale=0.85]{$\Delta_2(\flA^1,\flA_1)\quad$};
\draw (x21)++(0,-0.4) node[anchor=east,scale=0.85]{$=\Delta_2(\flA^0,\flA_0)$};
\draw (x22)++(0,0.7) node[anchor=east,scale=0.85]{$\Delta_2(\flA^3,\flA_0)\quad$};
\draw (x22)++(0,0.3) node[anchor=east,scale=0.85]{$=\Delta_2(\flA^2,\flA_0)$};
\draw[myorange] (x12) circle(5pt);
\draw (2,-3.5) node[draw,rectangle,rounded corners=0.3em]{$1,\overline{1},2,\overline{2},1,\overline{1}\ \sim\  1,2,{\color{myorange}\underline{{\color{black}\overline{1},1}}},\overline{2},\overline{1}$};
\end{scope}

\begin{scope}[xshift=9cm,yshift=-7.5cm,>=latex]
\draw[thick,|->] (-2,0) -- (-1,0);
\draw[blue,very thin] (0,2) -- (4,2) -- (4,-2) -- (0,-2) --cycle;
\draw[blue,very thin,rounded corners=5pt] (4,-2) -- (8/3,0) -- (4/3,0) -- (0,2);
\foreach \i in {0,4} \foreach \j in {-2,2} \fill(\i,\j) circle(1.5pt);
{\color{mygreen}
\foreach \i in {0,1,2,3,4}
\draw (8/3,\i-2) circle(2pt) coordinate(x1\i);
\foreach \i in {0,1,2}
\draw (4/3,2*\i-2) circle(2pt) coordinate(x2\i);
\qarrow{x10}{x11}
\qarrow{x11}{x12}
\qarrow{x13}{x12}
\qarrow{x14}{x13}
\qarrow{x20}{x21}
\qarrow{x22}{x21}
\qarrow{x11}{x20}
\qarrow{x13}{x22}
\qarrow{x12}{x21}
\qarrow{x21}{x11}
\qarrow{x21}{x13}
\qdrarrow{x20}{x10}
\qdlarrow{x22}{x14}
}
\draw (x10)++(0,-0.3) node[anchor=west,scale=0.85]{$\Delta_1(\flA^0,\flA_3)$};
\draw (x20)++(0,-0.3) node[anchor=east,scale=0.85]{$\Delta_2(\flA^0,\flA_2)\quad$};
\draw (x20)++(0,-0.7) node[anchor=east,scale=0.85]{$=\Delta_2(\flA^0,\flA_3)$};
\draw (x11) node[anchor=west,scale=0.85]{$\ \Delta_1(\flA^0,\flA_1)$};
\draw (x11)++(0,-0.4) node[anchor=west,scale=0.85]{$\ =\Delta_1(\flA^0,\flA_2)$};
\draw (x12) node[anchor=west,scale=0.85]{$\ \Delta_1(\flA^0,\flA_0)$};
\draw (x13) node[anchor=west,scale=0.85]{$\ \Delta_1(\flA^2,\flA_0)$};
\draw (x13)++(0,-0.4) node[anchor=west,scale=0.85]{$\ =\Delta_1(\flA^1,\flA_0)$};
\draw (x14)++(0,0.3) node[anchor=west,scale=0.85]{$\ \Delta_1(\flA^3,\flA_0)$};
\draw (x21)++(0,0.4) node[anchor=east,scale=0.85]{$\Delta_2(\flA^1,\flA_0)\quad$};
\draw (x21)++(0,0) node[anchor=east,scale=0.85]{$=\Delta_2(\flA^0,\flA_0)$};
\draw (x21)++(0,-0.4) node[anchor=east,scale=0.85]{$=\Delta_2(\flA^0,\flA_1)$};
\draw (x22)++(0,0.7) node[anchor=east,scale=0.85]{$\Delta_2(\flA^3,\flA_0)\quad$};
\draw (x22)++(0,0.3) node[anchor=east,scale=0.85]{$=\Delta_2(\flA^2,\flA_0)$};
\draw (2,-3.5) node[draw,rectangle,rounded corners=0.3em]{$1,2,1,\overline{1},\overline{2},\overline{1}$};
\end{scope}
\end{tikzpicture}
    \caption{Some clusters in $\cO(\Conf_\inn^\out \A_{SL_3})$ and the mutation sequences relating them. Here the frozen variables/vertices are omitted, and the mutated vertices are shown in orange.}
    \label{fig:A2_clusters}
\end{figure}
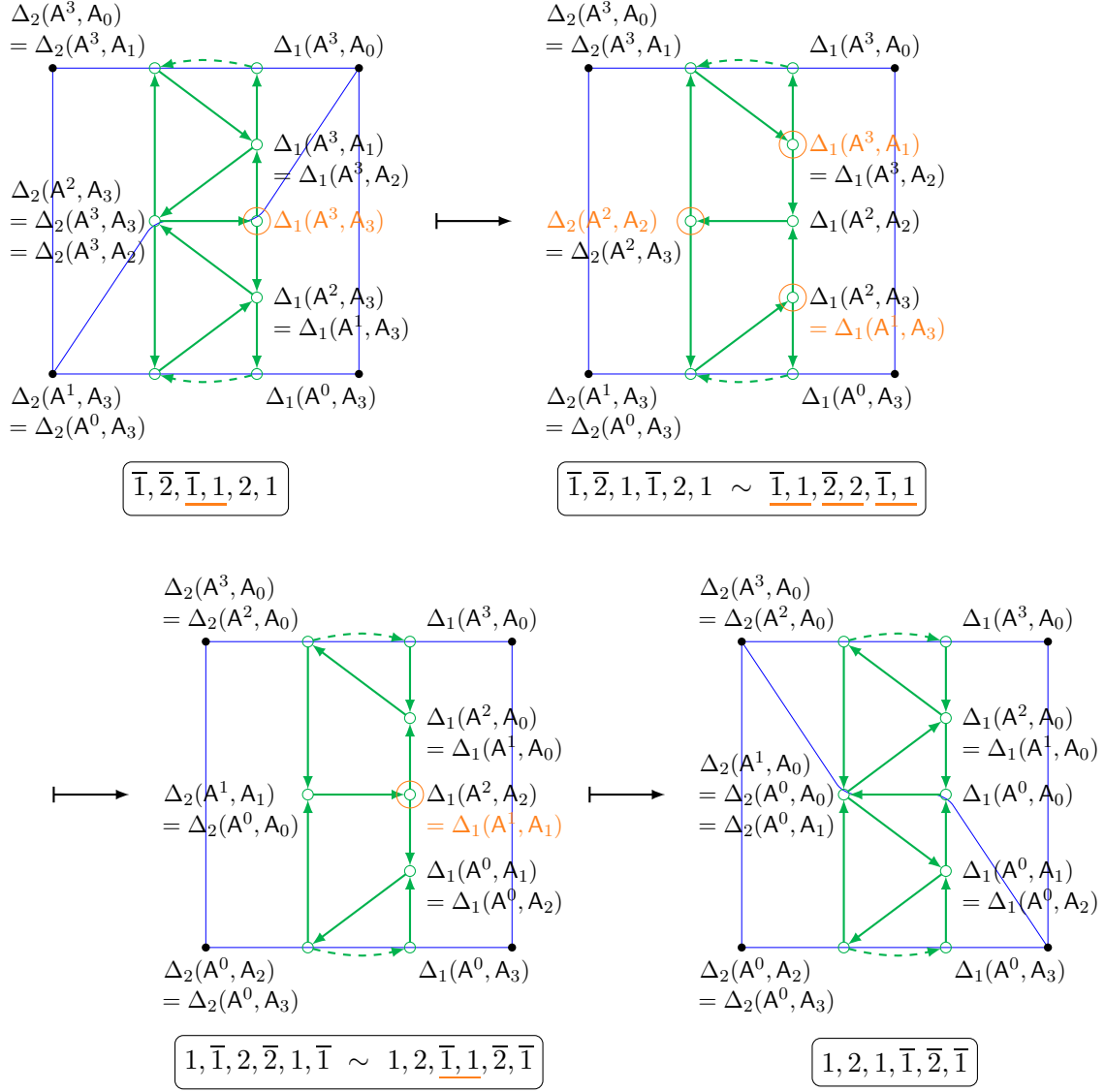

\begin{figure}[ht]
\begin{tikzpicture}[scale=0.8]
\draw[thick,|->] (5,0) -- (6,0);
\begin{scope}
\draw[thick] (0,2) -- (0,-2);
\draw[thick] (4,-2) -- (4,2);
\foreach \i in {0,4}
{
\foreach \j in {-2,-0.667,0.667,2}
\filldraw(\i,\j) circle(1.5pt);
}
\draw (-0,-2) coordinate(L3) node[left]{$\flA_3$};
\draw (-0,-0.667) coordinate(L2) node[left]{$\flA_2$};
\draw (-0,0.667) coordinate(L1) node[left]{$\flA_1$};
\draw (-0,2) coordinate(L0) node[left]{$\flA_0$};
\draw (4,-2) coordinate(R0) node[right]{$\flA^0$};
\draw (4,-0.667) coordinate(R1) node[right]{$\flA^1$};
\draw (4,0.667) coordinate(R2) node[right]{$\flA^2$};
\draw (4,2) coordinate(R3) node[right]{$\flA^3$};
\draw[dashed] (R3) -- (L3);
\foreach \i in {0,1,2}
{
\draw[dashed] (L3) -- (R\i);
\draw[dashed] (R3) -- (L\i);
}
\draw[orange] (L3) -- (R3);
\node[left,scale=0.9] at (0,-1.333) {$r_1^\ast$};
\node[left,scale=0.9] at (0,0) {$r_2^\ast$};
\node[left,scale=0.9] at (0,1.333) {$r_1^\ast$};
\node[right,scale=0.9] at (4,-1.333) {$r_1$};
\node[right,scale=0.9] at (4,0) {$r_2$};
\node[right,scale=0.9] at (4,1.333) {$r_1$};
\draw (2,-3) node[draw,rectangle,rounded corners=0.3em]{$\overline{1},\overline{2},{\color{myorange}\underline{{\color{black}\overline{1},1}}},2,1$};
\end{scope}
\begin{scope}[xshift=7cm]
\draw[thick] (0,2) -- (0,-2);
\draw[thick] (4,-2) -- (4,2);
\foreach \i in {0,4}
{
\foreach \j in {-2,-0.667,0.667,2}
\filldraw(\i,\j) circle(1.5pt);
}
\draw (-0,-2) coordinate(L3) node[left]{$\flA_3$};
\draw (-0,-0.667) coordinate(L2) node[left]{$\flA_2$};
\draw (-0,0.667) coordinate(L1) node[left]{$\flA_1$};
\draw (-0,2) coordinate(L0) node[left]{$\flA_0$};
\draw (4,-2) coordinate(R0) node[right]{$\flA^0$};
\draw (4,-0.667) coordinate(R1) node[right]{$\flA^1$};
\draw (4,0.667) coordinate(R2) node[right]{$\flA^2$};
\draw (4,2) coordinate(R3) node[right]{$\flA^3$};
\draw[dashed] (R2) -- (L2);
\foreach \i in {0,1,2}
{
\draw[dashed] (L3) -- (R\i);
\draw[dashed] (R3) -- (L\i);
}
\node[left,scale=0.9] at (0,-1.333) {$r_1^\ast$};
\node[left,scale=0.9] at (0,0) {$r_2^\ast$};
\node[left,scale=0.9] at (0,1.333) {$r_1^\ast$};
\node[right,scale=0.9] at (4,-1.333) {$r_1$};
\node[right,scale=0.9] at (4,0) {$r_2$};
\node[right,scale=0.9] at (4,1.333) {$r_1$};
\draw (2,-3) node[draw,rectangle,rounded corners=0.3em]{$\overline{1},\overline{2},1,\overline{1},2,1$};
\end{scope}

\begin{scope}[xshift=13cm]
\draw (-1,0) node[scale=1.3]{$=$};
\draw[thick] (0,2) -- (0,-2);
\draw[thick] (4,-2) -- (4,2);
\foreach \i in {0,4}
{
\foreach \j in {-2,-0.667,0.667,2}
\filldraw(\i,\j) circle(1.5pt);
}
\draw (-0,-2) coordinate(L3) node[left]{$\flA_3$};
\draw (-0,-0.667) coordinate(L2) node[left]{$\flA_2$};
\draw (-0,0.667) coordinate(L1) node[left]{$\flA_1$};
\draw (-0,2) coordinate(L0) node[left]{$\flA_0$};
\draw (4,-2) coordinate(R0) node[right]{$\flA^0$};
\draw (4,-0.667) coordinate(R1) node[right]{$\flA^1$};
\draw (4,0.667) coordinate(R2) node[right]{$\flA^2$};
\draw (4,2) coordinate(R3) node[right]{$\flA^3$};
\draw[dashed] (L3) -- (R0);
\draw[dashed] (L2) -- (R1);
\draw[dashed] (L1) -- (R2);
\draw[dashed] (L0) -- (R3);
\draw[dashed] (R2) -- (L2);
\draw[dashed] (R1) -- (L3);
\draw[dashed] (R3) -- (L1);
\draw[myorange] (R1) -- (L3);
\draw[myorange] (R2) -- (L2);
\draw[myorange] (R3) -- (L1);
\node[left,scale=0.9] at (0,-1.333) {$r_1^\ast$};
\node[left,scale=0.9] at (0,0) {$r_2^\ast$};
\node[left,scale=0.9] at (0,1.333) {$r_1^\ast$};
\node[right,scale=0.9] at (4,-1.333) {$r_1$};
\node[right,scale=0.9] at (4,0) {$r_2$};
\node[right,scale=0.9] at (4,1.333) {$r_1$};
\draw (2,-3) node[draw,rectangle,rounded corners=0.3em]{${\color{myorange}\underline{{\color{black}\overline{1},1}}},{\color{myorange}\underline{{\color{black}\overline{2},2}}},{\color{myorange}\underline{{\color{black}\overline{1},1}}}$};
\end{scope}

\draw[thick,|->](-0.5,-6.5) -- (0.5,-6.5);
\begin{scope}[xshift=1.5cm,yshift=-6.5cm,>=latex]
\draw[thick] (0,2) -- (0,-2);
\draw[thick] (4,-2) -- (4,2);
\foreach \i in {0,4}
{
\foreach \j in {-2,-0.667,0.667,2}
\filldraw(\i,\j) circle(1.5pt);
}
\draw (-0,-2) coordinate(L3) node[left]{$\flA_3$};
\draw (-0,-0.667) coordinate(L2) node[left]{$\flA_2$};
\draw (-0,0.667) coordinate(L1) node[left]{$\flA_1$};
\draw (-0,2) coordinate(L0) node[left]{$\flA_0$};
\draw (4,-2) coordinate(R0) node[right]{$\flA^0$};
\draw (4,-0.667) coordinate(R1) node[right]{$\flA^1$};
\draw (4,0.667) coordinate(R2) node[right]{$\flA^2$};
\draw (4,2) coordinate(R3) node[right]{$\flA^3$};
\draw[dashed] (L3) -- (R0);
\draw[dashed] (L2) -- (R1);
\draw[dashed] (L1) -- (R2);
\draw[dashed] (L0) -- (R3);
\draw[dashed] (R1) -- (L1);
\draw[dashed] (R0) -- (L2);
\draw[dashed] (R2) -- (L0);
\node[left,scale=0.9] at (0,-1.333) {$r_1^\ast$};
\node[left,scale=0.9] at (0,0) {$r_2^\ast$};
\node[left,scale=0.9] at (0,1.333) {$r_1^\ast$};
\node[right,scale=0.9] at (4,-1.333) {$r_1$};
\node[right,scale=0.9] at (4,0) {$r_2$};
\node[right,scale=0.9] at (4,1.333) {$r_1$};
\draw (2,-3) node[draw,rectangle,rounded corners=0.3em]{$1,\overline{1},2,\overline{2},1,\overline{1}$};
\end{scope}

\begin{scope}[xshift=7.5cm,yshift=-6.5cm,>=latex]
\draw (-1,0) node[scale=1.3]{$=$};
\draw[thick] (0,2) -- (0,-2);
\draw[thick] (4,-2) -- (4,2);
\foreach \i in {0,4}
{
\foreach \j in {-2,-0.667,0.667,2}
\filldraw(\i,\j) circle(1.5pt);
}
\draw (-0,-2) coordinate(L3) node[left]{$\flA_3$};
\draw (-0,-0.667) coordinate(L2) node[left]{$\flA_2$};
\draw (-0,0.667) coordinate(L1) node[left]{$\flA_1$};
\draw (-0,2) coordinate(L0) node[left]{$\flA_0$};
\draw (4,-2) coordinate(R0) node[right]{$\flA^0$};
\draw (4,-0.667) coordinate(R1) node[right]{$\flA^1$};
\draw (4,0.667) coordinate(R2) node[right]{$\flA^2$};
\draw (4,2) coordinate(R3) node[right]{$\flA^3$};
\draw[dashed] (L3) -- (R0);
\draw[dashed] (L1) -- (R0);
\draw[dashed] (L0) -- (R1);
\draw[dashed] (L0) -- (R3);
\draw[dashed] (R1) -- (L1);
\draw[dashed] (R0) -- (L2);
\draw[dashed] (R2) -- (L0);
\draw[orange] (R1) -- (L1);
\node[left,scale=0.9] at (0,-1.333) {$r_1^\ast$};
\node[left,scale=0.9] at (0,0) {$r_2^\ast$};
\node[left,scale=0.9] at (0,1.333) {$r_1^\ast$};
\node[right,scale=0.9] at (4,-1.333) {$r_1$};
\node[right,scale=0.9] at (4,0) {$r_2$};
\node[right,scale=0.9] at (4,1.333) {$r_1$};
\draw(2,-3) node[draw,rectangle,rounded corners=0.3em]{$1,2,{\color{myorange}\underline{{\color{black}\overline{1},1}}},\overline{2},\overline{1}$};
\end{scope}
\begin{scope}[xshift=14.5cm,yshift=-6.5cm,>=latex]
\draw[thick,|->] (-2,0) -- (-1,0);
\draw[thick] (0,2) -- (0,-2);
\draw[thick] (4,-2) -- (4,2);
\foreach \i in {0,4}
{
\foreach \j in {-2,-0.667,0.667,2}
\filldraw(\i,\j) circle(1.5pt);
}
\draw (-0,-2) coordinate(L3) node[left]{$\flA_3$};
\draw (-0,-0.667) coordinate(L2) node[left]{$\flA_2$};
\draw (-0,0.667) coordinate(L1) node[left]{$\flA_1$};
\draw (-0,2) coordinate(L0) node[left]{$\flA_0$};
\draw (4,-2) coordinate(R0) node[right]{$\flA^0$};
\draw (4,-0.667) coordinate(R1) node[right]{$\flA^1$};
\draw (4,0.667) coordinate(R2) node[right]{$\flA^2$};
\draw (4,2) coordinate(R3) node[right]{$\flA^3$};
\draw[dashed] (L3) -- (R0);
\draw[dashed] (L1) -- (R0);
\draw[dashed] (L0) -- (R1);
\draw[dashed] (L0) -- (R3);
\draw[dashed] (R0) -- (L0);
\draw[dashed] (R0) -- (L2);
\draw[dashed] (R2) -- (L0);
\node[left,scale=0.9] at (0,-1.333) {$r_1^\ast$};
\node[left,scale=0.9] at (0,0) {$r_2^\ast$};
\node[left,scale=0.9] at (0,1.333) {$r_1^\ast$};
\node[right,scale=0.9] at (4,-1.333) {$r_1$};
\node[right,scale=0.9] at (4,0) {$r_2$};
\node[right,scale=0.9] at (4,1.333) {$r_1$};
\draw (2,-3) node[draw,rectangle,rounded corners=0.3em]{$1,2,1,\overline{1},\overline{2},\overline{1}$};
\end{scope}
\end{tikzpicture}
    \caption{The transformations of dashed lines connecting the decorated flags corresponding to the mutation sequences in \cref{fig:A2_clusters}. The flipped diagonals are shown in orange, which are corresponding to mutations.}
    \label{fig:A2_diagonals}
\end{figure}
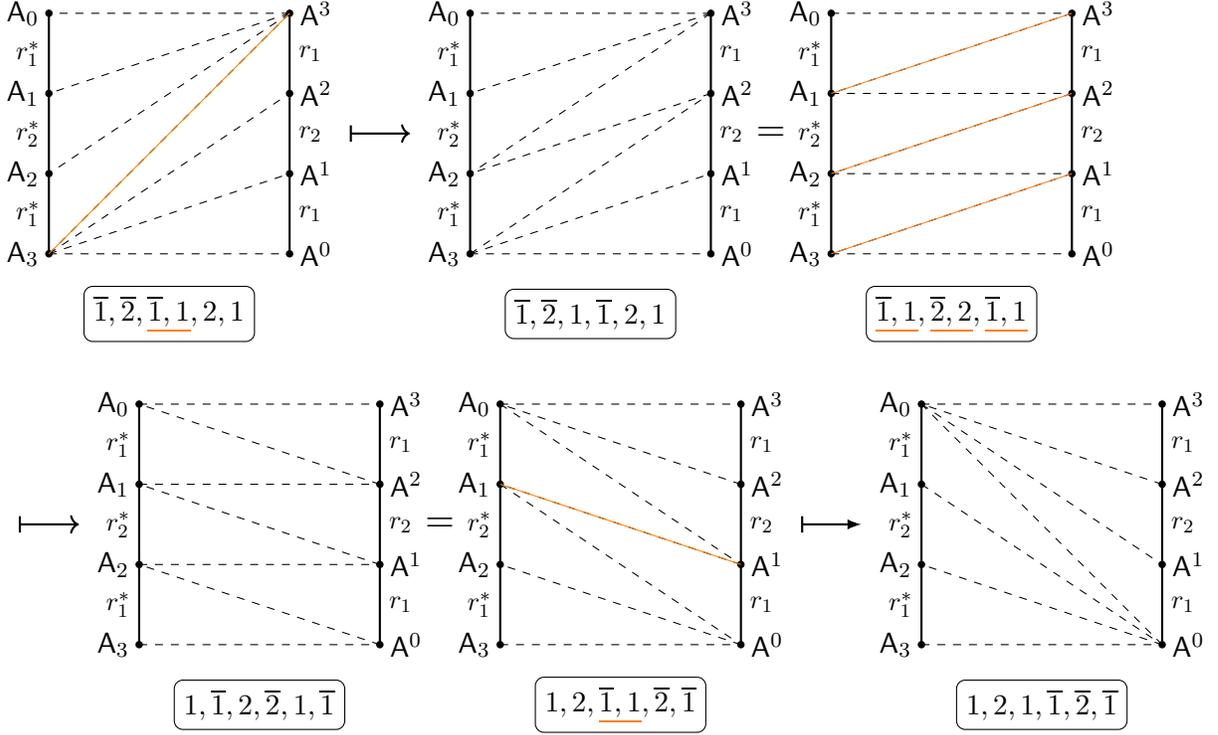

Recall the skein model studied in \cite{IY21}. For any marked surface as in \cref{subsec:marked_surface}, the first author and W.~Yuasa realized the quantum cluster algebra $\mathscr{A}^q_{\mathfrak{sl}_3,\Sigma}$ quantizing $\mathscr{A}_{\mathfrak{sl}_3,\Sigma}$ inside the skew-field of fractions of a certain $\mathfrak{sl}_3$-skein algebra $\mathscr{S}^q_{\mathfrak{sl}_3,\Sigma}$ consisting of $\mathfrak{sl}_3$-webs (\emph{i.e.}, oriented trivalent graphs whose vertices are either sinks
$
\tikz[baseline=-.6ex, scale=.1]{
    \draw[red,very thick,-<-] (0,0) -- (90:4);
    \draw[red,very thick, -<-] (0,0) -- (-30:4);
    \draw[red,very thick, -<-] (0,0) -- (210:4);
}
$
 or sources
 $
 \tikz[baseline=-.6ex, scale=.1]{
    \draw[red,very thick, ->-={.8}{red}] (0,0) -- (90:4);
    \draw[red,very thick, ->-={.8}{red}] (0,0) -- (-30:4);
    \draw[red,very thick, ->-={.8}{red}] (0,0) -- (210:4);
 }
 $), and shown the inclusion $\mathscr{S}^q_{\mathfrak{sl}_3,\Sigma}[\partial^{-1}] \subseteq \mathscr{A}^q_{\mathfrak{sl}_3,\Sigma}$. 
Here $\mathscr{S}^q_{\mathfrak{sl}_3,\Sigma}[\partial^{-1}]$ denote the localized skein algebra along the boundary $\mathfrak{sl}_3$-webs. It implies $\mathscr{S}^1_{\mathfrak{sl}_3,\Sigma}[\partial^{-1}] \subseteq \mathscr{A}_{\mathfrak{sl}_3,\Sigma}$ at the classical specialization $q=1 \in \C$. 

On the other hand, if the marked surface $\Sigma$ has at least two marked points, then we have $\mathscr{A}_{\mathfrak{sl}_3,\Sigma} = \mathscr{U}_{\mathfrak{sl}_3,\Sigma} = \cO(\A^\times_{SL_3,\Sigma})$
by \cref{main theorem}. Moreover, one can verify that each entry of the matrix \eqref{eq:A2_matrix_cluster} comes from the localized skein algebra $\mathscr{S}^1_{\mathfrak{sl}_3,\Sigma}[\partial^{-1}]$ by comparing the clusters in \cref{fig:A2_clusters} and the construction in \cite[Section 5]{IY21}. Explicitly, we have

\begin{align*}
    g_{[c]} = \begin{pmatrix}
    \tikz[scale=1.1]{
    \smallsq;
    \draw[myblue,thick,->-](0,0) to[bend right=20] (0,1); 
    \draw[myblue,thick,->-] (1,0) to[bend left=20] (1,1);
    \draw[red,thick,->-](0,0) -- (1,1);} & 
    \tikz[scale=1.1]{
    \smallsq;
    \draw[myblue,thick,->-](0,0) to[bend right=20] (0,1); 
	\draw[myblue,thick,-<-] (1,0) to[bend left=20] (1,1);
    \triv{0,0}{1,0}{1,1};} &
    \tikz[scale=1.1]{
    \smallsq;
    \draw[myblue,thick,->-](0,0) to[bend right=20] (0,1); 
    \draw[red,thick,->-](0,0) to[bend left=20] (1,0);}\vspace{2mm} \\
    \tikz[scale=1.1]{
    \smallsq;
	\draw[myblue,thick,-<-={0.6}{}](0,0) to[bend right=20] (0,1); 
    \draw[myblue,thick,->-] (1,0) to[bend left=20] (1,1);
    \trivop{0,0}{0,1}{1,1};} &
    \tikz[scale=1.1]{
    \smallsq;
    \draw[myblue,thick,-<-](0,0) to[bend right=20] (0,1); 
    \draw[myblue,thick,-<-] (1,0) to[bend left=20] (1,1);
	\draw[red,thick,-<-={0.7}{}] (0,0) -- (0.35,0.5);
	\draw[red,thick,-<-={0.7}{}] (0,1) -- (0.35,0.5);
	\draw[red,thick,->-={0.7}{}] (1,0) -- (0.65,0.5);
	\draw[red,thick,->-={0.7}{}] (1,1) -- (0.65,0.5);
	\draw[red,thick,-<-] (0.65,0.5) -- (0.35,0.5);} &
    \tikz[scale=1.1]{
    \smallsq;
    \draw[myblue,thick,-<-={0.4}{}](0,0) to[bend right=20] (0,1); 
    \trivop{0,0}{0,1}{1,0};}\vspace{2mm}\\
    \tikz[scale=1.1]{
    \smallsq;
    \draw[myblue,thick,->-] (1,0) to[bend left=20] (1,1);
    \draw[red,thick,->-](0,1) to[bend right=20] (1,1);} &
    \tikz[scale=1.1]{
    \smallsq;
	\draw[myblue,thick,-<-] (1,0) to[bend left=20] (1,1);
    \triv{1,0}{0,1}{1,1};} &
    \tikz[scale=1.1]{
    \smallsq;
    \draw[red,thick,->-](0,1) -- (1,0);}
    \end{pmatrix}.
\end{align*}
Here the inverse of a boundary $\mathfrak{sl}_3$-web is shown in blue. 
Then by the same line of argument as \cref{subsec:proof}, we see that the inclusion $\mathscr{U}_{\mathfrak{sl}_3,\Sigma} \subset \mathscr{S}^1_{\mathfrak{sl}_3,\Sigma}[\partial^{-1}] $ holds. Thus we get $\mathscr{S}^1_{\mathfrak{sl}_3,\Sigma}[\partial^{-1}] =\mathscr{A}_{\mathfrak{sl}_3,\Sigma} = \mathscr{U}_{\mathfrak{sl}_3,\Sigma}$, which confirms \cite[Conjecture 3]{IY21} at the classical level. 

\subsection{$Sp_4$-case}
The vector representation $V(\varpi_1)=\mathbb{C}^4$ of $Sp_4$ is minuscule. The weights in this representation are given by
\begin{align*}
    \varpi_1, \quad r_1\varpi_1=\varpi_2 -\varpi_1, \quad r_2r_1\varpi_1=\varpi_1-\varpi_2, \quad w_0\varpi_1=-\varpi_1.
\end{align*}
By convention, we choose the symplectic form to be 
\begin{align*}
    J:=\begin{pmatrix}
       &      &     &    1\\
       &      &   -1&     \\
       &     1&     & \\
     -1&      &     &     
    \end{pmatrix}
\end{align*}
which determines the vector representation $Sp_4 \cong Sp(\mathbb{C}^4,J)$. In this case, we need to use both of the reduced words $\bs^\bullet=(1,2,1,2)$ and $\hat{\bs}^\bullet=(2,1,2,1)$ of $w_0 \in W$ in order to obtain all the matrix coefficients in $V(\varpi_1)$. They give rise to the chains of decorated flags
\begin{align*}
    &\flA^L=\flA^4 \xrightarrow{r_2} \flA^3 \xrightarrow{r_1} \flA^2 \xrightarrow{r_2} \flA^1 \xrightarrow{r_1} \flA^0=\flA^R,\\
    &\flA^L=\hat{\flA}^4 \xrightarrow{r_1} \hat{\flA}^3 \xrightarrow{r_2} \hat{\flA}^2 \xrightarrow{r_1} \hat{\flA}^1 \xrightarrow{r_2} \hat{\flA}^0=\flA^R, 
\end{align*}
respectively. Similarly, the words $\bs_\bullet=(1,2,1,2)$ and $\hat{\bs}_\bullet=(2,1,2,1)$ give rise to the bottom chains $(\flA_l)$ and $(\hat{\flA_l})$, respectively. Similarly to the $SL_2$- and $SL_3$-cases, we can compute the Wilson line matrix $g_{[c]} \in Sp_4(\cO(\A_{Sp_4,\Sigma}^\times))$ and get
\begin{align}\label{eq:C2_matrix_cluster}
    g_{[c]} = \begin{pmatrix}
    \displaystyle\frac{\Delta_1(\flA^4,\flA_4)}{\clA_\inn^1 \clA_\out^1} & 
    \displaystyle\frac{\Delta_1(\hat{\flA}^3,\flA_4)}{\clA_\inn^1 \clA_\out^2} & 
    \displaystyle\frac{\Delta_1(\flA^2,\flA_4)}{\clA_\inn^1\clA_\out^1} &
    \displaystyle\frac{\Delta_1(\flA^0,\flA_4)}{\clA_\inn^1}\vspace{3mm} \\
    \displaystyle\frac{\Delta_1(\flA^4,\hat{\flA}_3)}{\clA_\inn^2 \clA_\out^1} & 
    \displaystyle\frac{\Delta_1(\hat{\flA}^3,\hat{\flA}_3)}{\clA_\inn^2 \clA_\out^2} & 
    \displaystyle\frac{\Delta_1(\flA^2,\hat{\flA}_3)}{\clA_\inn^2\clA_\out^1} &
    \displaystyle\frac{\Delta_1(\flA^0,\hat{\flA}_3)}{\clA_\inn^2}\vspace{3mm} \\
    \displaystyle\frac{\Delta_1(\flA^4,\flA_2)}{\clA_\inn^1 \clA_\out^1} & 
    \displaystyle\frac{\Delta_1(\hat{\flA}^3,\flA_2)}{\clA_\inn^1 \clA_\out^2} & 
    \displaystyle\frac{\Delta_1(\flA^2,\flA_2)}{\clA_\inn^1\clA_\out^1} &
    \displaystyle\frac{\Delta_1(\flA^0,\flA_2)}{\clA_\inn^1}\vspace{3mm}\\
    \displaystyle\frac{\Delta_1(\flA^4,\flA_0)}{ \clA_\out^1} & 
    \displaystyle\frac{\Delta_1(\hat{\flA}^3,\flA_0)}{\clA_\out^2} & 
    \displaystyle\frac{\Delta_1(\flA^2,\flA_0)}{\clA_\out^1} &
    \displaystyle\Delta_1(\flA^0,\flA_0)
    \end{pmatrix}.
\end{align}
The investigation of the clusters to which the cluster variables appearing in this expression belong is left to the reader. Apart from the $A_n$-cases, the choice of distinguished vertices $v_T$ of a triangulation is crucial: we remark here that the cluster for the double reduced word $\bs=(\overline{1},\overline{2},\overline{1},\overline{2},2,1,2,1)$ corresponds to the decorated triangulation shown in the left of \cref{fig:C2_initial_cluster}, whose underlying weighted quiver is shown in the right (see \cite[Appendix C]{IO20} for our convention on weighted quivers). 

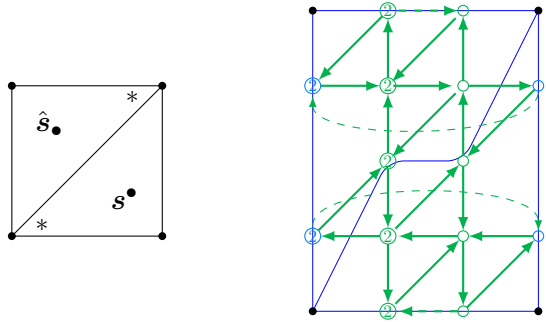
\begin{figure}[ht]
\centering
\begin{tikzpicture}
\draw (0,1) -- (2,1) -- (2,3) -- (0,3) --cycle;
\draw (0,1) -- (2,3);
\foreach \i in {0,2} \foreach \j in {1,3} \fill(\i,\j) circle(1.5pt);
\node[scale=0.9] at (.4,1.15) {$\ast$};
\node[scale=0.9] at (2-.4,3-.15) {$\ast$};
\node at (1.5,1.5){$\bs^\bullet$};
\node at (.5,2.5){$\hat{\bs}_\bullet$};
{\begin{scope}[xshift=7cm,rotate=90]
\draw[blue,very thin] (0,0) -- (4,0) -- (4,3) -- (0,3) --cycle;
\draw[blue,very thin,rounded corners=6pt] (4,0) -- (2,1) -- (2,2) -- (0,3);
\foreach \i in {0,4} \foreach \j in {0,3} \fill(\i,\j) circle(1.5pt);
\quiversquareC{0,0}{4,0}{4,3}{0,3};
\begin{scope}[>=latex]
{\color{mygreen}
\qsarrow{v21}{v20};
\qsarrow{v22}{v21};
\qsarrow{v22}{v23};
\qsarrow{v23}{v24};
\qarrow{v11}{v10};
\qarrow{v12}{v11};
\qarrow{v12}{v13};
\qarrow{v13}{v14};
\qstarrow{v20}{v11};
\qsharrow{v11}{v21};
\qstarrow{v21}{v12};
\qstarrow{v14}{v23};
\qsharrow{v23}{v13};
\qstarrow{v13}{v22};
\qshdarrow{v10}{v20};
\qstdarrow{v24}{v14};
\qarrow{v10}{yl};
\qarrow{yl}{v11};
\qarrow{v13}{yr};
\qarrow{yr}{v12};
\qsarrow{v21}{zl};
\qsarrow{zl}{v22};
\qsarrow{v24}{zr};
\qsarrow{zr}{v23};
\draw[dashed,->,shorten >=2pt,shorten <=4pt] (zl) ..controls ++(0.75,0) and ($(yl)+(0.75,0)$).. (yl);
\draw[dashed,<-,shorten >=2pt,shorten <=4pt] (zr) ..controls ++(-0.75,0) and ($(yr)+(-0.75,0)$).. (yr);
}
\end{scope}
\end{scope}}
\end{tikzpicture}
    \caption{The decorated triangulation corresponding to the double reduced word $\bs=(\overline{1},\overline{2},\overline{1},\overline{2},2,1,2,1)$ (left) and the underlying weighted quiver (right).}
    \label{fig:C2_initial_cluster}
\end{figure}

The skein model for the case is studied in \cite{IY_C2}. For any marked surface, the quantum cluster algebra $\mathscr{A}^q_{\mathfrak{sp}_4,\Sigma}$ quantizing $\mathscr{A}_{\mathfrak{sp}_4,\Sigma}$ is realized inside the skew-field of fractions of a certain $\mathfrak{sp}_4$-skein algebra 
consisting of $\mathfrak{sp}_4$-webs. 
An $\mathfrak{sp}_4$-web is represented by a trivalent graph with two types of edges 
$
\tikz[baseline=-.6ex, scale=.1]{
    \draw[webline] (-3,0) -- (3,0);
}
$ and
$
\tikz[baseline=-.6ex, scale=.1]{
    \draw[wline] (-3,0) -- (3,0);
}
$ (corresponding to the fundamental representations $V(\varpi_1)$ and $V(\varpi_2)$ respectively), and with trivalent vertices of the form
$
\tikz[baseline=-.6ex, scale=.1]{
    \draw[wline] (0,0) -- (90:3);
    \draw[webline] (0,0) -- (-30:3);
    \draw[webline] (0,0) -- (210:3);
}
$ in the interior. 

Similarly to the $\mathfrak{sl}_3$-case, we have an inclusion $\mathscr{S}^q_{\mathfrak{sp}_4,\Sigma}[\partial^{-1}] \subseteq \mathscr{A}^q_{\mathfrak{sp}_4,\Sigma}$ after localizing along the boundary $\mathfrak{sp}_4$-webs. 
Here $\mathscr{S}^q_{\mathfrak{sp}_4,\Sigma}$ denotes the \emph{$\mathbb{Z}_q$-form} of the $\mathfrak{sp}_4$-skein algebra introduced in \cite{IY_C2} (namely, $\mathscr{S}^{\mathbb{Z}_q}_{\mathfrak{sp}_4,\Sigma}$ in the notation there). 
It implies $\mathscr{S}^1_{\mathfrak{sp}_4,\Sigma}[\partial^{-1}] \subseteq \mathscr{A}_{\mathfrak{sp}_4,\Sigma}$ at the classical specialization $q=1 \in \C$. 

Again, each entry of the matrix \eqref{eq:A2_matrix_cluster} comes from the localized skein algebra $\mathscr{S}^1_{\mathfrak{sp}_4,\Sigma}[\partial^{-1}]$. Explicitly, we have
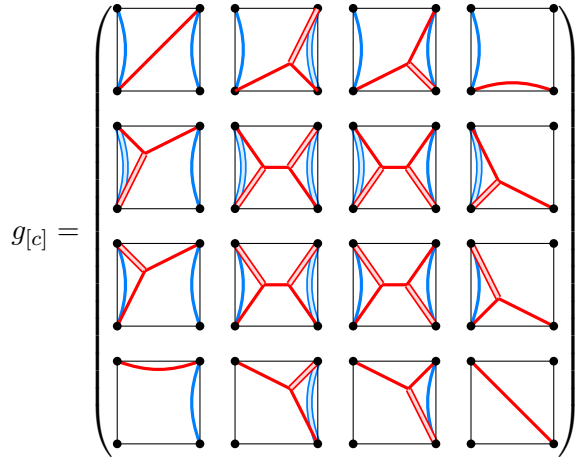
\begin{align*}
    g_{[c]} = \begin{pmatrix}
    \tikz[scale=1.1]{
    \draw[myblue,very thick](0,0) to[bend right=20] (0,1); 
    \draw[myblue,very thick] (1,0) to[bend left=20] (1,1);
    \draw[red,very thick](0,0) -- (1,1);
    \smallsq;} & 
    \tikz[scale=1.1]{
    \draw[myblue,very thick](0,0) to[bend right=20] (0,1); 
	\draw[myblue,wlinebdy] (1,0) to[bend left=20] (1,1);
	\smallsq;
    \trivL{0,0}{1,0}{1,1};} &
				\tikz[scale=1.1]{
    \draw[myblue,very thick](0,0) to[bend right=20] (0,1); 
				\draw[myblue,very thick] (1,0) to[bend left=20] (1,1);
    \trivR{0,0}{1,0}{1,1};
    \smallsq;} &
    \tikz[scale=1.1]{
    \draw[myblue,very thick](0,0) to[bend right=20] (0,1); 
    \draw[red,very thick](0,0) to[bend left=20] (1,0);
    \smallsq;}\vspace{2mm} \\
    \tikz[scale=1.1]{
	\draw[myblue,wlinebdy](0,0) to[bend right=20] (0,1); 
    \draw[myblue,very thick] (1,0) to[bend left=20] (1,1);
    \trivL{1,1}{0,1}{0,0};
    \smallsq;} &
    \tikz[scale=1.1]{
    \draw[myblue,wlinebdy](0,0) to[bend right=20] (0,1); 
    \draw[myblue,wlinebdy] (1,0) to[bend left=20] (1,1);
	\draw[red,wline] (0,0) -- (0.35,0.5);
	\draw[red,webline] (0,1) -- (0.35,0.5);
	\draw[red,webline] (1,0) -- (0.65,0.5);
	\draw[red,wline] (1,1) -- (0.65,0.5);
	\draw[red,webline] (0.65,0.5) -- (0.35,0.5);
    \smallsq;} &
    \tikz[scale=1.1]{
    \draw[myblue,wlinebdy](0,0) to[bend right=20] (0,1); 
    \draw[myblue,very thick](1,0) to[bend left=20] (1,1);
    \draw[red,wline] (0,0) -- (0.35,0.5);
	\draw[red,webline] (0,1) -- (0.35,0.5);
	\draw[red,wline] (1,0) -- (0.65,0.5);
	\draw[red,webline] (1,1) -- (0.65,0.5);
	\draw[red,webline] (0.65,0.5) -- (0.35,0.5);
    \smallsq;} &
	\tikz[scale=1.1]{
    \draw[myblue,wlinebdy](0,0) to[bend right=20] (0,1); 
    \trivL{1,0}{0,1}{0,0};
    \smallsq;}\vspace{2mm}\\
    \tikz[scale=1.1]{
    \draw[myblue,very thick] (1,0) to[bend left=20] (1,1);
    \draw[myblue,very thick](0,0) to[bend right=20] (0,1);
	\trivR{1,1}{0,1}{0,0};
    \smallsq;} &
    \tikz[scale=1.1]{
    \draw[myblue,very thick](0,0) to[bend right=20] (0,1);
	\draw[myblue,wlinebdy] (1,0) to[bend left=20] (1,1);
    \draw[red,webline] (0,0) -- (0.35,0.5);
	\draw[red,wline] (0,1) -- (0.35,0.5);
	\draw[red,webline] (1,0) -- (0.65,0.5);
	\draw[red,wline] (1,1) -- (0.65,0.5);
	\draw[red,webline] (0.65,0.5) -- (0.35,0.5);
    \smallsq;} &
    \tikz[scale=1.1]{
    \draw[myblue,very thick](0,0) to[bend right=20] (0,1);
	\draw[myblue,very thick] (1,0) to[bend left=20] (1,1);
    \draw[red,webline] (0,0) -- (0.35,0.5);
	\draw[red,wline] (0,1) -- (0.35,0.5);
	\draw[red,wline] (1,0) -- (0.65,0.5);
	\draw[red,webline] (1,1) -- (0.65,0.5);
	\draw[red,webline] (0.65,0.5) -- (0.35,0.5);
    \smallsq;} &
    \tikz[scale=1.1]{
    \draw[myblue,very thick](0,0) to[bend right=20] (0,1); 
    \trivR{1,0}{0,1}{0,0};
    \smallsq;}\vspace{2mm}\\
    \tikz[scale=1.1]{
    \draw[myblue,very thick] (1,0) to[bend left=20] (1,1);
	\draw[red,very thick](0,1) to[bend right=20] (1,1);
    \smallsq;} &
    \tikz[scale=1.1]{
	\draw[myblue,wlinebdy] (1,0) to[bend left=20] (1,1);
    \trivL{0,1}{1,0}{1,1};
    \smallsq;} &
    \tikz[scale=1.1]{
	\draw[myblue,very thick] (1,0) to[bend left=20] (1,1);
    \trivR{0,1}{1,0}{1,1};
    \smallsq;} &
    \tikz[scale=1.1]{
    \draw[red,very thick] (0,1) -- (1,0);
    \smallsq;}
    \end{pmatrix}.
\end{align*}
Then by exactly the same line of argument as in the $\mathfrak{sl}_3$-case, we get $\mathscr{S}^1_{\mathfrak{sp}_4,\Sigma}[\partial^{-1}]=\mathscr{A}_{\mathfrak{sp}_4,\Sigma} = \mathscr{U}_{\mathfrak{sp}_4,\Sigma}$.
\section{Supplements to \cref{subsec:U=O}}

\subsection{Amalgamation of upper cluster algebras}\label{app:amalgamation}


\begin{lem}\label{lem:remove_isolated}
Let $\mathscr{U}$ be an upper cluster algebra having a cluster $\mathbf{i}=(\{A_i\}_{i \in I},\ve)$ with \emph{isolated variables}, namely a subset $J \subset I$ such that $\ve_{ij}=0$ for all $j \in J$ and $i \in I$. Then the quotient of $\mathscr{U}$ by the ideal generated by $A_j-1$ for $j \in J$ is the upper cluster algebra $\mathscr{U}'$ having the seed obtained from $\mathbf{i}$ by deleting the data for $j \in J$. 
\end{lem}

\begin{proof}
It is easy because $\mathscr{U}=\mathscr{U}' \otimes \C[A_j^{\pm 1} \mid j \in J]$ in this case. 
\end{proof}

Fix a decorated triangulation $\boldsymbol{\mathcal{T}}$ of $\Sigma$. Let $\Lambda\subset e(\mathcal{T})$ be a subset of diagonals  so that we obtain a collection of disks with marked points when cutting along them. 
Let $\mathcal{A}_{G,\Sigma}^\times[\Lambda]\subset \mathcal{A}_{G,\Sigma}^\times$ be the open subspaces such that for every edge in $\Lambda$, its associated pair of decorated flags is generic. 
Let ${\bf i}_{\boldsymbol{\mathcal{T}}}[\Lambda]$ be the seed obtained from $\boldsymbol{\mathcal{T}}$ by freezing all the mutable vertices that are placed on the diagonals in $\Lambda$. 
We are going to prove:
\begin{prop}\label{equality}
$\mathcal{O}(\mathcal{A}_{G,\Sigma}^\times[\Lambda])= \mathscr{U}({\bf i}_{\boldsymbol{\mathcal{T}}}[\Lambda])$.
\end{prop}
\cref{local.iso.q} is covered as the special case $\Lambda=e(\mathcal{T})\setminus \{E\}$.

\begin{proof}
Let $N$ be the number of disks obtained  when  cutting $\Sigma$  along the diagonals in $\Lambda$. We shall prove  \cref{equality} via induction on  $N \geq 0$. The case $N=0$ is trivial. For $N\geq 1$, let $D$ be an obtained disk that  contains a boundary edge of $\Sigma$. Let $\Sigma'\subset \Sigma$ be the marked surface obtained by cutting off $D$ from $\Sigma$.  The decorated triangulation $\boldsymbol{\mathcal{T}}$ of $\Sigma$ naturally induces that of $\Sigma'$, which is denoted by $\boldsymbol{\mathcal{T}}'$. Similarly, we define the spaces $\mathcal{A}_{G,\Sigma'}^\times[\Lambda]\subset \mathcal{A}_{G,\Sigma'}^\times$ and the seed ${\bf i}_{\mathcal{T}'}[\Lambda]$. By induction, we  have 
\[
\mathcal{O}(\mathcal{A}_{G,\Sigma'}^\times[\Lambda])= \mathscr{U}({\bf i}_{\boldsymbol{\mathcal{T}'}}[\Lambda]).
\]

We label the marked points of $D$ from $1$ to $n$ in clockwise order, with the edge $\{1,n\}$ being a boundary interval of $\Sigma$. For $i=2,\ldots, n$, we color the edge $\{i-1,i\}$ of $D$ red if it is glued with either an edge of $\Sigma'$ or another edge $\{j-1,j\}$ of $D$ with $j<i$. Let $J$ be the set of indices $i$ such that $\{i-1,i\}$ is red.  See \cref{fig:decoration} for an illustration of the coloring scheme.

\begin{figure}[ht]
    \centering
\begin{tikzpicture}[scale=1.5]
\draw(120:1) -- (180:1);
\draw (300:1) -- (0:1);
\draw[red] (0:1) -- (60:1) -- (120:1);
\draw[red] (180:1) -- (240:1) -- (300:1);
\fill(0:1) circle(1.5pt) node[right]{$\flA_1$};
\fill(60:1) circle(1.5pt) node[above]{$\flA_2$};
\fill(120:1) circle(1.5pt) node[above]{$\flA_3$};
\fill(180:1) circle(1.5pt) node[left]{$\flA_4$};
\fill(240:1) circle(1.5pt) node[below] {$\flA_5$};
\fill(300:1) circle(1.5pt) node[below]{$\flA_6$};
\draw[thick,|->] (2.5,0) -- (3.5,0);
\begin{scope}[xshift=6cm]
\draw(120:1) -- (180:1);
\draw (300:1) -- (0:1);
\draw[red] (0:1) -- (60:1) -- (120:1);
\draw[red] (180:1) -- (240:1) -- (300:1);
\fill(0:1) circle(1.5pt) node[right]{$\flA_1$};
\fill(60:1) circle(1.5pt) node[above]{$\flA_2.h_2$};
\fill(120:1) circle(1.5pt) node[above]{$\flA_3.h_3$};
\fill(180:1) circle(1.5pt) node[left]{$\flA_4$};
\fill(240:1) circle(1.5pt) node[below] {$\flA_5.h_5$};
\fill(300:1) circle(1.5pt) node[below]{$\flA_6.h_6$};
\end{scope}
\end{tikzpicture}
\caption{The change of decorations with $n=6$ and $J=\{2,3,5,6\}$.}
\label{fig:decoration}
\end{figure}
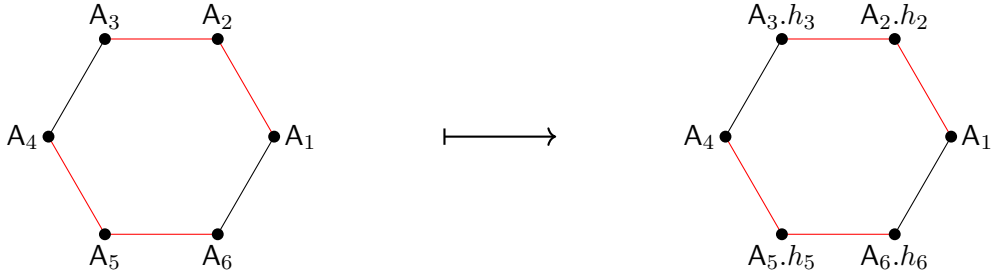

We define the cutting map
\begin{align*}
    \phi: ~ &\A_{G,\Sigma}^\times[\Lambda] \times H^{J} \longrightarrow
    \A_{G,\Sigma'}^\times[\Lambda] \times  \A_{G,D}^\times, \\
    &([\mathcal{L},\alpha], \{h_j\}_{j\in J}) \longmapsto ([\mathcal{L}',\alpha'], (\flA_1', \ldots, \flA_n')).
\end{align*}
Here $[\mathcal{L}',\alpha']$ is the restriction of  decorated twisted local system $[\mathcal{L},\alpha]$ from $\Sigma$ to $\Sigma'$. When restricting $[\mathcal{L},\alpha]$ to $D$, we obtain a $n$-tuple $(\flA_1, \ldots, \flA_n)$, and set 
\[
\flA_j':=\left\{
\begin{array}{ll}
   \flA_j.h_j  & \mbox{if~} j\in J;  \\
    \flA_j & \mbox{otherwise}.
\end{array}\right.
\]
Conversely, given the data $([\mathcal{L}',\alpha'], (\flA_1', \ldots, \flA_n'))$, one can recursively determine $h_j$ for $j\in J$ such that after rescaling each $\flA_j'$ by $h_j^{-1}$, the $h$-distance for the pair of decorated flags for every red edge coincides with the one that it is glued with. Gluing them back produces the preimage  of $([\mathcal{L}',\alpha'], (\flA_1', \ldots, \flA_n'))$ before $\phi$. For example, \cref{fig:splitting0} is the gluing map along one edge. Hence, the map $\phi$ is an isomorphism.

\begin{figure}[ht]
    \centering
\begin{tikzpicture}
\draw(0,2) -- (-1, 2) -- (-2,1) -- (-1,0) -- (0,0);
\draw[red] (0,0) -- (0,2);
\draw(1,2) -- (3,2);
\draw(1,0) -- (3,0);
\draw[red] (1,0) -- (1,2);
\fill(0,0) circle(1.5pt) node[below]{$\flA'_2$};
\fill(0,2) circle(1.5pt) node[above]{$\flA'_1$};
\fill(-2,1) circle(1.5pt) node[left]{$\flA'_4$};
\fill(-1,2) circle(1.5pt) node[above]{$\flA'_5$};
\fill(-1,0) circle(1.5pt) node[below]{$\flA'_3$};
\fill(1,0) circle(1.5pt) node[below]{$\flA_2$};
\fill(1,2) circle(1.5pt) node[above]{$\flA_1$};
\node at (-0.8,1) {$D$};
\node at (2,1) {$\Sigma'$};
\node[red] at (0.25,0.5) {$E'$};
\node[red] at (0.75,1.5) {$E$};
\draw[|->] (4,1) --node[midway,above]{$\phi^{-1}$} (5,1);
\begin{scope}[xshift=8.5cm]
\draw(0,2) -- (-1, 2) -- (-2,1) -- (-1,0) -- (0,0) -- (2,0);
\draw(0,2) -- (2,2);
\draw[red] (0,2) -- (0,0);
\fill(0,0) circle(1.5pt) node[below]{$\flA_2$};
\fill(0,2) circle(1.5pt) node[above]{$\flA_1$};
\fill(-2,1) circle(1.5pt) node[left]{$g.\flA'_4$};
\fill(-1,2) circle(1.5pt) node[above]{$g.\flA'_5$};
\fill(-1,0) circle(1.5pt) node[below]{$g.\flA'_3$};
\node at (-0.8,1) {$D$};
\node at (1,1) {$\Sigma'$};
\node[red] at (0.25,0.5) {$E$};
\node at (3,1) {$,\ h$};
\end{scope}
\end{tikzpicture}
    \caption{The gluing map $\phi^{-1}$. Here $h \in H$ is uniquely chosen so that $h(\flA'_1,\flA'_2.h^{-1}) = h(\flA_1,\flA_2)$. Then $g.(\flA'_1,\flA'_2.h^{-1}) = (\flA_1,\flA_2)$ for a unique $g \in G$.}
    \label{fig:splitting0}
\end{figure}
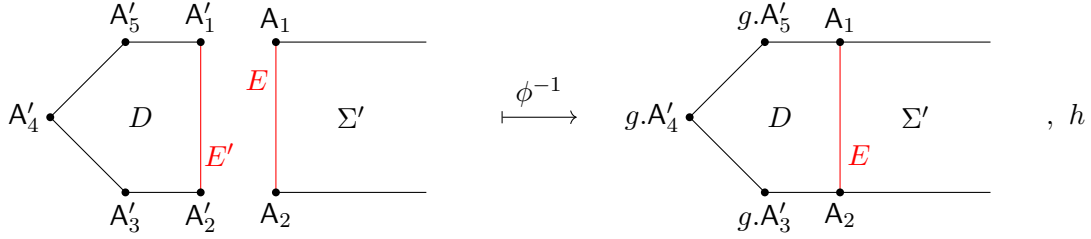

The decorated triangulation $\boldsymbol{\mathcal{T}}$ induces a decorated triangulation of $D$, which further gives rise to a cluster seed ${\bf i}_D$ of $\mathcal{A}_{G, D}^\times$. By Lemma \ref{disk.case}, we have $\mathcal{O}(\mathcal{A}_{G, D}^\times)= \mathscr{U}({\bf i}_D)$.
By \cite[Theorem 9.17]{GS19}, every cluster variable $A$ in ${\bf i}_D$ is homogeneous with respect to the $H^n$-actions. Denote by ${\lambda}(A):=(\lambda_1, \ldots, \lambda_n)$ its weight, so that
\[
A(\flA_1.h_1, \ldots, \flA_n.h_n)= A(\flA_1, \ldots, \flA_n)\cdot \prod_{i=1}^n h_i^{\lambda_i}.
\]
The coordinate ring $\cO(H^J)$ can be regarded as the upper cluster algebra $\mathscr{U}_m^{\mathrm{triv}}$ with isolated $m:=|S|\cdot |J|$ vertices and cluster variables $h_j^{\varpi_s}$ for $j\in J$ and $s \in S$. 

Let $A$ be a cluster variable in $\mathbf{i}_{\bm{\mathcal{T}'}}[\Lambda] \sqcup \mathbf{i}_{D}$. Under the induced isomorphism 
\begin{align}\label{eq:splitting_alg_hom}
    \phi^\ast:    \cO(\A_{G,\Sigma'}^\times[\Lambda]) \otimes  \cO(\A_{G,D}^\times) \xrightarrow{\sim}
    \cO(\A_{G,\Sigma}^\times[\Lambda]) \otimes \cO(H^{J}), 
\end{align}
we get
\begin{align*}
    &\phi^\ast (A) = \begin{cases}
    A & \mbox{if $A$ belong to $\mathbf{i}_{\bm{\mathcal{T}'}}[\Lambda]$}, \\
    A\cdot \prod_{j\in J}h_j^{\lambda(A)_j} & \mbox{if $A$ belongs to $\mathbf{i}_{D}$}.
    \end{cases} 
\end{align*}
 
We claim that $\phi^\ast$ is a quasi-isomorphism in the sense of Fraser \cite{Fraser}. Namely, $\phi^\ast$ rescales the cluster variables by Laurent monomials of frozen variables in such a way that the unfrozen cluster Poisson variables for algebras on both sides of \eqref{eq:splitting_alg_hom} associated with the seeds induced by $\boldsymbol{\mathcal{T}}$ are the same.
Indeed, following the geometric description of cluster Poisson charts on the moduli spaces $\P_{G,\Sigma}$ given in \cite{GS19}, the unfrozen cluster Poisson variables are invariant under the rescaling of decorated flags by $h\in H$. Therefore, the claim follows directly. 

As a consequence, the extension of $\phi^\ast$ to the fields of rational functions preserves the upper cluster algebras:
\begin{align}\label{eq:q-isom_upper}
    \phi^\ast:  \mathscr{U}( \mathbf{i}_{\bm{\mathcal{T}'}}[\Lambda] \sqcup \mathbf{i}_{D})  \xrightarrow{\sim}
    \mathscr{U}(\mathbf{i}_{\bm{\mathcal{T}}}[\Lambda]) \otimes \mathscr{U}_m^{\mathrm{triv}}.
\end{align}

By the induction assumption, we have $\mathscr{U}( \mathbf{i}_{\bm{\mathcal{T}'}}[\Lambda]\sqcup \mathbf{i}_{D} ) = \cO(\A_{G,\Sigma'}^\times[\Lambda]) \otimes  \cO(\A_{G,D}^\times)$. 
Therefore we get $\mathscr{U}(\mathbf{i}_{\bm{\mathcal{T}}}[\Lambda]) \otimes \mathscr{U}_m^{\mathrm{triv}} = \cO(\A_{G,\Sigma}^\times[\Lambda]) \otimes \cO(H^J)$ by \eqref{eq:splitting_alg_hom} and \eqref{eq:q-isom_upper}. 
Then by applying \cref{lem:remove_isolated} to the isolated variables $h_j^{\varpi_s}$ in $\mathscr{U}_m^\mathrm{triv}$, we get $\mathcal{O}(\mathcal{A}_{G,\Sigma}^\times[\Lambda])= \mathscr{U}({\bf i}_{\boldsymbol{\mathcal{T}}}[\Lambda])$ as desired.
\end{proof}

\subsection{Stratifications of $\A_{G,\Sigma}$}\label{app:codimension}
Our purpose is to prove that $\A_{G,\Sigma}^{\mathcal{T},E_1,E_2} \subset \A_{G,\Sigma}$ has codimension $\geq 2$. Along the way, we also obtain formulae for codimensions of the subspaces in more general stratifications of $\A_{G,\Sigma}$ obtained by prescribing $w$-distances along arcs and boundary intervals. Recall that the dimension of a quotient stack $\mathcal{X}=[X/G]$ is defined to be $\dim \mathcal{X}:=\dim X - \dim G$.

\begin{conv}
For an oriented ideal arc $\alpha$ on $\Sigma$, let $\flA_\alpha^+$ (resp. $\flA_\alpha^-$) denote the decorated flag assigned to the initial (resp. terminal) marked point of $\alpha$. We endow each boundary interval with the orientation induced from the boundary. 
\end{conv}

Fix two diagonals $E_1,E_2$ of an ideal triangulation $\mathcal{T}$, and endow them arbitrary orientations. For $u,v \in W$, let $\A_{G,\Sigma}^{u,v} \subset \A_{G,\Sigma}^\times$ be the subspace such that $w(\flA_{E_1}^+,\flA_{E_1}^-)=u$, $w(\flA_{E_2}^+,\flA_{E_2}^-)=v$. Then we have
\begin{align*}
    \A_{G,\Sigma}^{\mathcal{T},E_1,E_2} = \bigsqcup_{u\neq w_0;~v\neq w_0} \A_{G,\Sigma}^{u,v}.
\end{align*}

We are going to prove:
\begin{thm}\label{prop:codimension_count}
We have $\dim \A_{G,\Sigma}- \dim \A_{G,\Sigma}^{u,v} = 2l(w_0) - l(u) - l(v)$ for all $u,v \in W$. In particular, $\A_{G,\Sigma}^{\mathcal{T},E_1,E_2}$ has codimension $\geq 2$.
\end{thm}

It turns out to be useful to include the following more general subspaces into consideration:
\begin{dfn}\label{dfn:moduli_cells}\ 
\begin{itemize}
    \item Given a tuple $\bw=\{w_E\} \in W^{\bB(\Sigma)}$, let 
    \begin{align*}
        \A_{G,\Sigma}^{\bw} \subset \A_{G,\Sigma}
    \end{align*}
    denote the subspace such that $w(\flA_E^+,\flA_E^-)=w_E$ for each $E \in \bB(\Sigma)$. For $\bw_0=\{w_0\} \in W^{\bB(\Sigma)}$ that assigns the longest element to each boundary interval, we have $\A_{G,\Sigma}^\times=\A_{G,\Sigma}^{\bw_0}$. 
    \item Given a collection $\cC$ of disjoint  oriented ideal arcs in $\Sigma$ and a tuple $\bv_\cC=\{v_\alpha\} \in W^{\cC}$, let 
    \begin{align*}
        \A_{G,\Sigma}[\cC;\bv_\cC] \subset \A_{G,\Sigma}
    \end{align*}
    denote the subspace such that $w(\flA_\alpha^+,\flA_\alpha^-)=v_\alpha$ for each $\alpha \in \cC$. 
\end{itemize}
Furthermore, let us write $\A_{G,\Sigma}^{\bw}[\cC;\bv_\cC]:=\A_{G,\Sigma}^\bw \cap \A_{G,\Sigma}[\cC;\bv_\cC]$. The original subspace of our interest is $\A_{G,\Sigma}^{u,v}=\A_{G,\Sigma}^{\bw_0}[E_1,E_2;u,v]$.    
\end{dfn}




\begin{prop}\label{lem:codim_boundary}
Let $\bw=\{w_E\} \in W^{\bB(\Sigma)}$. 
If $\Sigma$ is a polygon, then we assume that there exists $E_0 \in \bB(\Sigma)$ such that $w_{E_0}=w_0$. Then
$\dim \A_{G,\Sigma} - \dim \A_{G,\Sigma}^{\bw} = \sum_{E \in \bB(\Sigma)} (l(w_0) - l(w_E))$. In particular, $\A_{G,\Sigma}^\times \subset \A_{G,\Sigma}$ is open dense. 
\end{prop}
It is proved in \cref{subsub:codim_boundary} below, based on a relation to the \emph{braid varieties}. 

Fix an oriented ideal arc $\alpha$ and an element $v_\alpha \in W$. 
Cutting the surface $\Sigma$ along $\alpha$, we obtain a new marked surface $\Sigma'$, where $\alpha$ is splitted into two boundary intervals $\alpha',\alpha''$, where $\alpha'$ is the one following the boundary orientation. 
Then we consider the cutting map
\begin{align}\label{eq:cutting_map}
    \mathrm{cut}_\alpha: \A_{G,\Sigma}^{\bw}[\alpha;v_\alpha] \to \A_{G,\Sigma'}^{\bw'},
\end{align}
where $\bw':=\bw \cup \{v_\alpha, v_\alpha^{-1}\}$ under the identification $\bB(\Sigma')=\bB(\Sigma) \cup \{\alpha',\alpha''\}$, $v_\alpha$ (resp. $v_\alpha^{-1}$) being assigned to $\alpha'$ (resp. $\alpha''$). 
The image of $\mathrm{cut}_\alpha$ is characterized by the closed condition 
\begin{align}\label{eq:h-constraint}
    h(\flA_{\alpha'}^+,\flA_{\alpha'}^-) = h(\flA_{\alpha''}^-,\flA_{\alpha''}^+)
\end{align}

\begin{prop}\label{lem:cutting_fibers}
Each fiber of $\mathrm{cut}_\alpha$ is isomorphic to $G_{v_\alpha}$, where $G_u$ denotes the stabilizer of the pair $([U^+],\overline{u}.B^+)$ for $u \in W$. In particular if $v_\alpha=w_0$, then $\mathrm{cut}_\alpha$ is an isomorphism onto its image. 
\end{prop}
It is proved in \cref{subsub:cutting_fibers} below, where we explicitly write down the presentation of $\mathrm{cut}_\alpha$ on an atlas by fixing a generating system on $\Sigma$.

These two propositions might be of independent interest.
Assuming them, let us first complete the proof of \cref{prop:codimension_count}. 

\begin{proof}[Proof of \cref{prop:codimension_count}]
Let us consider the cutting maps
\begin{align*}
    \A_{G,\Sigma}^{u,v}=\A_{G,\Sigma}[E_1,E_2;u,v] \xrightarrow{\mathrm{cut}_{E_1}} \A_{G,\Sigma'}^{\bw'}[E_2;v] \xrightarrow{\mathrm{cut}_{E_2}} \A_{G,\Sigma''}^{\bw''},
\end{align*}
where $\Sigma':=\Sigma \setminus E_1$, $\Sigma'':=\Sigma' \setminus E_2$, 
and $E_1$ (resp. $E_2$) is splitted into two boundary intervals $E'_1,E''_1$ in $\Sigma'$ (resp. $E'_2,E''_2$ in $\Sigma''$). 
The elements $\bw'=\{u,u^{-1}\}$, $\bw''=\{u,u^{-1},v,v^{-1}\}$ are the ones naturally inherited via the cutting. 

Recall that the image of the cutting map has the constraint \eqref{eq:h-constraint}.
Then from \cref{lem:codim_boundary,lem:cutting_fibers}, we get
\begin{align*}
    \dim \A_{G,\Sigma}^{u,v} 
    &= \A_{G,\Sigma'}^{\bw'}[E_2;v]- \dim H + \dim G_u \\
    &= (\A_{G,\Sigma''}^{\bw''}- \dim H + \dim G_v) -\dim H + \dim G_u \\
    &= \dim \A_{G,\Sigma''}^{\bw''}-2\dim H + \dim G_u + \dim G_v \\
    &= (\dim \A_{G,\Sigma''} - 2(2l(w_0)-l(u)-l(v))) -2\dim H + \dim G_u + \dim G_v
\end{align*}
Observe that $\dim \A^\times_{G,\Sigma''}-2\dim H = \dim \A^\times_{G,\Sigma}$ by the second statement of \cref{lem:cutting_fibers}, and that $\dim \A_{G,X}=\dim \A^\times_{G,X}$ for $X=\Sigma,\Sigma''$ by the second statement of \cref{lem:codim_boundary}.
We also have
\begin{align}\label{eq:codim_stabilizer}
    \dim G_w=\dim (U^+ \cap \overline{w}U^+ \overline{w}^{-1})=l(w_0) - l(w)
\end{align}
for any $w \in W$ \cite[Section 8.3]{Spr}. Thus we get
\begin{align*}
    \dim \A_{G,\Sigma}^{u,v} &= \dim \A_{G,\Sigma} - 2(2l(w_0)-l(u)-l(v)) + (l(w_0) -l(u)) + (l(w_0) -l(v)) \\
    &= \dim \A_{G,\Sigma} - (2l(w_0)-l(u)-l(v)),
\end{align*}
as desired. 
\end{proof}

\subsubsection{Proof of \cref{lem:cutting_fibers}}\label{subsub:cutting_fibers}

Here are preparatory discussions. We may assume that $\Sigma$ is connected, without loss of generality. Recall from \cref{lem:moduli_atlas} the presentation $\A_{G,\Sigma}=[A_{G,\Sigma}/G]$, where $A_{G,\Sigma}$ is a quasi-affine $G$-variety that can be identified with $\Hom^{\mathrm{tw}}(\pi_1(T'\Sigma,\xi),G)$. 

More precisely, such an identification is provided if we fix a system $\cC_{\mathrm{arc}}$ of arcs in $T'\Sigma$ from the basepoint $\xi$ to the boundary intervals. Let us further specify a collection $\cC_{\mathrm{loop}}$ of loops in $\Sigma$ based at $x:=\pi(\xi)$ that generate $\pi_1(\Sigma,x)$. See \cref{fig:curve_system}. 
Choosing framings of these loops (\emph{i.e.}, lifts of them to $T'\Sigma$), we get a splitting $\pi_1(\Sigma,x) \to \pi_1(T'\Sigma,\xi)$ of the exact sequence \eqref{eq:bundle_sequence}, and hence an isomorphism
\begin{align*}
    \Hom^{\mathrm{tw}}(\pi_1(T'\Sigma,\xi),G) \cong G^{\cC_{\mathrm{loop}}} = G^{-\chi(\Sigma)+1}.
\end{align*}
Let us call such a collection $\cC:=\cC_\mathrm{arc} \cup \cC_\mathrm{loop}$ a \emph{generating system}. 
Below we work on the atlases $A_{G,\Sigma} \cong G^{\cC_\mathrm{loop}} \times \A_G^\bM$ by choosing an appropriate generating system $\cC$ for a given $\alpha$. The atlases of the relevant moduli spaces are denoted by $A_{G,\Sigma}^{\bw} \subset A_{G,\Sigma}$, and so on.

\begin{proof}[Proof of \cref{lem:cutting_fibers}]

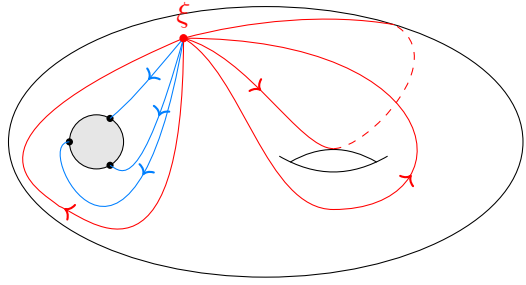
\begin{figure}[ht]
    \centering
\begin{tikzpicture}[scale=0.9]
\draw(2.5,0) ellipse (3.8 and 2);
\fill[gray!20](0,0) circle(0.4cm);
\draw(2.5+0.5*3.8,0.866*2) coordinate(y);

\draw(0,0) circle(0.4cm);
	\draw(150:0.9);
\fill(180:0.4) circle(1.5pt);
\fill(60:0.4) circle(1.5pt);
\fill(-60:0.4) circle(1.5pt);
\pic[scale=0.8] at (3.5,-0.3) {handle};

\node[red] at (60:0.6) {};
	\draw (50:2)  coordinate(x);
	
{\color{myblue}
	\draw[->-] (x) ..controls (50:1) and (50:0.8).. (60:0.4);
	\draw[->-] (x) ..controls (-50:1) and (-50:0.6).. (-60:0.4);
	\draw[->-] (x) ..controls (-80:3) and (-180:1).. (180:0.4);
}
\draw[red,->-] (x) ..controls ++(0,-4) and (-100:1.1).. (-120:1.1) ..controls (-140:1.1) and (180:2.3).. (x);
\draw[red] (-60:1.7);

\draw[red,->-] (x) ..controls ++(-15:1) and (3,-0.1).. (3.5,-0.1);
\draw[red,dashed] (3.5,-0.1) ..controls (4,-0.1) and ($(y)+(-30:1)$).. (y);
\draw[red] (y) ..controls ++(170:1) and ($(x)+(15:1)$).. (x);
\begin{scope}[xshift=3.5cm]
\draw[red,-<-] (x) ..controls ++(4,0) and (2,-1).. (0,-1) ..controls (-1,-1) and ($(x)+(-30:1)$).. (x);
\draw[red] (-90:0.8);
\end{scope}
\draw[red,fill] (50:2) circle(1.5pt) node[above]{$\xi$};
\end{tikzpicture}
    \caption{A generating system $\cC=\cC_\mathrm{\textcolor{myblue}{arc}}\cup\cC_\mathrm{\textcolor{red}{loop}}$ that determines an isomorphism $\A_{G,\Sigma} \cong G^{\cC_\mathrm{loop}} \times \A_G^\bM$. Here the framings of curves are omitted.}
    \label{fig:curve_system}
\end{figure}

The proof is devided into the three cases.

\paragraph{\textbf{Case 1: $\Sigma'$ is connected, and $\alpha$ connects different boundary components of $\Sigma$.}}
Choose the curves shown in the left in \cref{fig:curve_system_case1}, and extend it to a generating system $\cC$ on $\Sigma$ by choosing other curves disjointly from $\alpha$. The curves disjoint from $\alpha$ naturally descend to $\Sigma'$. Together with the curves shown in the right of \cref{fig:curve_system_case1}, they form a generating system $\cC'$ on $\Sigma'$. 

\begin{figure}[ht]
    \centering
\begin{tikzpicture}
\fill[gray!20](0,0) circle(0.5cm);
\draw(0,0) circle(0.5cm);
\foreach \i in {0,90,180,270} \fill(\i:0.5) circle(1.5pt) coordinate(y\i);
\fill[gray!20](4,0) circle(0.5cm);
\draw(4,0) circle(0.5cm);
\foreach \i in {60,180,-60} \fill (4,0)++(\i:0.5) circle(1.5pt) coordinate(z\i);
\draw[thick,->-] (y0) --node[midway,below=0.2em]{$\alpha$} (z180);

\draw (2,2) coordinate(x);
\draw[red,->-] (x) ..controls ++(-100:1) and (1,-1).. (0,-1) ..controls(-0.5,-1) and (-1,-0.5).. (-1,0) ..controls (-1,1) and ($(x)+(-180:1)$).. (x) node[pos=0.5,above]{$c_1$};
\draw[red,-<-] (x) ..controls ++(-80:1) and (3,-1).. (4,-1) ..controls(4.5,-1) and (5,-0.5).. (5,0) ..controls (5,1) and ($(x)+(0:1)$).. (x) node[pos=0.5,above]{$c_2$};
\draw[myblue,->-] (x) to[out=-110,in=20] (y0);
\draw[myblue,->-] (x) to[out=-70,in=160] (z180);
\fill[red] (2,2) circle(2pt) node[above=0.2em]{$\xi$};
\draw(y0) ++(0.2,0.4) node{$\flA_1$};
\draw(z180) ++(-0.2,0.4) node{$\flA_2$};
\draw[thick,->] (5.5,1) --node[midway,above]{$\mathrm{cut}_\alpha$} (6.5,1);
\node at (2,-2) {$\Sigma$};

\begin{scope}[xshift=8cm]
\fill[gray!20](0,0) circle(0.5cm);
\filldraw[draw=black,fill=gray!20](-0.5,0) arc(180:10:0.5) coordinate(w1) 
--++(3,0) coordinate(w2) 
arc(170:-170:0.5) coordinate(w3) 
--++(-3,0) coordinate(w4) 
arc(-10:-180:0.5);
\foreach \i in {90,180,270} \fill(\i:0.5) circle(1.5pt) coordinate(y\i);
\foreach \i in {60,-60} \fill (4,0)++(\i:0.5) circle(1.5pt) coordinate(z\i);
\draw (2,2) coordinate(x);
\draw[red,-<-] (x) ..controls ++(-180:1) and (-1,1).. (-1,0) ..controls(-1,-0.5) and (-0.5,-1).. (0,-1) --(4,-1) ..controls (4.5,-1) and (5,-0.5).. (5,0) ..controls (5,1) and ($(x)+(00:1)$).. (x);
\node[red] at (1.8,-1.3) {$c$};
\draw[myblue,->-] (x) to[out=-110,in=20] (w1);
\draw[myblue,->-] (x) to[out=-70,in=160] (w2);
\draw[myblue,->-] (x) ..controls ++(-160:1) and (-0.8,1).. (-0.8,0) ..controls(-0.8,-0.8) and (0.3,-1).. (w4); 
\draw[myblue,->-] (x) ..controls ++(-20:1) and (4.8,1).. (4.8,0) ..controls(4.8,-0.8) and (3.7,-1).. (w3); 
\fill[red] (2,2) circle(2pt) node[above=0.2em]{$\xi$};
\draw(w1) ++(0.2,0.4) node{$\flA_1$};
\draw(w2) ++(-0.2,0.4) node{$\flA_2$};
\draw[dashed] (w4) --++(0,-1.2) node[below]{$\rho(c_1)^{-1}.\flA_1$};
\draw[dashed] (w3) --++(0,-1.2) node[below]{$\rho(c_2).\flA_2$};
\foreach \i in {1,2,3,4} \fill(w\i) circle(1.5pt);
\node at (2,-2) {$\Sigma'$};
\end{scope}
\end{tikzpicture}
    \caption{The choice of curves: Case 1.}
    \label{fig:curve_system_case1}
\end{figure}
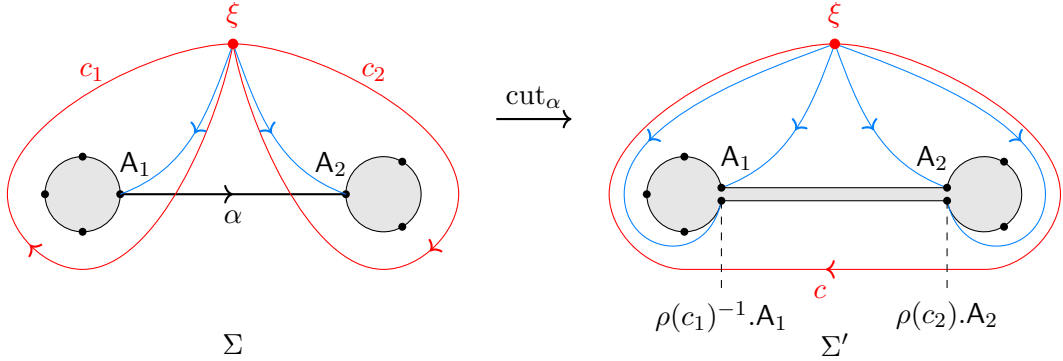

Note that $\mathrm{rank}\pi_1(\Sigma') = \mathrm{rank}\pi_1(\Sigma) -1$. We have an embedding $\iota:\pi_1(\Sigma',x) \to \pi_1(\Sigma,x)$ given by $\iota(c):=c_2\ast c_1$ and $\iota(d):=d$ for $d \in \cC'_\mathrm{loop} \setminus \{c\}$. It induces a projection 
\begin{align*}
    \iota^\ast: \Hom^{\mathrm{tw}}(\pi_1(T'\Sigma,\xi),G) \to \Hom^{\mathrm{tw}}(\pi_1(T'\Sigma',\xi),G),
\end{align*}
which is a principal $G$-bundle. Indeed, we have $\iota^\ast\rho(c)=\rho(c_2)\rho(c_1)$ and hence the $G$-action 
\begin{align}\label{eq:G-action_case1}
    (\rho(c_1),\rho(c_2)) \mapsto (g\rho(c_1),\rho(c_2)g^{-1}), \quad g \in G
\end{align}
parametrizes the fiber over $\rho \in \Hom^{\mathrm{tw}}(\pi_1(T'\Sigma',\xi),G)$. Then the relevant components of the presentation $\widetilde{\mathrm{cut}}_\alpha: A_{G,\Sigma}^{\bw}[\alpha;v_\alpha] \to \A_{G,\Sigma'}^{\bw'}$ of \eqref{eq:cutting_map} are given by
\begin{align*}
    &\Hom^{\mathrm{tw}}(\pi_1(T'\Sigma,\xi),G) \times \A_G^2 \to \Hom^{\mathrm{tw}}(\pi_1(T'\Sigma',\xi),G) \times \A_G^4, \\ &(\rho;\flA_1,\flA_2) \mapsto (\iota^\ast\rho;\flA_1,\flA_2, \rho(c_1)^{-1}.\flA_1,\rho(c_2).\flA_2).
\end{align*}
Observe that the $G$-action \eqref{eq:G-action_case1} preserves the $\A_G^4$-component if and only if $g \in \mathrm{Stab}(\flA_1,\flA_2) \subset G$. Hence the fiber of $\widetilde{\mathrm{cut}}_\alpha$ equals  $\mathrm{Stab}(\flA_1,\flA_2)$, which is isomorphic to $G_u$.

\paragraph{\textbf{Case 2: $\Sigma'$ is connected, and $\alpha$ connects the same boundary component $C$ of $\Sigma$.}}
In this case, one can find a handle that contains $\alpha$ as shown in \cref{fig:curve_system_case2}. (Indeed, if we shrink the boundary component $C$ to a puncture $p$, $\alpha$ becomes a based loop $\alpha_p$ at $p$. If $\alpha_p$ was non-essential, \emph{i.e.}, belonged to the kernel of the intersection form, then it would cut $\Sigma$ into a disconnected surface, which contradicts to our assumption. Then it is a standard argument in topology to find such a handle.) 

Choose the curves shown in the left in \cref{fig:curve_system_case2}, and extend it to a generating system $\cC$ on $\Sigma$ by choosing other curves disjointly from $\alpha$. The curves disjoint from $\alpha$ naturally descend to $\Sigma'$. Together with the curves shown in the right of \cref{fig:curve_system_case2}, they form a generating system $\cC'$ on $\Sigma'$. 

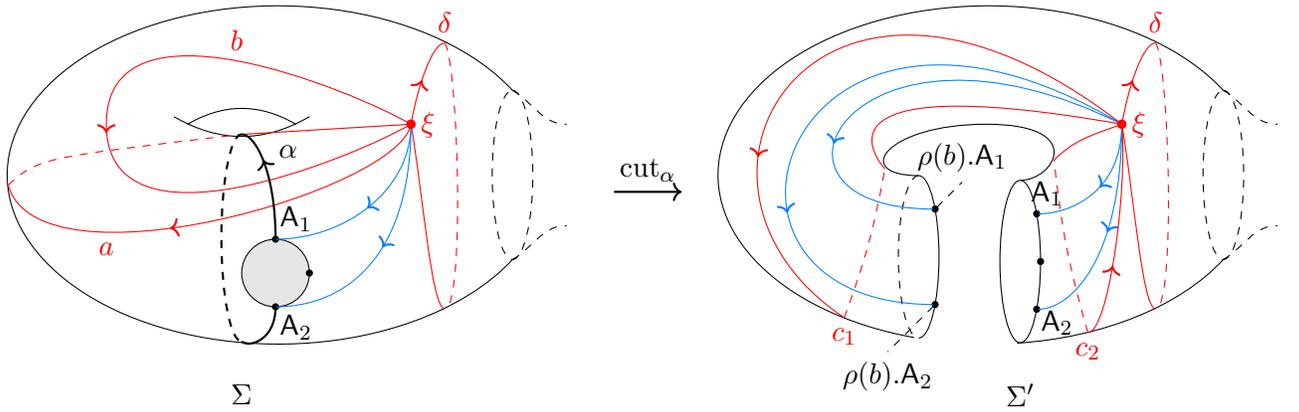
\begin{figure}[ht]
    \centering
\begin{tikzpicture}[scale=0.9]
\draw(0,1.5) arc(30:330:4cm and 2.5cm) coordinate(w);
\draw[dashed] (0,1.5) arc(90:-270:0.3cm and 1.25cm);
\draw[dashed] (0,1.5) to[out=-30,in=180] ++(0.8,-0.5);
\draw[dashed] (w) to[out=30,in=180] ++(0.8,0.5);
\pic at (-4,1) {handle};
\filldraw[fill=gray!20,draw=black] (-3.5,-1.2) coordinate(x) circle(0.5cm);
\foreach \i in {0,1,3} \fill(x)++(90*\i:0.5) circle(1.5pt) coordinate(z\i);
\draw(-1.5,1) coordinate(y);
\draw[red,-<-] (y) ..controls ++(-150:1) and (-6,-1).. (-6,1) ..controls (-6,3) and ($(y)+(150:1)$)..node[pos=0.5,above]{$b$} (y);
\draw[red,->-={0.6}{}] (y) ..controls ++(-100:1) and (-7,-1.5)..node[pos=0.7,below]{$a$} (-7.45,0.1);
\draw[red,dashed] (-7.45,0.1) ..controls (-7,0.5) .. (-4,0.85);
\draw[red] (-4,0.85) --(y);
\draw[red,->-] (y) ..controls++(80:0.5) and (-1.2,2.2).. (-1,2.2) node[above]{$\delta$};
\draw[red,dashed] (-1,2.2) ..controls++(0.2,-0.3) and (-0.7,-1.73).. (-1,-1.73);
\draw[red] (-1,-1.73) ..controls (-1.2,-1.73) and ($(y)+(-80:1)$).. (y);
\draw[myblue,->-] (y) ..controls ++(0,-1) and ($(z1)+(1,0)$).. (z1);
\draw[myblue,->-] (y) ..controls ++(0,-2) and ($(z3)+(1,0)$).. (z3);
\draw[thick,->-={0.7}{}](z1) ..controls++(0,1) and (-3.8,0.85)..node[pos=0.6,right=0.3em]{$\alpha$} (-4,0.85);
\draw[thick,dashed] (-4,0.85) arc(90:270:0.3cm and 1.53cm) coordinate(w);
\draw[thick] (z3) ..controls++(0,-0.5) and ($(w)+(0.1,-0.1)$).. (w);
\draw(z1)++(0.3,0.3) node{$\flA_1$};
\draw(z3)++(0.3,-0.3) node{$\flA_2$};
\fill[red](y) circle(2pt) node[right]{$\xi$};
\node at (-4,-3) {$\Sigma$};
\draw[thick,->] (1.5,0) --node[midway,above]{$\mathrm{cut}_\alpha$} (2.5,0);

\begin{scope}[xshift=10.5cm]
\draw(0,1.5) arc(30:254:4cm and 2.5cm);
\draw[dashed] (0,1.5) arc(90:-270:0.3cm and 1.25cm);
\draw[dashed] (0,1.5) to[out=-30,in=180] ++(0.8,-0.5);
\draw(0,1.5-2.5) coordinate(w);
\draw[dashed] (w) to[out=30,in=180] ++(0.8,0.5);
\draw(w) arc(-30:-84:4cm and 2.5cm);
\draw (-4.5,-1)++(0,1.24) coordinate(A1) arc(90:-90:0.3cm and 1.2cm) node[pos=0.3,inner sep=0](x1){} node[pos=0.7,inner sep=0](x2){} coordinate(A2);
\fill(x1) circle(1.5pt);
\fill(x2) circle(1.5pt);
\draw[dashed] (-4.5,-1)++(0,1.24) arc(90:270:0.3cm and 1.2cm);
\draw (-3,-1.07)++(0,1.24) coordinate(b') arc(90:-90:0.3cm and 1.2cm) node[pos=0.3,inner sep=0](y1){} node[pos=0.5,inner sep=0](y2){} node[pos=0.7,inner sep=0](y3){} coordinate(b);
\draw (-3,-1.07)++(0,1.24) arc(90:270:0.3cm and 1.2cm);
\fill(y1) circle(1.5pt);
\fill(y2) circle(1.5pt);
\fill(y3) circle(1.5pt);
\draw(-3,-1.03)++(0,1.2) coordinate(B1);
\draw(-3,-1.03)++(0,-1.2) coordinate(B2);
\draw(A1) ..controls++(-1,0) and (-5,1).. (-3.5,1) ..controls (-2,1) and ($(B1)+(0.5,0)$).. (B1);

\draw(-1.5,1) coordinate(y);
\draw[myblue,->-] (y) ..controls ++(0,-1) and ($(y1)+(0.5,0)$).. (y1);
\draw[myblue,->-] (y) ..controls ++(0,-2) and ($(y3)+(0.5,0)$).. (y3);
\draw[myblue,->-={0.7}{}] (y) ..controls ++(-4,2) and ($(x1)+(-3,0)$).. (x1);
\draw[myblue,->-={0.7}{}] (y) ..controls ++(-5,3) and ($(x2)+(-4,0)$).. (x2);
\draw[red,->-] (y) ..controls++(80:0.5) and (-1.2,2.2).. (-1,2.2) node[above]{$\delta$};
\draw[red,dashed] (-1,2.2) ..controls++(0.2,-0.3) and (-0.7,-1.73).. (-1,-1.73);
\draw[red] (-1,-1.73) ..controls (-1.2,-1.73) and ($(y)+(-80:1)$).. (y);
\draw[red,-<-={0.7}{}] (y) ..controls++(-85:1) and ($(b)+(1,0.15)+(0.3,0)$).. ($(b)+(1,0.15)$) node[below]{$c_2$};
\draw[red,dashed] ($(b)+(1,0.15)$) to[bend left=5] ($(b')+(0.5,0.3)$);
\draw[red] ($(b')+(0.5,0.3)$) ..controls++(0.3,0.3) and ($(y)+(200:0.5)$).. (y);
\draw[red,->-={0.7}{}] (y) ..controls ++(-4.5,3.5) and ($(A2)+(-1.1,0.3)+(-3,2)$).. ($(A2)+(-1.1,0.3)$) node[below]{$c_1$};
\draw[red,dashed] ($(A2)+(-1.1,0.3)$) to[bend right=5] ($(A1)+(-0.5,0.1)$);
\draw[red] ($(A1)+(-0.5,0.1)$) ..controls++(0,0) and ($(y)+(170:5)$).. (y);
\fill[red](y) circle(2pt) node[right]{$\xi$};
\draw(y1)++(0.15,0.3) node{$\flA_1$};
\draw(y3)++(0.3,-0.2) node{$\flA_2$};
\draw[dashed] (x1) --++(0.4,0.4) node[above]{$\rho(b).\flA_1$};
\draw[dashed] (x2) --++(-0.7,-0.7) node[below]{$\rho(b).\flA_2$};
\node at (-3,-3) {$\Sigma'$};
\end{scope}
\end{tikzpicture}
    \caption{The choice of curves: Case 2. Here $\delta$ can be written as a product of other loops in $\cC$.}
    \label{fig:curve_system_case2}
\end{figure}

Note that $\mathrm{rank}\pi_1(\Sigma') = \mathrm{rank}\pi_1(\Sigma) -1$. 
We have an embedding $\iota:\pi_1(\Sigma',x) \to \pi_1(\Sigma,x)$ given by $\iota(c_1):=b\ast a \ast b^{-1}$, $\iota(c_2):=b\ast a \ast b^{-1}\ast\delta$ and $\iota(d):=d$ for $d \in \cC'_\mathrm{loop} \setminus \{c_1,c_2\}$. It induces a projection 
\begin{align*}
    \iota^\ast: \Hom^{\mathrm{tw}}(\pi_1(T'\Sigma,\xi),G) \to \Hom^{\mathrm{tw}}(\pi_1(T'\Sigma',\xi),G),
\end{align*}
which is a principal $G$-bundle. Indeed, we have $\iota^\ast\rho(c_1)=\rho(b)\rho(a)\rho(b)^{-1}$ and $\iota^\ast\rho(c_2)=\rho(b)\rho(a)\rho(b)^{-1}\rho(\delta)$; hence the $G$-action 
\begin{align}\label{eq:G-action_case2}
    (\rho(a),\rho(b)) \mapsto (g^{-1}\rho(a)g,\rho(b)g), \quad g \in G
\end{align}
parametrizes the fiber over $\rho \in \Hom^{\mathrm{tw}}(\pi_1(T'\Sigma',\xi),G)$. Then the relevant components of the presentation $\widetilde{\mathrm{cut}}_\alpha: A_{G,\Sigma}^{\bw}[\alpha;v_\alpha] \to \A_{G,\Sigma'}^{\bw'}$ of \eqref{eq:cutting_map} are given by
\begin{align*}
    &\Hom^{\mathrm{tw}}(\pi_1(T'\Sigma,\xi),G) \times \A_G^2 \to \Hom^{\mathrm{tw}}(\pi_1(T'\Sigma',\xi),G) \times \A_G^4, \\ &(\rho;\flA_1,\flA_2) \mapsto (\iota^\ast\rho;\flA_1,\flA_2, \rho(b).\flA_1,\rho(b).\flA_2).
\end{align*}
Observe that the $G$-action \eqref{eq:G-action_case2} preserves the $\A_G^4$-component if and only if $g \in \mathrm{Stab}(\flA_1,\flA_2) \subset G$. Hence the fiber of $\widetilde{\mathrm{cut}}_\alpha$ equals  $\mathrm{Stab}(\flA_1,\flA_2)$, which is isomorphic to $G_u$.

\paragraph{\textbf{Case 3: $\Sigma'$ is disconnected.}}
Let $\Sigma'=\Sigma_1\sqcup \Sigma_2$ be the decomposition into connected components, where $\Sigma_1,\Sigma_2$ are connected marked surfaces. Choose the basepoint $x$ on $\alpha$. Then $\pi_1(\Sigma,x) = \pi_1(\Sigma_1,x) \ast \pi_1(\Sigma_2,x)$. We can choose the generating system $\cC$ on $\Sigma$ so that they do not cross $\alpha$, as shown in \cref{fig:curve_system_case3}. It induces a generating system $\cC'$ on $\Sigma'$, where the two arcs connecting to the endpoints of $\alpha$ are doubled. 

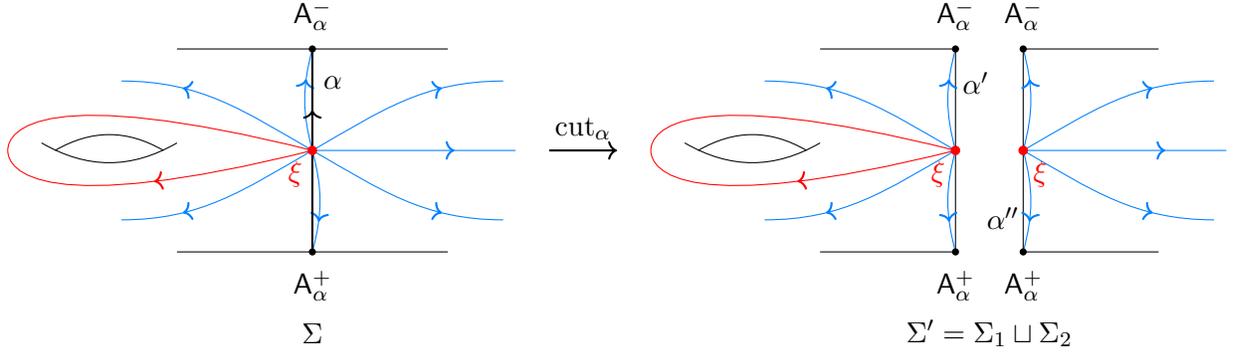
\begin{figure}[ht]
    \centering
\begin{tikzpicture}[scale=0.9]
\draw (-2,3) -- (2,3);
\draw (-2,0) -- (2,0);
\draw(0,1.5) coordinate(x);
\pic at (-3,1.5) {handle};
\draw[red,->-] (x) ..controls ++(-4,-1) and (-4.5,1.2).. (-4.5,1.5);
\draw[red] (x) ..controls ++(-4,1) and (-4.5,1.8).. (-4.5,1.5);
\draw[myblue,->-={0.7}{}] (x) to[out=150,in=0]++(160:3);
\draw[myblue,->-={0.7}{}] (x) to[out=-150,in=0]++(-160:3);
\draw[myblue,->-={0.7}{}] (x) to[out=30,in=180]++(20:3);
\draw[myblue,->-={0.7}{}] (x) to[out=-30,in=180]++(-20:3);
\draw[myblue,->-={0.7}{}] (x) --++(3,0);
\draw[myblue,->-={0.7}{}] (x) to[bend left=15] (0,3);
\draw[myblue,->-={0.7}{}] (x) to[bend left=15] (0,0);
\draw[thick,->-={0.7}{}] (0,0) -- (0,3);
\node at (0.3,2.5) {$\alpha$};
\fill[red] (0,1.5) circle(2pt) node[below left]{$\xi$};
\fill (0,0) circle(1.5pt);
\fill (0,3) circle(1.5pt);
\node at (0,-1.2) {$\Sigma$};
\node at (0,3.5) {$\flA_\alpha^-$};
\node at (0,-0.5) {$\flA_\alpha^+$};
\draw[thick,->] (3.5,1.5) --node[midway,above]{$\mathrm{cut}_\alpha$} (4.5,1.5);

\begin{scope}[xshift=9.5cm]
\draw (-2,3) -- (0,3);
\draw (-2,0) -- (0,0);
\draw (0,0) -- (0,3);
\draw (1,3) -- (3,3);
\draw (1,0) -- (3,0);
\draw (1,0) -- (1,3);
\draw(0,1.5) coordinate(x);
\draw(1,1.5) coordinate(y);
\pic at (-3,1.5) {handle};
\draw[red,->-] (x) ..controls ++(-4,-1) and (-4.5,1.2).. (-4.5,1.5);
\draw[red] (x) ..controls ++(-4,1) and (-4.5,1.8).. (-4.5,1.5);
\draw[myblue,->-={0.7}{}] (x) to[out=150,in=0]++(160:3);
\draw[myblue,->-={0.7}{}] (x) to[out=-150,in=0]++(-160:3);
\draw[myblue,->-={0.7}{}] (y) to[out=30,in=180]++(20:3);
\draw[myblue,->-={0.7}{}] (y) to[out=-30,in=180]++(-20:3);
\draw[myblue,->-={0.7}{}] (y) --++(3,0);
\draw[myblue,->-={0.7}{}] (x) to[bend left=15] (0,3);
\draw[myblue,->-={0.7}{}] (x) to[bend right=15] (0,0);
\draw[myblue,->-={0.7}{}] (y) to[bend right=15] (1,3);
\draw[myblue,->-={0.7}{}] (y) to[bend left=15] (1,0);
\fill[red] (0,1.5) circle(2pt) node[below left]{$\xi$};
\fill[red] (1,1.5) circle(2pt) node[below right]{$\xi$};
\fill (0,0) circle(1.5pt);
\fill (0,3) circle(1.5pt);
\fill (1,0) circle(1.5pt);
\fill (1,3) circle(1.5pt);
\node at (0,3.5) {$\flA_\alpha^-$};
\node at (0,-0.5) {$\flA_\alpha^+$};
\node at (1,3.5) {$\flA_\alpha^-$};
\node at (1,-0.5) {$\flA_\alpha^+$};
\node at (0.3,2.5) {$\alpha'$};
\node at (0.7,0.5) {$\alpha''$};
\node at (0.5,-1.2) {$\Sigma'=\Sigma_1\sqcup \Sigma_2$};
\end{scope}
\end{tikzpicture}
    \caption{The choice of curves: Case 3.}
    \label{fig:curve_system_case3}
\end{figure}

Let $\overline{A}_{G,\Sigma}^{\bw}[\alpha;v_\alpha] \subset A_{G,\Sigma}^{\bw}[\alpha;v_\alpha]$ denote the subspace such that $(\flA_\alpha^+,\pi(\flA_\alpha^-))=([U^+],\overline{u}.B^+)$. Since any point $(\rho,\lambda) \in A_{G,\Sigma}^{\bw}[\alpha;v_\alpha]$ can be translated into such a configuration by the $G$-action, we have
\begin{align*}
    \A_{G,\Sigma}^{\bw}[\alpha;v_\alpha] = [A_{G,\Sigma}^{\bw}[\alpha;v_\alpha]/G] = [\overline{A}_{G,\Sigma}^{\bw}[\alpha;v_\alpha]/G_u].
\end{align*}
A similar argument shows
\begin{align*}
    \A_{G,\Sigma'}^{\bw'} = [A_{G,\Sigma}^{\bw'}/(G \times G)] = [\overline{A}_{G,\Sigma}^{\bw'}/(G_u \times G_u)].
\end{align*}
Observe that $\overline{A}_{G,\Sigma}^{\bw}[\alpha;v_\alpha]$ is isomorphic via $\widetilde{\mathrm{cut}}_\alpha$ to the closed subspace of $\overline{A}_{G,\Sigma}^{\bw'}$ characterized by the condition $h(\flA_{\alpha'}^+,\flA_{\alpha'}^-) = h(\flA_{\alpha''}^-,\flA_{\alpha''}^+)$. Identifying these spaces, we have
\begin{align*}
    \mathrm{cut}_\alpha: [\overline{A}_{G,\Sigma}^{\bw}[\alpha;v_\alpha]/G_u] \to [\overline{A}_{G,\Sigma}^{\bw}[\alpha;v_\alpha]/(G_u \times G_u)],
\end{align*}
which has the fiber $(G_u \times G_u) /G_u \cong G_u$. The assertion is proved.
\end{proof}

\subsubsection{Proof of \cref{lem:codim_boundary}}\label{subsub:codim_boundary}

The knowledge on the \emph{braid varieties} is useful below. We refer the reader to \cite{CGGLSS22} for details. Here we introduce a stacky variant of them. 
Given a positive braid word $\beta=(s_1,\dots,s_l)$ with $s_1,\dots,s_l \in S$, the associated \emph{braid stack} is defined to be the quotient stack
\begin{align*}
    X(\beta):=\big[\{ &(\flA_0,\flB_1,\dots,\flB_{l-1},\flB_l) \in \A_G \times \B_G^{l} \mid \\
    &w(\flB_{i},\flB_{i-1})=r_{s_i}~\mbox{for $i=1,\dots,l$}; ~w(\flB_0,\flB_l)=\delta(\beta)\}/G\big].
\end{align*}
Here $\delta(\beta) \in W$ denotes the \emph{Demazure product}. A configuration in $X(\beta)$ is illustrated as
\begin{equation*}
    \begin{tikzcd}
    \flB_{l-1} \ar[d,"s_l"']& \ar[l,"s_{l-1}"'] \cdots & \flB_{1}  \ar[l,"s_{2}"']\\
    \flB_l & & \flA_0 \ar[u,"s_1"']\ar[ll,"\delta(\beta)"] .
    \end{tikzcd}
\end{equation*}
We call $\flB_l\xleftarrow{\delta(\beta)} \flA_0$ the \emph{bottom side} of the configuration. 
If $\delta(\beta)=w_0$, then the $G$-action is free and we can uniquely translate the generic pair $(\flA_0,\flB_{l+1})$ to the standard pinning $([U^+],B^-)$. Hence
\begin{align*}
    X(\beta) \cong \{ (\flB_1,\dots,\flB_{l-1}) \in \B_G^{l-1} \mid 
    w(\flB_{i-1},\flB_{i})=r_{s_i}~\mbox{for $i=1,\dots,l$},~\flB_0:=B^+,\flB_l:=B^-\}
\end{align*}
is an affine algebraic variety. In this case, it is known that $\dim X(\beta)=|\beta| -l(w_0)$ \cite[Theorem 20]{Escobar}, where $|\beta|=l$ is the length of $\beta$. 

In general, if $\delta(\beta)=w \in W$, the pair $(\flA_0,\flB_{l+1})$ can be translated to the position $([U^+],\overline{w}.B^+)$, whose stabilizer subgroup is $G_w$. Hence we get
\begin{align*}
    X(\beta) \cong \big[ \{ (\flB_1,\dots,\flB_{l-1}) \in \B_G^{l-1} \mid 
    w(\flB_{i-1},\flB_{i})=r_{s_i}~\mbox{for $i=1,\dots,l$},~\flB_0:=B^+,\flB_l:=\overline{w}.\flB_0\}/G_w\big]
\end{align*}
We again know the dimension of the affine algebraic variety before quotient by \cite[Theorem 20]{Escobar}, and we get 
\begin{align}\label{eq:dim_braid_variety}
    \dim X(\beta) = |\beta| - l(w)-\dim G_w = |\beta| - l(w_0).
\end{align}

Let us drop the condition on the relative position between $\flB_0$ and $\flB_l$, defining
\begin{align*}
    \overline{X}(\beta):=\big[\{ &(\flA_0,\flB_1,\dots,\flB_{l-1},\flB_l) \in \A_G \times \B_G^{l} \mid \\
    &w(\flB_{i},\flB_{i-1})=r_{s_i}~\mbox{for $i=1,\dots,l$}\}/G\big].
\end{align*}
Note that before quotient by $G$, the space is isomorphic to $\mathcal{A}_G\times\mathbb{A}^{|\beta|}$, since a pair of flags of $w$-distance $r_s$ for $s \in S$ can be translated into the position $(B^+,x_s(t)\overline{r}_s.B^+)$ for some $t \in \mathbb{A}$. Therefore
\begin{equation}
\dim \overline{X}(\beta) = \dim \mathcal{A}_G + |\beta| - \dim G = |\beta| - l(w_0).\label{eq:dim_braid_1}
\end{equation}

\paragraph{\textbf{Case 1: $\Sigma$ is a $k$-gon ($k \geq 3$).}} 
Enumerate the boundary intervals as $E_0,E_1,\dots,E_{k-1}$ in this counter-clockwise order so that $w_{E_0}=w_0$. Choose a reduced word $\bs(i)$ of $w_{E_i}$ for $i=1,\dots,k-1$. 
Then the concatenation
\begin{align*}
    \beta:=\bs(1)\bs(2)\dots\bs(k-1)
\end{align*}
is a positive braid word of length $l=\sum_{i=1}^{k-1}l(w_i)$ such that $\delta(\beta)=w_0$. Let $\flA_0$ be the decoration assigned to the terminal endpoint of $E_0$. 
Then we have an isomorphism
\begin{align*}
    &\A_{G,\Sigma}^{\bw} \xrightarrow{\sim} X(\beta) \times H^{k-1}, \\
    &[\flA_0,\dots,\flA_{k-1}] \mapsto ([\flA_0,\widetilde{\flB}_1,\dots,\widetilde{\flB}_{l-1}], h(\flA_0,\flA_1),\dots,h(\flA_{k-2},\flA_{k-1})). 
\end{align*}

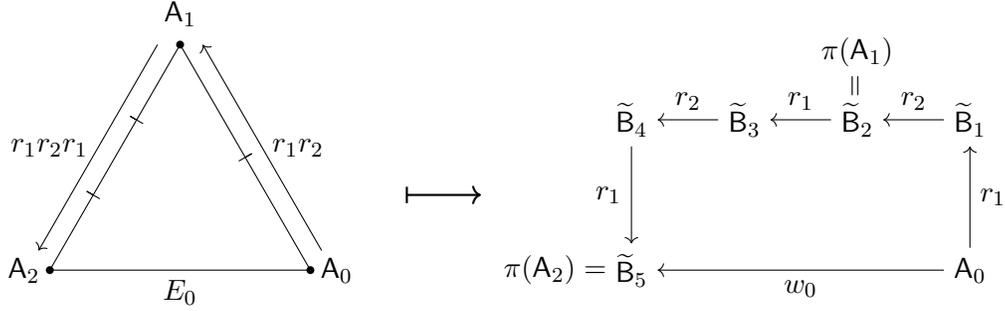
\begin{figure}
    \centering
\begin{tikzpicture} 
\foreach \i in {0,1,2}
{
\draw(\i*120-30:2) coordinate(A\i) --(\i*120+90:2);
\fill(\i*120-30:2) circle(1.5pt);
}
\draw(A0) node[right]{$\flA_0$};
\draw(A1) node[above=0.3em]{$\flA_1$};
\draw(A2) node[left]{$\flA_2$};
\draw($(A0)!0.5!(A2)$) node[below]{$E_0$};
\draw[<-] (A2)++(120:0.3) --node[midway,left]{$r_1r_2r_1$} ($(A1)+(-0.3,0)$);
\draw($(A2)!0.333!(A1)$) --++(150:0.1) --++(-30:0.2);
\draw($(A2)!0.666!(A1)$) --++(150:0.1) --++(-30:0.2);
\draw[->] (A0)++(60:0.3) --node[midway,right]{$r_1r_2$} ($(A1)+(0.3,0)$);
\draw($(A0)!0.5!(A1)$) --++(30:0.1) --++(-150:0.2);
\draw[|->,thick] (3,0) -- (4,0);
\begin{scope} [xshift=6cm]
\node (A) at (0,-1) {$\widetilde{\flB}_5$};
\node (C) at (4.5,-1) {$\flA_0$};
\node[anchor=east] at (-0.2,-1) {$\pi(\flA_2)=$};
\foreach \i [count=\xi] in {0,1,2,3} \node (B\i) at (4.5-\i*1.5,1) {$\widetilde{\flB}_{\xi}$};
\draw[->] (C) --node[midway,right]{$r_1$} (B0);
\draw[->] (B0) --node[midway,above]{$r_2$} (B1);
\draw[->] (B1) --node[midway,above]{$r_1$} (B2);
\draw[->] (B2) --node[midway,above]{$r_2$} (B3);
\draw[->] (B3) --node[midway,left]{$r_1$} (A);
\draw[<-] (A) --node[midway,below]{$w_0$} (C);
\draw(B1)++(0,0.45) node[rotate=90]{$=$};
\draw(B1)++(0,0.9) node{$\pi(\flA_1)$};
\end{scope}
\end{tikzpicture}
    \caption{A map from $\A_{G,\Sigma}^\bw$ to $X(\beta)$. Here $\Sigma$ is a triangle, and $\bw=(w_{E_0},w_{E_1},w_{E_2})=(w_0,r_1r_2,r_1r_2r_1)$.}
    \label{fig:polygon_braid}
\end{figure}

See \cref{fig:polygon_braid}.
Here for each pair $(\flB_{i-1},\flB_{i})$ with $i=1,\dots,k-1$, we take the sequence of interpolating flags with their distances given by $\bs(i)$, and denote by $\widetilde{\flB}_1,\dots,\widetilde{\flB}_{l-1}$ the resulting sequence from $\flB_0$ to $\flB_{k-1}$. The bottom side is $\flB_{k-1}\xleftarrow{w_0} \flA_0$. 
Observe that the data $[\flA_0,\flA_1,\dots,\flA_{k-1}]$ can be uniquely recovered from $[\flA_0,\flB_1,\dots,\flB_{k-1}]$ and $h(\flA_{i-1},\flA_i)$ for $i=1,\dots,k-1$. It follows that 
\begin{align*}
    \dim \A_{G,\Sigma}^{\bw}= \dim X(\beta) + (k-1)\rank G = \sum_{i=1}^{k-1} l(w_i) - l(w_0) +(k-1)\rank G
\end{align*}
by \eqref{eq:dim_braid_variety}. 
The open dense part $\A_{G,\Sigma}^\times$ corresponds to $w_{E_i}=w_0$ for all $i=1,\dots,k-1$. Therefore
\begin{align*}
    \dim \A_{G,\Sigma} - \dim \A_{G,\Sigma}^{\bw} = \dim \A^\times_{G,\Sigma} - \dim \A_{G,\Sigma}^{\bw} = \sum_{i=1}^{k-1}(l(w_0)-l(w_{E_i})),
\end{align*}
as desired. 

\begin{figure}[ht]
    \centering
\begin{tikzpicture}[scale=0.9]
\draw(2.5,0) ellipse (3.8 and 2);
\fill[gray!20](0,-0.5) circle(0.4cm);
\draw(2.5+0.5*3.8,0.866*2) coordinate(y);

\draw(0,-0.5) circle(0.4cm);
	\draw(150:0.9);
\fill(0,-0.5)++(180:0.4) circle(1.5pt);
\fill(0,-0.5)++(60:0.4) circle(1.5pt);
\fill(0,-0.5)++(-60:0.4) circle(1.5pt) coordinate(z);
\pic[scale=0.8] at (3.5,-0.3) {handle};

\fill[gray!20](-0.2,-0.3)++(50:2) arc(45:45+360:0.4cm);
\draw(-0.2,-0.3)++(50:2)  coordinate(x) arc(45:45+360:0.4cm);
	
{\color{myblue}
	\draw[->-] (x) ..controls ++(-45:1) and ($(z)+(30:1)$).. (z);
}

\draw[red,->-] (x) ..controls ++(-15:1) and (3,-0.1).. (3.5,-0.1);
\draw[red,dashed] (3.5,-0.1) ..controls (4,-0.1) and ($(y)+(-30:1)$).. (y);
\draw[red] (y) ..controls ++(170:1) and ($(x)+(15:1)$).. (x);
\begin{scope}[xshift=3.5cm]
\draw[red,-<-] (x) ..controls ++(4,0) and (2,-1).. (0,-1) ..controls (-1,-1) and ($(x)+(-30:1)$).. (x);
\end{scope}
\draw[fill] (x) circle(1.5pt) node[above,scale=0.9]{$m_0$};
\draw[fill] (z) circle(1.5pt) node[below,scale=0.9]{$m_1$};
\end{tikzpicture}
    \caption{A cut system $\cC$ of $\Sigma$, which consists of $2g$ loops based at $m_0$ and $b-1$ arcs connecting $m_0$ to $m_\nu$, $\nu=1,\dots,b-1$.}
    \label{fig:cut_system}
\end{figure}
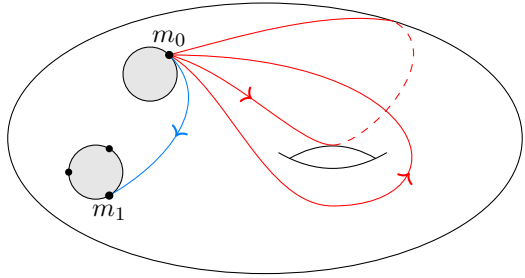

\paragraph{\textbf{Case 2: $\Sigma$ is not a polygon.}}
This case can be reduced to the polygon case, as follows. Let $\partial \Sigma=\bigsqcup_{\nu=0}^{b-1} \partial_\nu$.  
Choose a special point $m_\nu \in \partial_\nu$ for $\nu=1,\dots,b-1$, and a \emph{cut system} $\cC$ as shown in \cref{fig:cut_system}. The collection $\cC$ consists of $2g+b-1>0$ curves. 
By cutting $\Sigma$ along the curves in $\cC$, we get a polygon $\Pi$ with $4g+2(b-1)+|\bM|$ sides. Let us consider the subspace $\A_{G,\Sigma}^{\bw}[\cC;v_\cC] \subset \A_{G,\Sigma}^{\bw}$, where $v_\cC=\{v_\alpha\}_{\alpha \in \cC} \in W^\cC$ and recall \cref{dfn:moduli_cells}. Let us consider the composite of cutting maps along the arcs in $\cC $:
\begin{align}\label{eq:cut_polygon}
    \mathrm{cut}_\cC:=\prod_{\alpha \in \cC} \mathrm{cut}_\alpha: \A_{G,\Sigma}^{\bw}[\cC;v_\cC] \to \A_{G,\Pi}^{\bw_\cC},
\end{align}
where $\bw_\cC:=\bw\cup \{v_\alpha,v_\alpha^{-1}\mid \alpha \in \cC\}$ is naturally inherited via the cutting, the pair $\{v_\alpha^\pm\}$ being assigned to the two boundary intervals arising from $\alpha$. 
See \cref{fig:surface_braid}. 
Observe that the image $\mathrm{Im}(\mathrm{cut}_\cC) \subset \A_{G,\Pi}^{\bw_\cC}$ is characterized by the closed condition \eqref{eq:h-constraint}, one for each $\alpha \in \cC$. By \cref{lem:cutting_fibers}, the fiber of $\mathrm{cut}_\cC$ is isomorphic to the product $\prod_{\alpha \in \cC} G_{v_\alpha}$. 

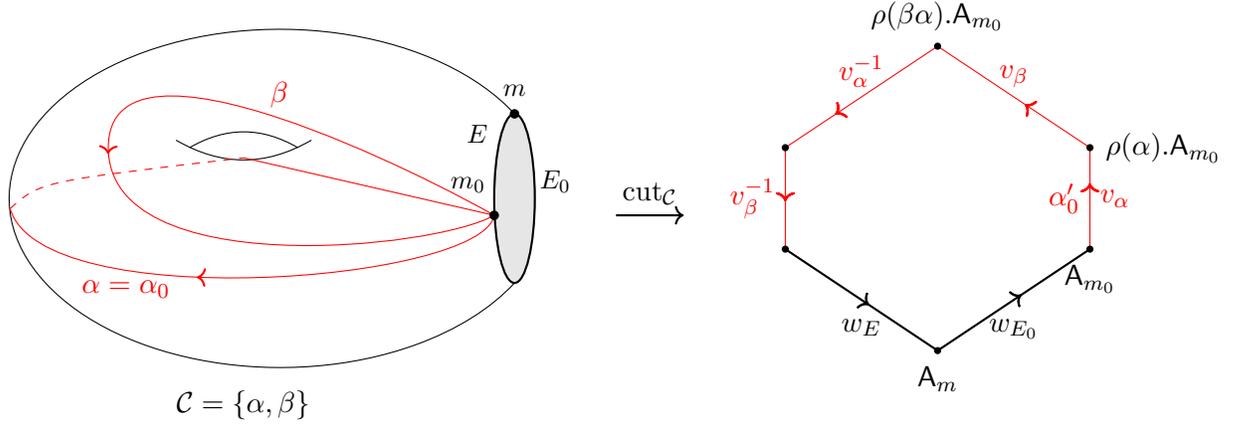
\begin{figure}[ht]
    \centering
\begin{tikzpicture}[scale=0.9]
\draw(0,1.5) arc(30:330:4cm and 2.5cm) coordinate(w);
\fill[gray!20] (0,1.5) arc(90:-270:0.3cm and 1.25cm);
\draw[thick] (0,1.5) arc(90:-270:0.3cm and 1.25cm);
\fill(0,1.5) circle(2pt) node[above=0.3em,scale=0.9]{$m$};
\pic at (-4,1) {handle};
\draw(-0.3,0) coordinate(y);
\draw[red,-<-] (y) ..controls ++(-150:1) and (-6,-1).. (-6,1) ..controls (-6,3) and ($(y)+(150:1)$)..node[pos=0.5,above]{$\beta$} (y);
\draw[red,->-={0.6}{}] (y) ..controls ++(-100:1) and (-7,-1.5)..node[pos=0.7,below]{$\alpha=\alpha_0$} (-7.45,0.1);
\draw[red,dashed] (-7.45,0.1) ..controls (-7,0.5) .. (-4,0.85);
\draw[red] (-4,0.85) --(y);
\fill(y) circle(2pt);
\draw(y)++(0,0.2) node[above left,scale=0.9]{$m_0$};
\draw(y)++(-0.25,1.2) node[scale=0.9]{$E$};
\draw(y)++(0.9,0.5) node[scale=0.9]{$E_0$};
\node at (-4,-2.8) {$\cC=\{\alpha,\beta\}$};
\draw[->,thick] (1.5,0) --node[midway,above]{$\mathrm{cut}_\cC$} (2.5,0);

\begin{scope}[xshift=4cm,yshift=-0.5cm]
\draw[red,-<-={0.1}{},-<-={0.3}{},-<-={0.7}{},-<-={0.9}{}](0,0) coordinate(A) -- (0,1.5) coordinate(B) -- (2.25,3)  coordinate(C) -- (4.5,1.5) coordinate(E) -- (4.5,0) coordinate(F);
\draw(2.25,-1.5) coordinate(G);
\draw[-<-={0.25}{},-<-={0.75}{},thick](F) -- (G) -- (A);
\foreach \i in {A,B,C,E,F,G} \fill(\i) circle(1.5pt);
\node[above=0.2em] at (C) {$\rho(\beta\alpha).\flA_{m_0}$};
\node[right=0.2em,anchor=west] at (E) {$\rho(\alpha).\flA_{m_0}$};
\node[below=0.2em] at (F) {$\flA_{m_0}$};
\node[below=0.2em] at (G) {$\flA_m$};
\draw($(A)!0.5!(G)$) node[below=0.3em]{$w_{E}$};
\draw($(F)!0.5!(G)$) node[below=0.3em]{$w_{E_0}$};
{\color{red}
\draw($(F)!0.5!(E)$) node[right]{$v_\alpha$} node[left,anchor=east]{$\alpha'_0$};
\draw($(E)!0.5!(C)$) node[above]{$v_\beta$};
\draw($(C)!0.5!(B)$) node[above]{$v_\alpha^{-1}$};
\draw($(B)!0.5!(A)$) node[left]{$v_\beta^{-1}$};
}
\end{scope}
\end{tikzpicture}
    \caption{The cutting map $\mathrm{cut}_\cC$. In this example, we get two braid words $\beta(\bw,\bv_\cC;\alpha'_0)=\bs(v_\beta)\bs(v_\alpha^{-1})\bs(v_\beta^{-1})\bs(w_E)\bs(w_{E_0})$ and $\beta(\bw,\bv_\cC;E_0)=\bs(v_\alpha)\bs(v_\beta)\bs(v_\alpha^{-1})\bs(v_\beta^{-1})\bs(w_E)$, where $\bs(w)$ for $w \in W$ is an arbitrary reduced word.}
    \label{fig:surface_braid}
\end{figure}

Choose arbitrary reduced words of $\bw$ and $\bv_\cC$. We further choose a side $\alpha'_0$ of the polygon $\Pi$ which comes from a curve $\alpha_0 \in \cC$. Then we get a braid word $\beta(\bw,\bv_\cC;\alpha'_0)$ by reading off the reduced words assigned to the sides of $\Pi$ except for $\alpha'_0$ along the boundary orientation of $\Pi$. 

Then we get an embedding 
\begin{align}
    &\mathrm{Im}(\mathrm{cut}_\cC) \hookrightarrow \overline{X}(\beta(\bw,\bv_\cC;\alpha'_0)) \times H^{2g+b-1+|\bM|-1}.
\end{align}
From this embedding,
we can compute 
\begin{align}
    \dim \A_{G,\Sigma}^\bw[\cC;v_\cC] &\leq  \sum_{\alpha \in \cC} \dim G_{v_\alpha}+ (2g+b-1+|\bM|-1)\rank G + \dim \overline{X}(\beta(\bw,\bv_\cC;\alpha'_0)).\label{eq:dimension_0} 
\end{align}
In particular, when $w_{\alpha'_0}=w_0=\delta(\beta(\bw,\bv_\cC;\alpha'_0))$, we get an isomorphism 
\begin{align}
    &\mathrm{Im}(\mathrm{cut}_\cC) \xrightarrow{\sim} {X}(\beta(\bw,\bv_\cC;\alpha'_0)) \times H^{2g+b-1+|\bM|-1}.
\end{align} 
In this case, we get 
\begin{align}
    \dim \A_{G,\Sigma}^\bw[\cC;v_\cC] &=  \sum_{\alpha \in \cC} \dim G_{v_\alpha}+ (2g+b-1+|\bM|-1)\rank G + \dim \overline{X}(\beta(\bw,\bv_\cC;\alpha'_0)).\label{eq:dimension_1} 
\end{align}

Let $\bv_0:=\{w_0\}_{\alpha \in \cC} \in W^\cC$. We have  the following Lemma.

\begin{lem}\label{lem:dimension_general}
\begin{enumerate}
    \item Fixing $\bv_\cC=\bv_0$ to be the longest, we get
    \begin{align*}
        \dim \A_{G,\Sigma}^\times[\cC;\bv_0] - \dim \A_{G,\Sigma}^\bw[\cC;\bv_0] &=  \dim X(\beta(\bw_0,\bv_0;\alpha'_0)) -\dim X(\beta(\bw,\bv_0;\alpha'_0)) \\
        &= \sum_{E \in \bB(\Sigma)} (l(w_0)- l(w_E)).
    \end{align*}
    In particular, the subspace $\A_{G,\Sigma}^\times[\cC;\bv_0] \subset \A_{G,\Sigma}[\cC;\bv_0]$ is open dense. 
    \item Fixing any distance $\bw \in W^{\bB(\Sigma)}$, we get
    \begin{align*}
         \dim \A_{G,\Sigma}^{\bw}[\cC;\bv_0] - \dim \A_{G,\Sigma}^{\bw}[\cC;\bv_\cC] \geq 0
    \end{align*}
    Therefore the subspace $\A_{G,\Sigma}^\bw[\cC;\bv_0] \subset \A_{G,\Sigma}^\bw$ has the maximal dimension.
\end{enumerate}
\end{lem}

\begin{proof}
(1): The first equality follows from \eqref{eq:dimension_1}. Since the $w$-distances on the boundary intervals are set to be the longest, we have $\delta(\beta(\bw,\bv_\cC;\alpha'_0))=w_0$. Then the second equality follows from \eqref{eq:dim_braid_variety}. 

(2): 
We apply \eqref{eq:dimension_1} to compute $\dim \A_{G,\Sigma}^{\bw}[\cC;\bv_0]$ and \eqref{eq:dimension_0} to $\dim \A_{G,\Sigma}^{\bw}[\cC;\bv_c]$. By \eqref{eq:dim_braid_variety} and \eqref{eq:dim_braid_1}, we get
\begin{align*}
    \dim &\A_{G,\Sigma}^{\bw}[\cC;\bv_0] - \dim \A_{G,\Sigma}^{\bw}[\cC;\bv_\cC] \geq  \dim X(\beta(\bw,\bv_0;\alpha_0')) - \dim \overline{X}(\beta(\bw,\bv_\cC;\alpha_0'))-\sum_{\alpha \in \cC} \dim G_{v_\alpha} 
    \\
    =& |\beta(\bw,\bv_0;\alpha_0')|-|\beta(\bw,\bv_\cC;\alpha_0')|-\sum_{\alpha \in \cC} (l(w_0)-l(v_\alpha)). 
\end{align*}
Then the claim follows from 
\[|\beta(\bw,\bv_0;\alpha_0')|-|\beta(\bw,\bv_\cC;\alpha_0')|=(l(w_0)-l(v_{\alpha_0}))+2\sum_{\alpha \in \cC\setminus\{\alpha_0\}} (l(w_0)-l(v_\alpha)),\] where notice that the pair $\{v_\alpha,v_\alpha^{-1}\}$ necessarily appears along the braid word $\beta(\bw,\bv_\cC;E_0)$. 
\end{proof}
Now let us complete the proof of \cref{lem:codim_boundary}. By \cref{lem:dimension_general} (2), it suffices to compute the codimension
\begin{align*}
    \dim\A_{G,\Sigma} -\dim\A_{G,\Sigma}^\bw = \dim\A_{G,\Sigma}[\cC;\bv_0] -\dim\A_{G,\Sigma}^\bw[\cC;\bv_0].
\end{align*}
Then by \cref{lem:dimension_general} (1), it can be computed as
\begin{align*}
    \dim\A_{G,\Sigma} -\dim\A_{G,\Sigma}^\bw 
    &= \dim\A_{G,\Sigma}[\cC;\bv_0] -\dim\A_{G,\Sigma}^\bw[\cC;\bv_0] \\
    &= \dim\A_{G,\Sigma}^\times[\cC;\bv_0] -\dim\A_{G,\Sigma}^\bw[\cC;\bv_0] 
    = \sum_{E \in \bB(\Sigma)} (l(w_0) - l(w_E)),
\end{align*}
as desired.

\bibliographystyle{alpha}

\end{document}